\documentclass[smallextended]{svjour3}   
\usepackage{amsmath,amssymb,mathtools,listings,amsfonts}
\usepackage{graphicx,stmaryrd,rotating}
\journalname{Submitted article} 
\smartqed 

\newtheorem{procedure}{\bf Procedure}[section]

\newcommand*{\chl}{\mathrm{coshlog}^{\star}(\id)}

\newcommand*{\id}{\mathrm{id}}
\newcommand{\overbar}[1]{\mkern 1.5mu\overline{\mkern-1.5mu#1\mkern-1.5mu}\mkern 1.5mu}
\newcommand*{\meas}{\rho}
\newcommand*{\sh}
{{\,\begin{sideways}\begin{sideways}\begin{sideways}
$\tiny{\exists}$\end{sideways}\end{sideways}\end{sideways}}}
\newcommand*{\sJ}{\mathfrak{J}}
\newcommand*{\g}{\mathrm{g}}
\newcommand*{\R}{{\mathbb R}}
\newcommand*{\Ab}{{\mathbb A}}
\newcommand*{\rd}{{\mathrm{d}}}
\newcommand*{\pa}{\partial}
\newcommand*{\la}{\langle}
\newcommand*{\ra}{\rangle}
\newcommand*{\quas}{*}

\DeclareSymbolFont{bbold}{U}{bbold}{m}{n}
\DeclareSymbolFontAlphabet{\mathbbold}{bbold}

\begin{document}

\title{Algebraic Structures and Stochastic Differential Equations 
driven by L\'{e}vy processes}
\titlerunning{SDEs driven by L\'{e}vy processes}    
\author{Charles Curry, Kurusch Ebrahimi--Fard, 
Simon J.A.~Malham and Anke Wiese}

\institute{Charles Curry \at Maxwell Institute for Mathematical Sciences 
and School of Mathematical and Computer Sciences, 
Heriot-Watt University, Edinburgh EH14 4AS, UK.\\ 
\emph{Present address:} of Charles Curry \at Department of Mathematical Sciences, 
NTNU, 7491 Trondheim, Norway.\\  
\and Kurusch Ebrahimi--Fard \at Department of Mathematical Sciences, 
NTNU, 7491 Trondheim, Norway.\\
\and Simon J.A.~Malham \at 
Maxwell Institute for Mathematical Sciences
and School of Mathematical and Computer Sciences,
Heriot--Watt University, Edinburgh EH14 4AS, UK.\\
\and Anke Wiese \at
Maxwell Institute for Mathematical Sciences
and School of Mathematical and Computer Sciences,
Heriot--Watt University, Edinburgh EH14 4AS, UK.}

\date{5th November 2018}

\maketitle


\begin{abstract}
We construct an efficient integrator for stochastic differential systems 
driven by L\'{e}vy processes. An efficient integrator is a strong approximation 
that is more accurate than the corresponding stochastic Taylor approximation, 
to all orders and independent of the governing vector fields. This holds provided the 
driving processes possess moments of all orders and the vector fields
are sufficiently smooth. Moreover the efficient integrator in question
is optimal within a broad class of perturbations for half-integer
global root mean-square orders of convergence. We obtain these
results using the quasi-shuffle algebra of multiple iterated integrals 
of independent L\'{e}vy processes. 
\end{abstract}

\section{Introduction}
We consider the simulation of It\^o stochastic differential systems 
driven by independent L\'{e}vy processes possessing moments 
of all orders. Our goal is to derive a new strong numerical 
integration scheme for such systems that is efficient in the following sense.
The efficient scheme has a strong error at leading order that is always smaller 
than the strong error of the corresponding stochastic Taylor approximation. 
This is true for any such stochastic differential systems of any size, 
and for all orders of approximation. Moreover, for half-integer orders, our 
efficient integrator is optimal in the sense that 
its global root mean-square strong error realizes its smallest possible value
in a broad class of perturbations  
compared to the error of the corresponding stochastic Taylor approximation.
L\'{e}vy processes are a class of stochastic processes which are
continuous in probability and possess stationary increments, independent 
of the past. They are examples of stochastic processes with a well 
understood structure that incorporate jump discontinuities.
Due to the L\'{e}vy--It\^o decomposition of a L\'{e}vy process the systems 
we consider can be written in the form
\begin{equation*}
y_t=y_0+\sum_{i=0}^d\int_0^t V_i(y_{s})\mathrm{d}W^i_s 
+\sum_{i=d+1}^{\ell}\int_0^t V_i(y_{s_-})\mathrm{d}J^{i}_s,
\end{equation*}
for $y_t\in\R^N$ with $N\in\mathbb N$; see for example Applebaum~\cite{App}.
We assume all the governing autonomous vector fields 
$V_i\colon\mathbb{R}^N\to\mathbb{R}^N$ are smooth and in general non-commuting.
Note that by convention we suppose $W^0_t\equiv t$.
The system is thus driven by the independent Wiener processes
$W^i$ for $i=1,\ldots,d$, and the purely discontinuous martingales $J^{i}$
for $i=d+1,\ldots,\ell$ which are expressible in the form
\begin{equation*}
J^i_t=\int_0^t\int_{\mathbb{R}}v\overbar{Q}^i(\mathrm{d}v,\mathrm{d}s),
\end{equation*}
where 
$\overbar{Q}^i(\mathrm{d}v,\mathrm{d}s)\coloneqq Q^i(\mathrm{d}v,
\mathrm{d}s)-\meas^i(\mathrm{d}v)\mathrm{d}s$.
The $Q^i$ are Poisson measures on $\mathbb{R}\times\mathbb{R}_+$ 
with intensity measures $\meas^i(\mathrm{d}v) \mathrm{d}s$, 
where the $\meas^i$ are measures on $\mathbb{R}$ with $\meas^i(\{0\})=0$ 
and $\int_{\mathbb{R}} (1\wedge v^2) \meas^i(\mathrm{d}v)<\infty$. 
Intuitively, $Q^i(B,(a,b])$ counts the number of jumps of the 
process $J^i$ taking values in the set $B$ in the time interval $(a,b]$, 
whilst $\meas^i(B)$ measures the expected number of jumps per unit time 
the process $J^i$ accrues taking values in the set $B$. 
We emphasize that we assume that the $J^i$ possess moments of all orders. 

Let us outline our approach to the strong numerical simulation 
of L\'evy-driven stochastic differential systems via steps (1)---(4) below. 
As we do so we will present the new results we prove in this paper and put them 
into context. We will also describe the algebraic approach we employ 
to prove the results and its usefulness.

\emph{Step (1).} We begin by studying strong integration schemes for stochastic differential systems
such as those above. Such schemes are typically derived from the stochastic Taylor expansion 
for the solution given by  
\begin{equation*}
y_t=\sum_{w\in A^*} I_w(t)\bigl[\tilde{V}_w(y_0)].
\end{equation*}
A derivation can be found in Platen \& Bruti-Liberati \cite{PlaBl}; also see
Platen \cite{Pla80,Pla82} and Platen \& Wagner \cite{PlaWag}. It is 
produced via Picard iteration using the chain rule for the solution process $y_t$.
We show this explicitly at the beginning of \textsection~\ref{sec:sepST}.
In the expansion the set $A^*$ consists of all the multi-indices (or words) that
can be constructed from the set of letters $A\coloneqq\{0,1,\ldots,d,d+1,\ldots,\ell\}$.
The terms $I_w(t)\bigl[\tilde{V}_w(y_0)]$ represent iterated integrals 
of the integrands $\tilde{V}_w(y_0)$ which involve compositions of the vector fields 
(as first order partial differential and difference operators) 
governing the L\'evy-driven stochastic differential system.
See the beginning of \textsection~\ref{sec:sepST} and in particular
Theorem~\ref{th:ST} for a more precise prescription.
For sufficiently smooth vector fields, integration schemes 
of arbitrary strong order of convergence may be constructed by 
truncating the series expansion to an appropriate number of terms. 
Platen \& Bruti-Liberati~\cite{PlaBl} show which terms must be retained to 
obtain a strong integration scheme with a given strong order of convergence. 
We assume hereafter the governing vector fields are sufficiently smooth 
for the stochastic Taylor expansion to exist. Our first new result in
this paper is Theorem~\ref{mark taylor} in \textsection~\ref{sec:sepST} in which
we show that the terms in the stochastic Taylor expansion above can be written 
in separated form as $I_w(t) \tilde{V}_w(y_0)$. 
We observe that the time-dependent stochastic information
which is encoded in the now integrand-free multiple integrals $I_w$
are separated as scalar multiplicative factors of the structure 
information which is encoded in the terms $\tilde{V}_w(y_0)$ which
involve compositions of the vector fields as first order partial
differential operators only. Indeed the separated expansion is
essentially achieved by Taylor expanding all the shift operators 
associated with the purely discontinuous terms in the original
stochastic Taylor expansion. Note that the result of this is
that the sum is over an expanded set of multi-indices/words 
$\mathbb A^*$ constructed from the alphabet $\mathbb A$ which 
contains $A$. The additional letters in $\mathbb A$ are 
compensated power brackets associated with the purely 
discontinuous terms. See \textsection~\ref{sec:sepST} for more precise
details as well as Curry, Ebrahimi--Fard, Malham \& Wiese~\cite{CEfMW}.
The advantages of this separated form are: (i) the simulation of
the stochastic components $I_w$ of the system, for which targeted generic
simulation algorithms can be developed, is now separated from
the inherent structure components; (ii) it provides the 
context/basis to develop the more general strong integrators
we are about to discuss next and (iii) we can take advantage
of well developed algebraic structures which underlie such expansions
as we shall see.

As is well known the order of a strong numerical scheme is determined by the
set of multiple integrals $I_w$ retained in any approximation. 
We also emphasize here that in any Monte--Carlo simulation, the bulk of the 
computational effort goes into simulating the higher order
multiple integrals included, in particular those indexed by
words of length two or more involving distinct non-zero 
letters---though see Malham \& Wiese~\cite{MW2014} for new efficient
methods for simulating the L\'evy area.
Once this set of multiple integrals is fixed (with its associated
computational burden) the order of convergence is fixed, i.e.\/ 
the power of the stepsize in the leading order term of the
strong error. More accurate schemes correspond to those
that generate smaller coefficients in the leading order term.

\emph{Step (2).} We then discuss how strong integrators based on truncations of the 
stochastic Taylor expansion are just one example of 
a more general class of integrators we shall call map-truncate-invert schemes. 
General map-truncate-invert schemes were first introduced in Malham \& Wiese \cite{MalWie}
for Stratonovich drift-diffusion equations and are constructed as follows.
We start with the flowmap $\varphi_t\colon y_0\mapsto y_t$ lying in $\mathrm{Diff}(\mathbb{R}^N)$
which maps the initial data $y_0$ to the solution $y_t$ at time $t>0$. Here $\mathrm{Diff}(\mathbb{R}^N)$
denotes the space of diffeomorphisms on $\mathbb{R}^N$. From our discussion above,
the separated stochastic Taylor expansion for the flowmap has the form 
\begin{equation*}
\varphi_t=\sum_{w\in\mathbb A^*} I_w(t) \tilde{V}_w.
\end{equation*}
This separated form is key and crucial to all our subsequent developments herein. 
Now, given any invertible map 
$f\colon\mathrm{Diff}(\mathbb{R}^N)\rightarrow\mathrm{Diff}(\mathbb{R}^N)$, 
the corresponding map-truncate-invert scheme is constructed thus: 
we expand the series $f(\varphi_t)$, truncate according to a chosen grading 
and then apply $f^{-1}$. See \textsection~\ref{sec:convalg} for more details, 
in particular how to construct such integrators in the L\'evy-driven context. 
The case $f=\log$ corresponds to the exponential Lie series integrator of 
Castell--Gaines for Stratonovich drift-diffusion equations; see 
Castell \& Gaines \cite{CasGai95,CasGai96}.
In practice the $f=\log$ case can be implemented as follows. 
The series $\psi_t=\log\varphi_t$ lies in the Lie algebra generated
by the underlying vector fields with the scalar coefficients given by
combinations of multiple Stratonovich integrals. Any truncation $\hat\psi_t$
with the multiple integrals replaced by suitable approximate samples is
also a vector field. The corresponding solution approximation is 
$\hat y_t=\exp\hat\psi_t\circ y_0$. This can be generated by solving the 
ordinary differential system $u'=\hat\psi_t\circ u$ for $u=u(\tau)$ 
on $\tau\in[0,1]$ with $u(0)=y_0$. Typically this is achieved using 
a suitable numerical scheme. The approximation $u(1)$ represents $\hat y_t$.
Generally, apart from the low order cases stipulated below, 
the Castell--Gaines method is not asymptotically efficient. 
This was shown by Malham \& Wiese \cite{MalWie} who consequently 
considered the problem of designing so-called efficient integrators for 
Stratonovich drift-diffusion equations driven by multiple independent Wiener processes. 
Here, efficient integrators of a given order have leading order error 
which is always less than that for the corresponding stochastic Taylor scheme, 
independent of the governing vector fields. Malham \& Wiese \cite{MalWie} then 
proved that an integration scheme based on the map $f=\mathrm{sinhlog}$ 
is efficient. These results were obtained in the absence of a drift term 
and extended to incorporate drift in 
Ebrahimi--Fard, Lundervold, Malham, Munthe--Kaas \& Wiese \cite{EfLMMkW}.
It is the generalization of this latter result to L\'evy-driven systems that
we consider presently.

\emph{Step (3).} We temporarily revert our discussion to stochastic Taylor approximations. 
We prove a new result for such schemes which is an important step towards our main result. 
This concerns the relative accuracy of strong numerical methods for L\'evy-driven  
stochastic differential systems generated by truncating the stochastic Taylor
expansion according to word length or mean-square grading. The length of any word
$w\in\mathbb A$ is the number of letters it contains while its mean-square grading
is similar but zero letters are counted twice---consistent with the $L^2$-norm
of $I_w$. Suppose $R_t^{\mathrm{wl}}(y_0)$ and $R_t^{\mathrm{ms}}(y_0)$ 
are the remainders generated by truncating the separated stochastic Taylor expansion 
for $y_t$ by keeping all words of length and mean-square grade 
less than or equal to $n$, respectively. Then in \textsection~\ref{sec:wlvsms} 
(see Theorem~\ref{th:msvswo}) we prove that
at leading order for all $n$ we have
\begin{equation*}
\bigl\|R_t^{\mathrm{wl}}(y_0)\bigr\|_{L^2}^2\leqslant\bigl\|R_t^{\mathrm{ms}}(y_0)\bigr\|_{L^2}^2.
\end{equation*}   
There is computational effort involved in simulating any such word 
length stochastic Taylor approximation due to the 
extra iterated integrals included. However those that are included 
are indexed by words which include at least
one drift (zero) letter, i.e.~integrating with respect to ${\mathrm d}t$, or 
repeated L\'evy (non-zero) letters 
which are computationally less burdensome to compute 
than those indexed by words containing completely distinct non-zero letters.
We assess the trade off of accuracy versus computational effort once we introduce
our main result next.

\emph{Step (4).} Lastly, we present our new antisymmetric sign reverse integrator
and prove that it is an efficient integrator for L\'evy-driven systems; 
see Theorem~\ref{main result}(a).
In addition, we prove that it is optimal at half-integer global 
root mean-square orders of convergence; see Theorem~\ref{main result}(b). 
It is an example of a map-truncate-invert scheme and the principle ideas 
underlying its construction and properties are as follows. 
To begin with we consider the result of applying the first two
steps of the map-truncate-invert scheme as described above, i.e.\/
considering a function $f(\varphi_t)$ of the flowmap and then truncating
the result according to a given grading up to and including terms of
grade $n$. The pre-remainder is the remainder after truncating $f(\varphi_t)$ 
in this way. We then show that for word length grading, at leading order, 
the pre-remainder is equal to the true remainder, 
i.e.\/ equal to the remainder after applying $f^{-1}$ to the truncated $f(\varphi_t)$. 
Herein lies the room-for-manouver of which we take advantage. Naturally
the map-truncate-invert scheme must emulate the scheme based on the corresponding 
stochastic Taylor expansion truncated according to the same grading. 
The question is thus, is the pre-remainder (and thus remainder) associated with the 
map-truncate-invert scheme at hand less than the remainder associated with
the corresponding stochastic Taylor approximation to leading order in the 
mean-square measure? The answer is ``yes'' for the 
antisymmetric sign reverse integrator. This integrator is constructed 
by taking half the difference of the stochastic Taylor expansion and
the stochastic Taylor expansion with all the words indexing the multiple 
integrals reversed in order and signed according to the length of the words.
Hence the pre-remainder and thus remainder are 
\begin{equation*}
R_t^{\mathrm{ASRI}}(y_0)=\sum_{|w|=n+1}\tfrac12\bigl(I_w-I_{S(w)}\bigr)V_w(y_0).
\end{equation*}
Here $|w|$ represents the length of the word $w$ and 
$S$ is the signed reversal map defined as 
$S\colon a_1a_2\ldots a_{n}\mapsto (-1)^n\,a_{n}a_{n-1}\ldots a_1$
for any word consisting of letters $a_1$, $a_2$, \ldots, $a_{n}$ from $\mathbb A$.
We naturally assume $I_{-w}\equiv -I_w$. Our \emph{principal main result} is thus 
\begin{equation*}
\bigl\|R_t^{\mathrm{ASRI}}(y_0)\bigr\|_{L^2}^2\leqslant\bigl\|R_t^{\mathrm{wl}}(y_0)\bigr\|_{L^2}^2,
\end{equation*}
for all $y_0$, all L\'evy-driven systems, to all orders $n$. Using our
result above this means the antisymmetric sign reverse integrator is more
accurate than the corresponding stochastic Taylor approximation truncated 
according to mean-square grading. Furthermore we show that when $n$ is odd,
the inequality just above is optimal in the following sense. Suppose we 
perturb the antisymmetric sign reverse integrator by a general class of 
perturbations---see \textsection~\ref{sec:ASRI} for more details. 
Then when $n$ is odd we show that the difference between the 
two mean-square error measures for $R_t^{\mathrm{wl}}(y_0)$ and $R_t^{\mathrm{ASRI}}(y_0)$ 
above are minimized when the perturbation is zero. One further technical
practical aspect comes into play. In the Castell--Gaines case we 
compute $f^{-1}=\exp$ by solving an ordinary differential system.
If all the underlying vector fields are linear $f^{-1}$ may be computed 
as a matrix-valued compositional inverse function; see Malham \& Wiese \cite{MalWie}.
However we cannot do this in general for map-truncate-invert schemes.
To circumvent this issue we have developed direct-map-truncate-invert schemes
with a special case being the direct antisymmetric sign reverse integrator.
In direct-map-truncate-invert schemes we utilize the observation above that
map-truncate-invert schemes naturally emulate the scheme based on the corresponding 
stochastic Taylor approximation to the same order, but have a different remainder 
at leading order as different terms at leading order in the remainder are 
included in the integrator.
Thus in order to simulate a direct-map-truncate-invert scheme we simulate the
corresponding stochastic Taylor approximation to the same order as well as the
terms that lie in the difference between the stochastic Taylor approximation remainder
and map-truncate-invert approximation remainder at leading order. These are 
additional terms indexed by words of length $n+1$, the terms concerned will not involve
iterated integrals involving completely distinct letters (those would belong to 
the next order integrator). Our discussion above in the last paragraph on the 
computational burden involved thus applies here. Indeed we round off our results
herein with global convergence results in \textsection~\ref{sec:glob} and 
numerical simulations demonstrating our conclusions in \textsection~\ref{sec:NS}.

We remark that several challenges are encountered upon generalising from drift-diffusions 
to L\'evy-driven equations. The stochastic Taylor expansions derived 
by Platen \& Bruti--Liberati~\cite{PlaBl} do not display the separation 
of geometric and stochastic required for the deployment of map-truncate-invert schemes. 
Furthermore, due to the discontinuities of the driving paths, 
it is not possible to use Stratonovich integrals to obtain the usual 
product rule for iterated integrals. Algebraically, 
this means that we must use the quasi-shuffle algebra of 
iterated integrals with respect to L\'evy processes 
(see Curry, Ebrahimi--Fard, Malham \& Wiese~\cite{CEfMW}) 
rather than the shuffle algebra used in the derivation 
of map-truncate-invert schemes for drift-diffusions.
Further, it is rather intriguing that the efficient integrator 
in the L\'evy-driven (quasi-shuffle) context herein is
the antisymmetric sign reverse integrator, i.e.\/ half the 
difference of the identity and sign-reverse endomorphisms. 
This coincides with efficient integrator in the 
drift-diffusion (shuffle) case. In the shuffle case the 
sign-reverse endomorphism coincides with the antipode. 
However in the general quasi-shuffle case this is no longer true. 

We prove our key results above in the context of combinatorial algebras,
in particular the quasi-shuffle Hopf algebra and the convolutional
algebra of endomorphisms defined on the quasi-shuffle Hopf algebra. 
It is this abstraction that not only allows us to prove our main results 
but also explicitly identify the `Efficient Integrator'.   
These algebraic structures underlie the separated stochastic Taylor
expansion for the flowmap associated with L\'evy-driven stochastic differential
systems and all the manipulations we apply to the separated stochastic Taylor expansion
when we construct map-truncate-invert schemes and assess their properties. 
The separated stochastic Taylor expansion is the starting point for our analysis. Key for its derivation
is the repeated application of It{\^o}'s formula. The stochastic information of our L\'{e}vy-driven
system is thus given by the Wiener processes and by the purely discontinuous martingales in the 
L{\'e}vy--It\^o decomposition of the  driving L{\'e}vy processes 
as well as by their quadratic variation processes 
and power brackets. We encode this stochastic information 
in our augmented alphabet  $\mathbb A$: 
each letter in $\mathbb A$ is associated with a Wiener processes 
or a purely discontinous processes driving 
the stochastic differential equation or with a compensated power bracket generated from these. 
For technical reasons it is opportune to use compensated power brackets rather than power brackets. 
Once we have constructed the alphabet, we can associate words constructed 
from this alphabet with the corresponding iterated integral. For example,
for two Wiener processes $W^1$ and $W^2$, the word $\omega=12$ is associated with the iterated integral 
$I_{12}=\int_{0\leqslant\tau_1\leqslant\tau_2\leqslant t}\mathrm{d}W^1_{\tau_1}\mathrm{d}W^2_{\tau_2}$.
For a Wiener process $W^1$ and purely discontinuous martingale $J^i$ 
with $i\in\{d+1,\ldots,\ell\}$, the word $\omega=1i$ is associated with the iterated integral 
$I_{1i}=\int_{0\leqslant\tau_1\leqslant\tau_2\leqslant t}\mathrm{d}W^1_{\tau_1}\mathrm{d}J^i_{\tau_2}$.
See Curry \textit{et al.\/} \cite{CEfMW} as well as 
\textsection~\ref{sec:sepST} for more details on the necessity 
and construction of the augmented alphabet $\mathbb A$. 
Associated with each member of the alphabet $\mathbb A$ 
constructed from the stochastic information is a corresponding 
partial differential operator. With each word we associate the partial differential operator 
obtained by compositions of the differential operators associated 
with the letters constituting the word (see \textsection~\ref{sec:sepST} for more details). 
Our first main result is that the separated stochastic Taylor expansion for the flowmap above
is the sum $\varphi_t=\sum_{w\in\mathbb A^*} I_w(t) \tilde{V}_w$ 
over the real multiplication of the multiple integrals $I_w$ and differential operators $\tilde{V}_w$
which are indexed by all words/multi-indices that can be constructed from
the augmented alphabet $\mathbb A$ containing the original alphabet $A$. 
Now consider functions 
$f\colon\mathrm{Diff}(\mathbb{R}^N)\rightarrow\mathrm{Diff}(\mathbb{R}^N)$ of the flowmap which have
simple power series expansions such as $f(\varphi)=\sum_{k\geqslant0}c_k\varphi^k$,
for some real coefficients $c_k$. Applying such a function to the flowmap
generates linear combinations of products of the multiple integrals 
and linear combinations of compositions of the differential operators, 
i.e.\/ operations in the respective algebras generated by the multiple integrals and 
by the differential operators. 
The algebra generated by the multiple integrals can be identified with the quasi-shuffle
algebra of the words constructed from $\mathbb{A}^*$, 
again see Curry \textit{et al.\/} \cite{CEfMW} as well as 
\textsection~\ref{sec:sepST}. 
The algebra generated by the composition of differential operators can be 
identified with a concatenation algebra of the words. Specifically, if we strip
away the $I$'s and $\tilde{V}$'s we can represent the stochastic Taylor
series for the flowmap by 
\begin{equation*}
\sum_{w\in\mathbb{A}^*} w\otimes w.
\end{equation*}
This lies in a product algebra in which the bilinear product of two terms is
given by
\begin{equation*}
(u\otimes x)(v\otimes y)=(u\ast v)\otimes (xy),
\end{equation*}
for any words $u,v,x,y$ constructed from $\mathbb{A}^*$.
On the left $u\ast v$ is the quasi-shuffle product of 
$u$ and $v$ representing the real product $I_uI_v$.
On the right $xy$ is the concatentation of $x$ and $y$ 
representing the composition of the differential operators $\tilde{V}_x$
and $\tilde{V}_y$. Substituting the abstract series representation for
the flowmap above into $f(\varphi)$, where $f$ has the power series
expansion with real coefficients $c_k$ indicated above,
by direct computation and rearrangement
we have (see for example Ebrahimi--Fard, \textit{et al.\/} \cite{EfLMMkW}),
\begin{equation*}
f(\varphi)=\sum_{w\in\mathbb{A}^*} F(w)\otimes w,
\end{equation*}
where
\begin{equation*}
F(w)=\sum_{k=0}^{|w|}c_k\sum_{\substack{u_1,\ldots,u_k\in\Ab^*\\u_1\ldots u_k=w}}
u_1\ast\cdots\ast u_k.
\end{equation*} 
Here $F$ is a linear endomorphism on words. 
This last calculation shows us that we can encode functions of the flowmap 
$f(\varphi)$ by endomorphisms $F$ on the quasi-shuffle algebra of words. 
A fundamental building block of the endomorphism $F$ is the action 
\begin{equation*}
w \mapsto \sum_{\substack{u_1,\ldots,u_k\in\Ab^*\\u_1\ldots u_k=w}}
u_1\ast\cdots\ast u_k,
\end{equation*}
i.e. the sum of the quasi-shuffle product $u_1\ast\cdots\ast u_k$ of any 
possible separation of $w$ into any subwords $u_1,\, \ldots,\, u_k$. 
This object is the $k$-th convolution power of the identity endomorphism $\id^{\star k}$ 
acting on $w$. We can
thus abstract the context one level further, and this abstraction is natural and key.
The function $f$ of the flowmap  $f(\varphi)$ can be represented by 
$F(\id)=\sum_{k\geqslant0}c_k\id^{\star k}$, 
where the action of 
$\id^{\star k}$ is $\id^{\star k}(w)=\sum_{u_1\ldots u_k=w}u_1\ast\cdots\ast u_k$. 
Once the stochastic system under consideration 
is fixed, i.e. we have prescribed the driving L\'evy processes and governing vector fields, 
the words are fixed. Different integrators are distinguished by their actions on
those words and are represented by endomorphisms acting on them. 
The stochastic Taylor expansion itself corresponds to the case 
where $F$ is the identity `$\id$' endomorphism. 

We thus utilize two consecutive levels of abstraction herein. The first
abstraction is to the product algebra with quasi-shuffles on one side and 
concatenations on the other. The second abstraction is to the algebra of endomorphisms 
that act on the Hopf quasi-shuffle algebra.
On this algebra we can define an inner product representing the $L^2$ inner product
on stochastic processes. Our abstraction encapsules the information required 
for our problem under consideration, and we are naturally able to 
represent map-truncate-invert approximations succinctly and to analyse their
properties independent of the underlying system, data and truncation order.
This enables us to establish the general results we present.

L\'{e}vy processes have many applications. Most notably though they appear 
in mathematical finance in the construction of models going beyond the 
Black--Scholes--Merton model to incorporate discontinuities in stock prices. 
See for example Cont \& Tankov~\cite{ConTan} and Barndorff--Nielsen, Mikosch 
\& Resnick \cite{BnMR} and the references therein. There are also applications 
to physical sciences, see for instance 
Barndorff--Nielsen, Mikosch \& Resnick~\cite{BnMR}. 
Work on Euler type methods for simulating L\'evy-driven 
stochastic differential systems including
how to incorporate the jump processes can for example be found in Jacod~\cite{Jacod},
Fournier~\cite{Fournier}. 
See Higham \& Kloeden~\cite{HK2005,HK2006} for implicit methods and
Protter \& Talay~\cite{PT} for weak approximations.
The monograph by Platen \& Bruti-Liberati \cite{PlaBl} 
provides a comprehensive introduction 
to numerical simulation schemes for stochastic differential systems 
driven by L\'{e}vy processes and includes financial applications.
Barski~\cite{Barski} also develops general high order schemes for L\'evy-driven
systems. Pathwise integrals of L\'{e}vy processes have been constructed 
in the framework of rough paths by Friz \& Shekhar \cite{FriShe}. 

The problem of minimizing the coefficient in the leading error term in 
stochastic Taylor integrators
was considered by Clark \cite{Cla} and Newton \cite{New86,New91} 
for drift-diffusion equations driven by a single Wiener process; also see
Kloeden \& Platen~\cite[Section~13.4]{KloPla}.
They derived integration schemes that are asymptotically efficient 
in the sense that the coefficient of the leading order error is minimal 
among all schemes. For Stratonovich Wiener-driven stochastic differential equations, 
Castell \& Gaines \cite{CasGai95,CasGai96} constructed integration schemes 
using the Chen-Strichartz exponential Lie series expansion (see Strichartz \cite{Str}).
In the strong order one case for one driving Wiener process, and the strong order
one-half case for two or more driving Wiener processes, their schemes 
are asymptotically efficient in the sense of Clark and Newton.

The quasi-shuffle algebra of iterated stochastic integrals was considered
by Gaines \cite{Gai} for the case of 
multiple iterated integrals of Wiener processes back in 1994. A few years later 
Li \& Liu \cite{LiLiu} 
considered multiple iterated integrals of Wiener processes and standard 
Poisson processes. More recently Curry \textit{et al.\/} \cite{CEfMW} proved 
that the algebra of multiple integrals of a minimal family of semimartingales 
is isomorphic to the combinatorial quasi-shuffle algebra of words. A set of L\'evy
processes represents one example. The quasi-shuffle algebra is an extension of the 
shuffle algebra. Its historical development is of noteworthy interest. 
Indeed, it was introduced abstractly in a 1979 paper by Newman and Radford \cite{NewRad}, 
where the authors attempt to endow the free coalgebra 
over an associative algebra with a Hopf algebra structure. 
A quarter of a century later, and independently from 
Gaines, in a sequence of papers, Hudson, Parthasarathy and collaborators presented 
a combinatorial product of iterated quantum stochastic 
integrals which is equivalent to the Gaines quasi-shuffle product;
see Hudson \& Parthasarathy~\cite{HP1,HP2,HP3} and 
Cohen, Eyre \& Hudson~\cite{Beasley}.
This product was coined the sticky-shuffle, see Hudson~\cite{Hudson_2}, 
and was studied from a Hopf algebra viewpoint in Hudson~\cite{Hudson_1}. 
Independently of these developments, 
Hoffman~\cite{Hof} comprehensively studied the quasi-shuffle product 
using a Hopf algebraic framework. 
The significance of the shuffle algebra was cemented 
in the work of Eilenberg \& Maclane~\cite{EiMac}, Sch\"{u}tzenberger \cite{Sch} 
and Chen \cite{Che}. See Reutenauer \cite{Reu} for more details. 
One advantage of abstracting to the quasi-shuffle algebra is that we can immediately 
identify the minimum set of iterated integrals that need to be simulated to implement 
a given accurate strong scheme. This optimizes the total computation time which is dominated
by the strong simulation of the iterated integrals.
Indeed, Radford~\cite{Rad} proved that the shuffle algebra is generated by Lyndon words. 
This was extended to the quasi-shuffle algebra by Newman \& Radford~\cite{NewRad} and
independently later by Hoffman~\cite{Hof}. 
Gaines \cite{Gai} established this independently for the case of
Wiener processes, as did Li \& Liu \cite{LiLiu} for the case of
Wiener processes and standard Poisson processes, while Sussmann~\cite{Suss} had 
considered a Hall basis. Hence the set of iterated integrals we need to
simulate at any given order are identified by Lyndon words. 

The study and application of such structures 
in systems, control and stochastic processes
can be traced back to the work of Chen~\cite{Che},
Magnus~\cite{Magnus}, Kunita~\cite{Kunita}, Fliess~\cite{Fliess}, 
Azencott~\cite{Azencott}, Sussmann~\cite{Suss}, 
Strichartz~\cite{Str}, Ben Arous~\cite{Ben Arous},
Grossman \& Larson~\cite{GL:realzn}, 
Reutenauer~\cite{Reu},
Castell~\cite{Castell}, Gaines~\cite{Gai}, 
Lyons~\cite{Lyons}, Li \& Liu~\cite{LiLiu}, Burrage \& Burrage~\cite{BB},
Kawski~\cite{Kawski}, Baudoin~\cite{Baudoin} and Lyons \& Victoir~\cite{LV}.
Many of these authors focus on the exponential Lie series and its applications.
A particular impetus of the study of such stuctures arose in the late nineties
when Brouder~\cite{Brouder} demonstrated the connection between 
the Hopf algebra used by Connes \& Kreimer~\cite{CK} for the renormalization problem in
perturbative quantum field theory and the Butcher group used to study Runge--Kutta methods 
by Butcher~\cite{Butcher} in the late sixties and early seventies; 
see Hairer, Lubich \& Wanner~\cite{HLW}. 
Hopf algebraic structures are now a natural lexicon in the study
of (to name a few): (i) numerical methods for deterministic systems, 
such as Runge--Kutta methods and geometric and 
symplectic integrators---see for example Munthe--Kaas \& Wright~\cite{M-KW},
Lundervold \& Munthe--Kaas~\cite{LM-K} and McLachlan, Modin, Munthe--Kaas
\& Verdier~\cite{MMMV}; 
(ii) approximations for stochastic differential equations---see
for example Castell \& Gaines~\cite{CasGai95,CasGai96}, Malham \& Wiese~\cite{MalWie},
Ebrahimi--Fard, \textit{et al.\/} \cite{EfLMMkW}, 
Ebrahimi--Fard, Malham, Patras \& Wiese~\cite{EfMPW1,EfMPW2}; 
(iii) rough paths---see for example Hairer \& Kelly~\cite{HK} and Gubinelli \& Tindel~\cite{GT}
and (iv) regularity structures for stochastic partial differential equations---see
for example Gubinelli~\cite{G:controlledRP} and Hairer~\cite{H:regstructures}.

In summary in this paper, what is new, representing our \emph{main results}, is that we:
\begin{enumerate}
\item Show the stochastic Taylor series expansion for the flowmap
can be written in separated form, i.e.\/ as a series of terms,
each of which is decomposable into a multiple It\^o integral 
and a composition of associated differential operators (see \textsection~\ref{sec:sepST};
in particular Theorem~\ref{mark taylor}); 
\item Describe the class of map-truncate-invert schemes for 
L\'{e}vy-driven equations and give an algebraic framework for 
encoding and comparing such schemes. These results 
generalize those in Malham \& Wiese \cite{MalWie} 
and Ebrahimi--Fard \textit{et al.\/} \cite{EfLMMkW} 
(see \textsection~\ref{sec:convalg}, Procedure~\ref{proc});
\item Prove truncations according to the word length grading give more 
accurate schemes than approximations of the same order of convergence obtained by 
truncations according to the mean-square grading (see \textsection~\ref{sec:wlvsms};
in particular Theorem~\ref{th:msvswo});
\item Show how to compute map-truncate-invert schemes in practice and 
how to deal with the inversion stage. We call these direct map-truncate-invert schemes
(see~\textsection~\ref{sec:ASRI}, Corollary~\ref{inv corollary});
\item Introduce the antisymmetric sign reverse integrator, a new integration scheme for 
L\'{e}vy-driven equations, represented as half the difference of the 
identity and sign reverse endomorphisms on the vector space generated by words indexing
multiple integrals. This scheme is efficient in the sense that its
leading order mean-square error is less than that of the corresponding
stochastic Taylor scheme, independent of the governing vector fields 
(see \textsection~\ref{sec:ASRI}; in particular Theorem~\ref{main result}(a));
\item Prove the antisymmetric sign reverse integrator is optimal in the sense outlined 
above at half-integer global root mean-square orders of convergence 
(see Theorem~\ref{main result}(b)).
\end{enumerate}
We round off this paper by establishing global convergence results 
from local error estimates in \textsection~\ref{sec:glob} and providing
some explicit antisymmetric sign reverse integrators 
as well as numerical experiments demonstrating our results in \textsection~\ref{sec:NS}.
Further numerical results and details can be found in the electronic
supplementary material, including an introduction to the role
of quasi-shuffle algebras in the theory of stochastic differential equations.

\section{Separated stochastic Taylor expansions}\label{sec:sepST}
Our setting is a complete, filtered probability space 
$\bigl(\Omega,\mathcal{F},(\mathcal{F}_t)_{t\geq0},P\bigr)$, 
assumed to satisfy the usual hypotheses---see Protter \cite[p.3]{Pro}. 
For the It\^o stochastic differential system driven by L\'evy processes
presented in the introduction, we assume the initial data 
$y_0\in L^2(\Omega,\mathcal{F}_0,P)$. 

A stochastic Taylor expansion is an expression for the flowmap as a 
sum of iterated integrals. We write down the stochastic Taylor expansion 
for L\'{e}vy-driven equations and show how and when it can be written in separated form. 
Recall the flowmap is defined as the map $\varphi_{s,t}\colon y_s\mapsto y_t$. 
It acts on sufficiently smooth functions $f\colon\mathbb R^N\to\mathbb R^N$ 
as the pullback $\varphi_{s,t}(f(y_s)) \coloneqq f(y_t)$. 
We set $\varphi_t\coloneqq\varphi_{0,t}$. It\^{o}'s formula 
(see for example Applebaum~\cite[p.~203]{App} or Protter~\cite[p.~71]{Pro}) implies
\begin{equation*}
f(y_t) = f(y_0) + \sum_{i=0}^d \int_0^t \tilde{V}_i\circ f(y_s) 
\,\mathrm{d}W^i_s + \sum_{i=d+1}^\ell \int_0^t \int_{\mathbb{R}} 
(\tilde{V}_{i}\circ f)(y_{s_-},v) \,\overbar{Q}^i(\mathrm{d}v,\mathrm{d}s),
\end{equation*}
where for $i=1,\ldots,\ell$, the $\tilde{V}_i$ are operators defined as follows
\begin{equation*}
\tilde{V}_i\circ f\coloneqq
\begin{dcases}
(V_i \cdot \nabla) f, & \mathrm{if}\quad i=1,\ldots, d, \\
f\bigl(\,\cdot\, + v V_{i}(\,\cdot\,)\bigr) - f(\,\cdot\,),  & \mathrm{if} \quad i = d+1,\ldots,\ell,
\end{dcases}
\end{equation*}
and $\tilde{V}_0$ is the operator defined by 
\begin{equation*}
\tilde{V}_0\circ f\coloneqq (V_0 \cdot \nabla)f 
+ \frac{1}{2} \sum_{i=1}^d \sum_{j,k=1}^N V_i^j V_i^k \partial_{x_j}\partial_{x_k} f 
+ \sum_{i=d+1}^\ell \int_{\mathbb{R}} \big[ (\tilde{V}_i\circ f)(\cdot,v) 
- v(V_i \cdot \nabla)f \big] \,\meas^i(\mathrm{d}v).
\end{equation*}
Note for $i=d+1,\ldots,\ell$, the $\tilde{V}_i$ introduce an 
additional dependence on a real parameter $v$. The stochastic Taylor expansion 
is derived by expanding the integrands in the L\'{e}vy-driven equation using It\^{o}'s formula. 
This procedure is repeated iteratively, where the iterations are encoded as follows. 
Let $A$ be the alphabet $A \coloneqq \{0,1,\ldots,\ell\}$. For any such alphabet,
we use $A^*$ to denote the free monoid over $A$---the set of words $w=a_1\ldots a_m$ 
constructed from letters $a_i\in A$. 
We write $\mathbbold{1}$ for the empty word. For a given word $w=a_1\ldots a_m$ 
define the operator $\tilde{V}_w \coloneqq \tilde{V}_{a_1} \circ \cdots \circ \tilde{V}_{a_m}$. 
Let $\mathfrak s(w)$ be the number of letters of $w$ from the subset $\{d+1,\ldots,\ell\}\subset A$. 
For a given integrand $g(t,v)\colon\mathbb{R}_+\times \mathbb{R}^{\mathfrak s(w)} \rightarrow \mathbb{R}^N$, 
we define the iterated integrals $I_w(t)[g]$ inductively as follows. 
We write $I_{\mathbbold{1}}(t)[g] \coloneqq g(t)$, and
\begin{equation*}
I_w(t)[g] \coloneqq 
\begin{dcases}
\int_0^t I_{a_1\ldots a_{m-1}}(s)[g] \,\mathrm{d}W^{a_m}_s, & \mathrm{if}\quad a_m = 0,1,\ldots,d, \\
\int_0^t \int_{\mathbb{R}} I_{a_1\ldots a_{m-1}}(s_-) [g(\cdot,v)] \,\overbar{Q}^{a_m}(\mathrm{d}v,\mathrm{d}s), 
& \mathrm{if}\quad a_m = d+1,\ldots,\ell.
\end{dcases}
\end{equation*}
Iteratively applying the chain rule generates the following 
(see Platen \& Bruti-Liberati \cite{PlaBl}).
\begin{theorem}[Stochastic Taylor expansion]\label{th:ST}
For a L\'{e}vy-driven equation the action of the flowmap $\varphi_t$
on sufficiently smooth functions $f\colon\mathbb R^N\to\mathbb R^N$ 
can be expanded as follows 
\begin{equation*}
\varphi_t\circ f=\sum_{w\in A^*} I_w(t)\bigl[\tilde{V}_w\circ f\bigr].
\end{equation*}
\end{theorem}
\begin{remark}[Stochastic Taylor expansion convergence]\label{rmk:lineargrowth}
For integration schemes derived from the stochastic Taylor expansion to converge, 
it suffices that the terms $\tilde{V}_w\circ f$ included in the expansion, and
those at leading order in the remainder, satisfy global Lipschitz 
and linear growth conditions (see Platen \& Bruti-Liberati \cite{PlaBl}). 
Hereafter we assume these conditions are satisfied.
\end{remark}
\begin{remark}[Platen and Bruti--Liberati form]
The stochastic Taylor expansion derived in Platen \& Bruti-Liberati \cite{PlaBl} 
is an equivalent though different representation; we show this in the 
electronic supplementary material.
\end{remark}
Given the stochastic Taylor expansion for the flowmap in Theorem~\ref{th:ST},
we now show how we can write it in separated form. One component of the
separated form are iterated integrals which are free in the sense of having 
no integrand. We define these abstractly for an arbitrary given alphabet 
for the moment, the reason for this will be apparent presently.
\begin{definition}[Free multiple iterated integrals]
Given a collection of stochastic processes $\{Z^{a_i}_t\}_{a_i\in\mathbb{A}}$, 
indexed by a given countable alphabet $\mathbb{A}$,
free multiple iterated integrals take the form
\begin{equation*}
I_w(t) \coloneqq \int_{0<\tau_{1}<\cdots<\tau_m<t}
\,\mathrm{d}Z^{a_1}_{\tau_1}\cdots\,\mathrm{d}Z^{a_m}_{\tau_m},
\end{equation*}
where $w=a_1\ldots a_m$ are words in $\mathbb A^*$. 
\end{definition}
We need to augment $\{W^0,W^1,\ldots,W^d,J^{d+1},\ldots,J^\ell\}$ 
with a set of extended driving processes as follows.
A key component in the characterization of the algebra generated 
by the vector space of free iterated integrals of the driving processes
$\{W^0,W^1,\ldots,W^d, J^{d+1},\ldots,J^\ell\}$ are the compensated power brackets defined for each
$i \in \{d+1, \ldots,\, \ell\}$ and for $p\geq 2$ by
\begin{equation*}
J^{i^{(p)}}_t\coloneqq \int_0^t \int_{\mathbb{R}} v^p \overbar{Q}^i(\mathrm{d}v,\mathrm{d}s).
\end{equation*}
Equivalently we have $J^{i^{(p)}}_t=[J^i]^{(p)}-t\int_{\mathbb R} v^p\,\rho^i(\mathrm{d}v)$,
where $[J_t^i]^{(p)}$ is the $p$th order nested quadratic covariation
bracket of $J_t^i$, i.e.~$[J_t^i]^{(2)}:=[J_t^i, J_t^i]$ 
and $[J_t^i]^{(p)}:= [J_t^i, [J_t^i]^{(p-1)}]$ for $p\geqslant 3$. 
Note that for $p\geqslant 3$ the $p$-th power bracket  $[J_t^i]^{(p)}$ 
equals the sum of the $p$-th power of the jumps of $J^i$ to time $t$.
Importantly, the compensated power brackets have the property that if $J^{i^{(p)}}$ is contained in the
linear span of $\{t,J^i_t,J^{i^{(2)}}_t,\ldots,J^{i^{(p-1)}}_t\}$ for some $p\geq 2$,
then $J^{i^{(q)}}$ is also in this linear span for all $q\geq p$.
Hence inductively for $p\geq 2$, we augment our family 
$\{W^0,W^1,\ldots,W^d, J^{d+1},\ldots,J^\ell\}$ to include the 
compensated power brackets $J^{i^{(p)}}$ as long as they are not contained in the
linear span of $\bigl\{t,J^i_t,J^{i^{(2)}}_t,\ldots,J^{i^{(p-1)}}_t\bigr\}$. 
By doing so, we obtain a possibly infinite family of stochastic processes. 
The iterated integrals of this extended family
form the algebra generated by the iterated integrals of our driving processes. 
See Curry \textit{et al.\/} \cite{CEfMW} for further details. In summary we define 
our extended alphabet as follows.
\begin{definition}[Extended alphabet]
We define our alphabet $\mathbb{A}$ to contain the letters $0,1,\ldots,\ell$, 
associated to the driving processes $\{W^0,W^1,\ldots,W^d, J^{d+1},\ldots,J^\ell\}$, 
and the additional letters $i^{(p)}$ corresponding to any $J^{i^{(p)}}$ 
contained in the extended family as described above.
\end{definition}
\begin{definition}[Separated stochastic Taylor expansion]
The flowmap for a L\'{e}vy-driven stochastic differential equation
possesses a separated stochastic Taylor expansion if it can be written in the form
\begin{equation*}
\varphi_t=\sum_{w\in\mathbb A^*} I_w \tilde{V}_w,
\end{equation*}
where $\{I_w\}_{w\in\mathbb A^*}$ are the free 
iterated integrals associated to the extended driving processes 
and $\tilde{V}_w = \tilde{V}_{a_1}\circ\cdots\circ\tilde{V}_{a_m}$ 
are operators indexed by words that compose associatively.
\end{definition}
\begin{remark}[Separated expansion for jump-diffusion equations] 
In the case of jump-diffusion equations, i.e.\/ L\'evy-driven equations for which 
all the discontinuous driving processes $J^i$ are standard Poisson processes, 
the stochastic Taylor expansion is of the separated form. 
Indeed, as standard Poisson processes have jumps of size one only, 
the operators $\tilde{V}_i$ with $i=d+1,\ldots,\ell$ do not introduce 
a dependence on an additional parameter. The integrands are thus constant 
across the range of integration, and hence the expansion is separated.
\end{remark}
We can construct a separated stochastic Taylor expansion for the flowmap as follows.
By Taylor expansion of the term $f\bigl(y+vV_i(y)\bigr)$ 
appearing in the shift $\tilde{V}_i\circ f$, we have
\begin{equation*}
(\tilde{V}_i\circ f)(y,v)=\sum_{m\geq 1}v^m\tilde{V}_{i^{(m)}} \circ f(y),
\end{equation*}
where we write $V_{i}=\bigl(V_{i}^1,\ldots, V_{i}^N\bigr)^{\mathrm{T}}$ 
and $f=\bigl(f^1,\ldots,f^N\bigr)^{\mathrm{T}}$. 
We define the operators $\tilde{V}_{i^{(m)}}$ by
\begin{equation*}
\tilde{V}_{i^{(m)}}\circ f^j 
\coloneqq\sum_{k\geq1}\sum_{i_1+\ldots + i_k = m} \frac{1}{m!} V_{i}^{i_1}\cdots V_{i}^{i_k} 
\frac{\partial^m f^j}{\partial y^{i_1}\cdots \partial y^{i_k}},
\end{equation*}
where the $i_j\in\mathbb N$. 
The product in $V_{i}^{i_1}\cdots V_{i}^{i_k}$ is multiplication in $\mathbb{R}$. 
We then have 
\begin{equation*}
I_i(t)[(\tilde{V}_i\circ f)(\cdot,v)] 
= \sum_{m\geq 1} \int_0^t \bigl(\tilde{V}_{i^{(m)}}\circ f\bigr)\,\mathrm{d}J^{i^{(m)}}_s,
\end{equation*}
for $i=d+1,\ldots,\ell$. Inserting the above into the relation 
\begin{equation*}
I_w(t)\bigl[\tilde{V}_w\circ f\bigr] 
= I_{a_2\ldots a_m}(t)\Big[I_{a_1}(\cdot)\big[\tilde{V}_{a_1}\circ (\tilde{V}_{a_2\ldots a_m}\circ f)\big]\Big]
\end{equation*} 
and iterating gives the separated expansion, 
where the operators $\tilde{V}_a$ are those of 
the stochastic Taylor expansion for $a\in\{0,1,\ldots,d\}$, 
and for $a=i^{(m)}$ are given by the $\tilde{V}_{i^{(m)}}$ defined above.
We have thus just established the following result.
\begin{theorem}[Separated stochastic Taylor expansion: existence]\label{mark taylor}
For a L\'evy-driven equation suppose the terms $\tilde V_w\circ f$ in the stochastic
Taylor expansion for the flowmap are analytic on $\mathbb R^N$. Then it can 
be written in separated form, i.e.\/ as a separated stochastic Taylor expansion.
\end{theorem}
\begin{remark}[Linear vector fields and linear diffeomorphisms] 
In this special case the separated expansion has an especially simple form;
see the electronic supplementary material. Note that the identity map is a 
special case of a linear diffeomorphism.
\end{remark}
Hereafter we assume the existence of a separated Taylor expansion for the flowmap, 
and all iterated integrals are free iterated integrals.

\section{Convolution algebras and map-truncate-invert schemes}\label{sec:convalg}
We now introduce a class of numerical integration schemes we call 
map-truncate-invert schemes and show how they can be encoded algebraically. 
The algebraic structures arise naturally from the products of iterated integrals 
and the composition of operators appearing in the separated stochastic Taylor expansion. 
Consider the class of numerical integration schemes obtained from 
the stochastic Taylor expansion by simulating truncations of the expansion 
on each subinterval of a uniform discretization of a given time domain $[0,T]$. 
\begin{definition}[Grading function and truncations]
A grading function $\mathrm{g}\colon\mathbb{A}^*\rightarrow\mathbb{N}$ assigns
a positive integer to each non-empty word $w\in\mathbb{A}^*$ and zero to the empty word.
A truncation is specified by a grading function and truncation value $n\in\mathbb{N}$. 
We write $\pi_{\g=n}$, $\pi_{\g\leq n}$ and $\pi_{\g\geq n}$ for the projections of $\mathbb{A}^*$ 
onto the following subsets: 
(i) $\pi_{\g=n}(\mathbb{A}^*)\coloneqq\{w\in\mathbb{A}^*\colon \mathrm{g}(w)=n\}$;
(ii) $\pi_{\g\leq n}(\mathbb{A}^*)\coloneqq\{w\in\mathbb{A}^*\colon \mathrm{g}(w)\leq n\}$ and
(iii) $\pi_{\g\geq n}(\mathbb{A}^*)\coloneqq\{w\in\mathbb{A}^*\colon \mathrm{g}(w)\geq n\}$.
\end{definition}
A numerical integration scheme based on the stochastic Taylor expansion is thus 
given by successive applications of the approximate flow
\begin{equation*}
\sum_{w\in\pi_{\g\leq n}(\mathbb{A}^*)} I_w(t) \tilde{V}_w
\end{equation*}
across the computational subintervals. 
Now more generally, consider a larger class of numerical schemes that are constructed as follows. 
For a given invertible map $f\colon\mathrm{Diff}(\mathbb{R}^N)\rightarrow\mathrm{Diff}(\mathbb{R}^N$), 
we construct a series expansion for $f(\varphi_t)$ using the stochastic Taylor expansion for $\varphi_t$. 
We truncate the series and simulate the retained iterated It\^o integrals across each
computational subinterval. An integration scheme is obtained by 
computing the inverse map $f^{-1}$ of the simulated truncations at each step; 
see Malham \& Wiese~\cite{MalWie} and Ebrahimi--Fard \textit{et al.\/} \cite{EfLMMkW}. 
We now develop an algebraic framework for studying such map-truncate-invert schemes.
The starting point is the quasi-shuffle algebra; see Hudson~\cite{Hudson_2} and Hoffman~\cite{Hof}.
This gives an explicit description of the algebra of iterated integrals 
of the extended driving processes. Let $\mathbb{RA}$ denote the 
$\mathbb{R}$-linear span of $\mathbb{A}$, and let  $\mathbb{R}\langle \mathbb{A}\rangle$ 
denote the vector space of polynomials 
in the non-commuting variables in $\mathbb{A}$.
\begin{definition}[Quasi-shuffle product]\label{qs defn}
For a given alphabet $\mathbb{A}$, suppose 
$[\,\cdot\,,\,\cdot\,]\colon\mathbb{RA} \otimes \mathbb{RA} \rightarrow \mathbb{RA}$ 
is a commutative, associative product on $\mathbb{RA}$. 
The quasi-shuffle product 
on $\mathbb{R}\langle \mathbb{A}\rangle$, which is commutative, 
is generated inductively as follows:
if $\mathbbold{1}$ is the empty word then $u*\mathbbold{1}=\mathbbold{1}*u=u$ and
\begin{equation*}
ua*vb = (u*vb)a + (ua*v)b + (u*v)[a,b],
\end{equation*}
for all words $u,v\in\mathbb{A}^*$ and letters $a,b\in\mathbb{A}$. 
Here $ua$ denotes the concatenation of $u$ and $a$. 
\end{definition}
\begin{remark}[Word-to-integral isomorphism] 
The word-to-integral map $\mu\colon w\mapsto I_w$ 
is an algebra isomorphism. Here the domain is the vector space 
$\mathbb{R}\langle\mathbb{A}\rangle$ equipped with the quasi-shuffle product.
Iterated integrals indexed by polynomials are defined by linearity, i.e.\/ 
$I_{k_u u + k_v v} = k_u I_{u} + k_v I_{v}$,
for any constants $k_u,k_v\in\mathbb{R}$ and words $u,v\in\mathbb{A}^*$.
This was proved in Curry \textit{et al.\/} \cite{CEfMW}, it had already been
established by Gaines \cite{Gai} for drift-diffusions
and Li \& Liu \cite{LiLiu} for jump-diffusions. 
\end{remark}
\begin{remark}[Shuffle product]
The quasi-shuffle product is a deformation of the shuffle product on 
$\mathbb{R}\langle \mathbb{A}\rangle$. The deformation is induced 
by an additional product $[\,\cdot\,,\,\cdot\,]$ defined on $\mathbb{RA}$. 
If the product $[\,\cdot\,,\,\cdot\,]$ is 
identically zero,
the quasi-shuffle product is just the shuffle product. 
\end{remark}
\begin{remark}[Quadratic covariation bracket]
Here we associate the product 
$[\,\cdot\,,\,\cdot\,]\colon\mathbb{RA}\otimes\mathbb{RA}\rightarrow\mathbb{RA}$ 
underlying the quasi-shuffle with the pullback under the word-to-integral map $\mu$ 
of the quadratic covariation bracket of semimartingales. 
Explicitly, $[0,a]$ is zero for all $a$, and
$[i^{(p)},j^{(q)}] = \delta_{ij} \bigl(\lambda(i,p,q)\,\cdot\,0 
+ \big(1_{\{d+1,\ldots,\ell\}}(i)\big)\,\cdot\, i^{(p+q)}\bigr)$,
where $\lambda(i,p,q)=1$ if $i\in\{1,\ldots,d\}$, 
and $\lambda(i,p,q) = \int_{\mathbb{R}} v^{p+q} \meas^i(\mathrm{d}v)$ if $i\in\{d+1,\ldots,\ell\}$. 
The $0$ refers to the letter $0\in\mathbb{A}$, 
and $1_{\{d+1,\ldots,\ell\}}$ is the indicator function of the set $\{d+1,\ldots,\ell\}$. 
\label{rmk:cov}
\end{remark}
\begin{remark}[Word-to-operator homomorphism]
The word-to-operator map $\kappa\colon w\mapsto \tilde{V}_w$ is an algebra homomorphism. 
Here the domain is $\mathbb{R}\langle\mathbb{A}\rangle$ equipped with the concatenation product. 
This follows as $\mathbb{R}\langle\mathbb{A}\rangle$ equipped with the concatenation product 
is the free associative $\mathbb{R}$-algebra over $\mathbb{A}$, see Reutenauer~\cite{Reu}, 
and each $\tilde{V}_w$ is given by the associative composition of operators $\tilde{V}_{a_i}$. 
\end{remark}
We write $\mathbb{R}\langle\mathbb{A}\rangle_*$ for the quasi-shuffle algebra, 
otherwise the product on $\mathbb{R}\langle\mathbb{A}\rangle$ is concatenation. 
The preceding observations combine to give an encoding 
of integration schemes in the algebra 
$\mathbb{R}\langle\mathbb{A}\rangle_* \overline{\otimes}\, \mathbb{R}\langle\mathbb{A}\rangle$, 
where $\overline{\otimes}$ is the completed tensor product of Reutenauer~\cite[p.~18, 29]{Reu}. 
\begin{proposition}[Algebraic encoding of the flowmap]\label{prop encoding}
For a given L\'{e}vy-driven equation, the map 
$\mu\otimes\kappa$ is a homomorphism from 
$\mathbb{R}\langle\mathbb{A}\rangle_* \overline{\otimes}\, \mathbb{R}\langle\mathbb{A}\rangle$ 
to the tensor product of the algebra of iterated integrals 
of the extended driving processes and the composition algebra 
generated by the set of operators $\tilde{V}_w$. 
Moreover, the flowmap of the equation is the image under $\mu\otimes\kappa$ of the element
\begin{equation*}
\sum_{w\in \mathbb{A}^*} w\otimes w.
\end{equation*}
\end{proposition}
Truncations of this abstract series representation in 
$\mathbb{R}\langle\mathbb{A}\rangle_*\overline{\otimes}\,\mathbb{R}\langle\mathbb{A}\rangle$ 
generate approximations of the flowmap and hence classes 
of stochastic Taylor numerical integration schemes. 
Using this context, we now construct an abstract representation for 
$f(\varphi_t)$ for any given map 
$f\colon\mathrm{Diff}(\mathbb R^N)\to\mathrm{Diff}(\mathbb R^N)$ 
which is expressible as a power series. 
The key idea is that $f(\varphi_t)$ may be rewritten as the image under 
$\mu\otimes\kappa$ of a series $\sum F(w)\otimes w$, 
where $F\in\mathrm{End}(\mathbb{R}\langle \mathbb{A}\rangle_*)$, 
the space of $\mathbb{R}$-linear maps from the quasi-shuffle algebra $\mathbb{R}\langle \mathbb{A}\rangle_*$ 
to itself; see Malham \& Wiese~\cite{MalWie} who established this in the shuffle product context. 
The explicit form of $F\in\mathrm{End}(\mathbb{R}\langle \mathbb{A}\rangle_*)$ is obtained using the 
convolution algebra associated to the quasi-shuffle product.
\begin{definition}[Quasi-shuffle convolution product]
For a given quasi-shuffle product $*$ on $\mathbb{R}\langle\mathbb{A}\rangle_*$, 
the convolution product $\star$ of $F$ and $G$ in 
$\mathrm{End}(\mathbb{R}\langle\mathbb{A}\rangle_*)$ is given by
\begin{equation*}
F\star G \coloneqq *\circ(F\otimes G)\circ\Delta,
\end{equation*} 
where $\Delta$ is the deconcatenation coproduct that 
sends a word $w$ to the sum of all its two-partitions $u\otimes v$ and extends linearly
to $\mathbb{R}\langle\mathbb{A}\rangle$, 
i.e.\/ explicitly for any word $w\in\mathbb A^*$ we have 
$\Delta(w)=\sum_{uv=w}u\otimes v$ and 
\begin{equation*}
\bigl(F\star G\bigr)(w)=\sum_{uv=w}F(u)*G(v).
\end{equation*} 
\end{definition}
The quasi-shuffle algebra $\mathbb{R}\langle\mathbb{A}\rangle_*$ together 
with the coproduct $\Delta$ forms a bialgebra; 
see Hoffman~\cite{Hof} and Hudson~\cite{Hudson_1}. 
In particular, when equipped with the convolution product, 
the space $\mathrm{End}(\mathbb{R}\langle\mathbb{A}\rangle_*)$ becomes a unital algebra, 
where the unit $\nu$ is given by the composition of the unit of the quasi-shuffle algebra 
and the counit of the deconcatenation coalgebra, see Abe \cite{Abe}. 
Explicitly, $\nu$ is the linear map that sends non-empty words to zero and the empty word to itself. 
We define the embedding $\mathrm{End}(\mathbb{R}\langle\mathbb{A}\rangle_*)
\rightarrow\mathbb{R}\langle\mathbb{A}\rangle_*\overline{\otimes}\,\mathbb{R}\langle\mathbb{A}\rangle$ 
by 
\begin{equation*}
F\mapsto \sum_{w\in \mathbb{A}^*} F(w)\otimes w.
\end{equation*}
This is an algebra homomorphism for the quasi-shuffle convolution product;
see Reutenauer \cite{Reu}, Curry~\cite{Cur} 
and Ebrahimi--Fard \textit{et al.\/} \cite{EfMPW2}.
Given a power series $f(x)=\sum_{k\geq0} c_k x^k$ with $c_k\in\mathbb{R}$, 
we define the convolution power series $F^{\star}(X)\coloneqq\sum_{k\geq0} c_k X^{\star k}$, 
where the $X^{\star k}$ are the $k$th convolution powers of 
$X\in\mathrm{End}(\mathbb{R}\langle\mathbb{A}\rangle_*)$. 
The following representation of $f(\varphi_t)$ is immediate.
\begin{proposition}[Convolution power series]\label{conv ps}
If $f\colon\mathbb{R}\langle\mathbb{A}\rangle_*\overline{\otimes}\,\mathbb{R}\langle\mathbb{A}\rangle
\to\mathbb{R}\langle\mathbb{A}\rangle_*\overline{\otimes}\,\mathbb{R}\langle\mathbb{A}\rangle$ 
has a power series $f(x)=\sum_{k\geq0} c_k x^k$, we have 
\begin{equation*}
f\left(\sum_{w\in\mathbb{A}^*} w\otimes w\right) 
=\sum_{w\in\mathbb{A}^*} F^{\star}(\id)(w) \otimes w .
\end{equation*}
Further, when the power series $f$ has an inverse 
$f^{-1}(x)=\sum_{k\geq0} b_k x^k$ with $b_k\in\mathbb{R}$, 
the compositional inverse of the convolution power series for $F$ 
is given by the associated convolution power series $F^{-1}$. 
\end{proposition}
In other words the pre-image under $\mu\otimes\kappa$ of $f(\varphi_t)$ for any such
$f\colon\mathrm{Diff}(\mathbb{R}^N)\rightarrow\mathrm{Diff}(\mathbb{R}^N)$,
which we also represent by
$f\colon\mathbb{R}\langle\mathbb{A}\rangle_*\overline{\otimes}\,\mathbb{R}\langle\mathbb{A}\rangle
\to\mathbb{R}\langle\mathbb{A}\rangle_*\overline{\otimes}\,\mathbb{R}\langle\mathbb{A}\rangle$,
is given by the term on the left in the proposition, and thus $f(\varphi_t)$
can be represented by $F^\star(\id)$.
Combining the algebraic encoding of the flowmap
in Proposition~\ref{prop encoding} with 
Proposition~\ref{conv ps}, 
we construct the map-truncate-invert scheme associated 
to a given power series $f$ as follows.
\begin{procedure}[Map-truncate-invert scheme]\label{proc}
Let $f\colon\mathrm{Diff}(\mathbb{R}^N)\rightarrow\mathrm{Diff}(\mathbb{R}^N)$ 
be an invertible map admitting an expansion as a power series. 
For a given truncation function $\pi_{\g\leq n}$, 
the associated map-truncate-invert scheme across 
a fixed computational interval $[0,t]$ is obtained as follows.
\begin{enumerate}
\item Construct the series $F^\star(\id)$;
\item Simulate the truncation $\pi_{\g\leq n}\circ F^\star(\id)$ given by
$\hat{\sigma}_t\coloneqq\sum_{w\in\pi_{\g\leq n}(\mathbb{A}^*)}I_{F^{\star}(\id)(w)}(t) \tilde{V}_w(\id)$.
\item Compute the approximation $f^{-1}(\hat{\sigma}_t)\circ y_0$.
\end{enumerate}
\end{procedure}
\begin{remark}
In $\hat{\sigma}_t$ in (ii) above: The identity map in $F^{\star}(\id)(w)$ is the 
identity endomorphism in $\mathrm{End}(\mathbb{R}\langle\mathbb{A}\rangle_*)$,
while the identity map in $\tilde{V}_w(\id)$ is the identity diffeomorphism
on $\mathbb{R}^N$.
\end{remark}
\begin{remark}[Map-truncate-invert endomorphism]
Map-truncate-invert schemes are realizations 
of the integration scheme associated to the endomorphism 
$F^{-1}\circ\pi_{\g\leq n}\circ F^{\star}(\id)\in\mathrm{End}(\mathbb{R}\langle\mathbb{A}\rangle)$. 
\end{remark}
An endomorphism which will prove useful in what follows is the augmented ideal projector.
\begin{definition}[Augmented ideal projector]
The augmented ideal projector denoted by $\sJ$ is given by $\sJ\coloneqq\id-\nu$. 
In other words, it acts as the identity on non-empty words, but sends the empty word to zero. 
\end{definition}
\begin{example}[Exponential Lie series]
Recall from the introduction that the motivating example 
of a map-truncate-invert scheme was the exponential Lie series integrator, 
corresponding to $f=\log$. For any endomorphism $X$ that maps the empty word to itself, 
$\log^{\star}(X)$ is a power series in $X-\nu$. 
In particular, we have 
$\log^{\star}(\id)=\sJ-\sJ^{\star 2}/2+\cdots+(-1)^{k+1}\sJ^{\star k}/k+\cdots$.
\end{example}
\begin{remark} 
For any word $w$, say of length $k$, we can naturally truncate a series
in convolutional powers of $\sJ$ with real coefficients $c_1$, $c_2$, $c_3$ and so forth, 
to $c_1\sJ+c_2\sJ^{\star2}+\cdots+c_k\sJ^{\star k}$. This is because the action of 
$\sJ^{\star(k+1)}$ and subsequent terms in the series is, by convention, 
zero as a word of length $k$ cannot be partitioned into more than $k$ non-empty parts. 
\end{remark}
The computation of $f^{-1} (\hat{\sigma}_t)$ from $\hat{\sigma}_t$ is in general non-trivial. 
For the exponential Lie series integrator, 
Castell \& Gaines \cite{CasGai95,CasGai96} proposed computing $\exp(\hat{\sigma}_t)(y_0)$ 
by numerical approximation of the ordinary differential equation 
$u'=\hat{\sigma}_t(u)$, subject to the initial condition $u(0)=y_0$. 
An approximate solution may then be recovered from $u(1)$. 
For the sinhlog integrator in the shuffle product context, Malham \& Wiese \cite{MalWie} 
showed that for linear constant coefficient equations, since $\hat{\sigma}$ is a square matrix, 
the approximate flow 
$\exp\sinh^{-1}(\hat{\sigma}_t)=\hat{\sigma}_t+\bigl(\id+\hat{\sigma}^2_t\bigr)^{1/2}$ 
may be computed using a matrix square root. If the vector fields are nonlinear 
Malham \& Wiese \cite{MalWie} suggested expanding the square root to sufficiently  
high degree terms. In \textsection5 we introduce direct map-truncate-invert schemes
which evaluate $F^{-1}\circ\pi_{\g\leq n}\circ F^{\star}(\id)$ directly and circumvent
the step involving the computation of $f^{-1}$ in general.

\section{Endomorphism inner product, gradings and error analysis}\label{sec:wlvsms}
We establish convergence and accuracy of integration schemes at an algebraic level. 
We first define an inner product on the space of endomorphisms corresponding 
to the $L^2$ inner product of the associated approximate flows. 
We then define the mean-square grading and introduce stochastic Taylor
schemes of a specific local and then global order. We subsequently consider 
word length grading. We show stochastic Taylor integration schemes of a 
given strong order obtained by truncating according to word length are always 
more accurate than those obtained by truncating according to the mean-square grade. 
We measure accuracy by the leading order term 
in the $L^2$-norm of the remainder after truncation. To start, to define the 
endomorphism inner product, we require an algebraic encoding of the expectation 
of the iterated integrals.
\begin{definition}[Expectation map]\label{expectation}
For any word $w\in\mathbb{A}^*$, the expectation map 
$E\colon \mathbb{R}\langle\mathbb{A}\rangle \rightarrow \mathbb{R}[t]$ is defined by
$E\colon w\mapsto t^{|w|}/|w|!$ if $w\in\{0\}^*$ and is zero for all other words. 
Here $|w|$ denotes the length of the word $w$, 
$\{0\}^*\subset\mathbb{A}^*$ is the free monoid over the letter $0$ 
and $\mathbb{R}[t]$ is the polynomial ring 
over a single indeterminate $t$ commuting with $\mathbb{R}$.
\end{definition}
\begin{remark}[Expectation of iterated It\^o integrals]
The expectation map corresponds to the expectation of iterated It\^o integrals
as follows. First, integrals indexed by words not ending in the letter $0$ are martingales 
and hence have zero expectation. Second, consider integrals indexed by a word with 
at least one non-zero letter. Fubini's Theorem implies such integrals also have zero expectation; 
see Protter~\cite{Pro}. 
\end{remark}
We now define an inner product on $\mathrm{End}(\mathbb{R}\langle\mathbb{A}\rangle_*)$, 
following Ebrahimi--Fard \textit{et al.\/ } \cite{EfLMMkW}.
Suppose we apply two separate functions to the flowmap which are
characterised by the endomorphisms $F$ and $G$ in 
$\mathrm{End}(\mathbb{R}\langle\mathbb{A}\rangle_*)$.
The stochastic processes generated by these 
functions of the flowmap, for given initial data $y_0$, are 
$\sum_{w\in\mathbb A^*} I_{F(w)}\tilde{V}_w(y_0)$ and $\sum_{w\in\mathbb A^*}I_{G(w)}\tilde{V}_w(y_0)$.
Our goal is to define an inner product $\langle F, G\rangle$ 
on $\mathrm{End}(\mathbb{R}\langle\mathbb{A}\rangle_*)$ 
which matches the $L^2$-inner product of these two vector-valued stochastic processes.
However we would like all of our results to be independent of 
the initial data $y_0$ and the governing vector fields appearing in the $\tilde{V}_w$ terms. 
We achieve this by replacing the vectors $\tilde{V}_w(y_0)$ 
with a set of indeterminate 
vectors indexed by words, $\{\mathbf{V}_w\}_{w\in \mathbb{A}^*}$. 
We write $(u,v)$ for the inner product of $\mathbf{V}_u$ and $\mathbf{V}_v$, 
i.e.\/ $(u,v)\coloneqq \mathbf{V}_u^{\mathrm{T}}\,\mathbf{V}_{v}$.
Let $\mathbf{V}$ denote the infinite square matrix indexed by the words $u,v\in\mathbb A^*$
with entries $(u,v)$.
\begin{definition}[Inner product]\label{inner product 2}
We define the inner product of endomorphisms $F$ and $G$ in
$\mathrm{End}(\mathbb{R}\langle\mathbb{A}\rangle_*)$ with respect to $\mathbf{V}$ to be
\begin{equation*}
 \langle F, G\rangle \coloneqq \sum_{u,v\in \mathbb{A}^*}E\,\bigl(F(u)*G(v)\bigr)\, (u,v).
\end{equation*}
\end{definition}
All results we subsequently establish will hold independent of $\mathbf V$.
\begin{remark}[Positive definiteness]
As the operators $\tilde{V}_i$ typically include second 
(or higher) order differential operators that send the identity map to zero, 
distinct endomorphisms may be associated with the same stochastic process. 
For linear constant coefficient equations for instance, 
the operators $\tilde{V}_{i^{(m)}}$ are zero operators for all $m>1$, 
and the stochastic process associated to any endomorphism is trivial if its image lies 
in the two-sided ideal in $\mathbb{R}\langle\mathbb{A\rangle_*}$ 
generated by the letters $i^{(m)}, m>1$. 
To obtain positive definiteness of the inner product, 
we pass to the quotient of $\mathrm{End}(\mathbb{R}\langle\mathbb{A}\rangle_*)$ 
under the equivalence relation for which endomorphisms yielding the same 
stochastic process are equivalent. Positive definiteness on the quotient space 
follows by similar arguments to those in Ebrahimi--Fard \textit{et al.\/} \cite{EfLMMkW}. 
\end{remark}
The norm, in this quotient space, 
of an endomorphism $F\in\mathrm{End}(\mathbb{R}\langle\mathbb{A}\rangle_*)$
is $\|F\|\coloneqq\langle F, F\rangle^{1/2}$.

The first part of our error analysis uses mean-square grading, 
see Platen \& Bruti-Liberati~\cite{PlaBl}.
\begin{definition}[Mean-square grading]
For any word $w\in\mathbb A^*$, the map 
$\mathrm{g}^{\mathrm{ms}}\colon w\mapsto2\zeta(w)+\xi(w)$
is the mean-square grading,
where $\zeta(w)$ and $\xi(w)$ are the number of zero 
and non-zero letters in $w$, respectively.
\end{definition}
\begin{definition}[Reduced words]
For a given word $w$ the reduced word $\mathrm{red}(w)$ is defined as the 
word obtained by deleting any zero letters, and replacing any 
letters of the form $i^{(m)}$ with $i$. 
\end{definition}
\begin{example}
If $w=010024^{(3)}30$, we have $\mathrm{red}(w)=1243$. 
\end{example}
The following result will be useful in our subsequent error analysis,
see Kloeden \& Platen~\cite{KloPla}.
\begin{lemma}[Expectation of products]\label{MSG lemma}
Let $\mathbb{R}\langle\mathbb{A}\rangle_*$ be the quasi-shuffle algebra based on
the set of independent L\'{e}vy processes $\{t,W^1,\ldots,W^d,J^{d+1},\ldots,J^{\ell}\}$
extended by covariation. 
For any words $u,v\in\mathbb{A}^*$, if $\mathrm{red}(u)=\mathrm{red}(v)$ then 
there is a non-zero constant $C=C(u,v)$ such that 
\begin{equation*}
E(u*v)=C(u,v)t^{(\mathrm{g}^{\mathrm{ms}}(u)+\mathrm{g}^{\mathrm{ms}}(v))/2},
\end{equation*}
otherwise the expectation of the product $u*v$ is zero.
\end{lemma}
\begin{proof}
Let $\langle p,w\rangle_{\mathbb{R}\langle\mathbb{A}\rangle}$ denote the 
coefficient of a given word $w$ in the polynomial $p$. We then have
\begin{equation*}
E(u*v)=\sum_{w\in\{0\}^*}\langle u*v,w\rangle_{\mathbb{R}\langle\mathbb{A}\rangle}\frac{1}{|w|!}t^{|w|}.
\end{equation*}
Consider the generating relation of the quasi-shuffle product in Definition~\ref{qs defn}, 
which states that $ua*vb = (u*vb)a + (ua*v)b + (u*v)[a,b]$.
We see that if a summand in the polynomial $ua*vb$ 
is to be a multiple of $w\in\{0\}^*$, we require either $a=0$, $b=0$ or 
$\langle[a,b],0\rangle_{\mathbb{R}\langle\mathbb{A}\rangle}$ non-zero. 
As the quadratic covariation of $t$ with any L\'{e}vy process vanishes, 
all zero letters contribute to the expectation only through the first two terms in the 
quasi-shuffle generating relation above.
In contrast, non-zero letters $a$ and $b$ contribute only through the third term. 
We see that each zero letter appears exactly once in any $w\in\{0\}^*$ 
for which $\langle u*v,w\rangle_{\mathbb{R}\langle\mathbb{A}\rangle}$ is non-zero, 
and any pair of non-zero letters $a,b$ contribute one letter together. 
In particular, we have $E(u*v)=C(u,v) t^{(\mathrm{g}^{\mathrm{ms}}(u)+\mathrm{g}^{\mathrm{ms}}(v))/2}$, 
where $C(u,v)$ may equal zero. 

By the independence of the driving processes, we have $[i,j]$, $[i^{(p)},j]$ and $[i^{(p)},j^{(q)}]$
are all zero for all $i,j,p,q$, $i\neq j$. Moreover, we have
$[J^{i^{(p)}},J^{i^{(q)}}]_t = J^{i^{(p+q)}}_t + t\int_{\mathbb{R}} v^{p+q} \meas^i(\mathrm{d}v)$. 
In particular, we have that $\langle[i^{(p)},i^{(q)}],0\rangle_{\mathbb{R}\langle\mathbb{A}\rangle}$ 
is non-zero for all $i>d$ and all $p,q\geq 1$ for which $i^{(p)}$ and $i^{(q)}$ exist. 
Moreover, as $[W^i,W^i]_t=t$, we have $[i,i]=0$ for $i=1\ldots,d$.

The expression $u*v$ is a linear combination of words, each of which arises 
from a choice of one of the three terms at each stage in the 
inductive quasi-shuffle generating relation.
Indeed the $k$th letter in a given word thus obtained 
is the letter $a$, $b$ or $[a,b]$ chosen at the $k$th application of the inductive definition. 
For this word to consist of only zeros, we must choose the first or the second term 
as long as either $a$ or $b$ is zero. When $a$ and $b$ are both non-zero, 
we must choose the third term. From the preceding paragraph, 
this will be a sum featuring a multiple of the zero letter provided 
both letters are either equal to $i$, where $i=1,\ldots,d$, 
or are $i^{(p)}$ and $i^{(q)}$ respectively for some $i>d$ and $p,q\geq 1$. 
Continuing this procedure, we see that we will obtain a word comprising 
only zeros if and only if the reduced words $\mathrm{red}(u)$ and $\mathrm{red}(v)$ are equal.
\end{proof}
We wish to compare the errors of different integration schemes across 
a given time step at the algebraic level, using the inner product defined above. 
As the flowmap corresponds to the identity in $\mathrm{End}(\mathbb{R}\langle\mathbb{A}\rangle_*)$, 
we define the remainder endomorphism as follows.
\begin{definition}[Remainder endomorphism]\label{remainder}
For any endomorphism $H\in\mathrm{End}(\mathbb{R}\langle\mathbb{A}\rangle_*)$ 
encoding an approximation of the flowmap, 
we define the associated remainder endomorphism $R$ to be 
\begin{equation*}
R\coloneqq\id-H.
\end{equation*}
\end{definition}
\begin{example} 
A simple example is $H=\pi_{\g\leq n}$ corresponding to 
a stochastic Taylor expansion truncated according to a given grading `$\mathrm{g}$'.
In this case $R=\id-\pi_{\g\leq n}=\pi_{\g\geq n+1}$. 
More generally we might have $H=F^{-1}\circ\pi_{\g\leq n}\circ F(\id)$ so
$R=\id-F^{-1}\circ\pi_{\g\leq n}\circ F(\id)$. 
\end{example}
For the rest of this section we focus on numerical schemes constructed
by truncating the stochastic Taylor expansion, first, according to mean-square
grading and, second, according to word length grading.
We note the flowmap $\varphi_t\in L^2$, see Lemma~\ref{flow lemmas}.
Let $\hat\varphi_t^{\mathrm{ms}}$ 
denote the truncated stochastic Taylor expansion, truncated according to mean-square grading. 
We denote the corresponding remainder by 
$R_t^{\mathrm{ms}}\coloneqq\varphi_t-\hat\varphi_t^{\mathrm{ms}}$
and note $R_t^{\mathrm{ms}}\in L^2$, see Platen \& Bruti--Liberati~\cite{PlaBl}. 
Locally the numerical scheme is of mean-square order $n$ if 
\begin{equation*}
R_t^{\mathrm{ms}}=\sum_{w\in\pi_{\g\geq n+1}(\mathbb{A}^*)}I_w(t)\tilde V_w.
\end{equation*}
In particular this implies, for sufficiently small $t$, that 
$\|R_t^{\mathrm{ms}}(y_0)\|_{L^2}^2=(1+|y_0|^2)\cdot\mathcal O(t^{n+1})$ 
for any initial data $y_0\in\mathbb R^N$.
We naturally apply the numerical scheme $\hat\varphi_t^{\mathrm{ms}}$
successively over a suitably fine discretization of the global time interval
of integration to obtain a suitably accurate numerical approximation. This 
means we need to determine the global order of convergence for such
a scheme. Milstein's Theorem \cite{Mil,MilBook} provides a mechanism for 
inferring global convergence estimates from local ones in the case of 
drift-diffusion equations. This can be extended to L\'evy driven 
equations and is provided in Theorem~\ref{mt conv} in \textsection\ref{sec:glob}. 
We call it the Generalized Milstein Theorem. To apply this theorem we
require local expectation estimates for $R_t^{\mathrm{ms}}(y_0)$.
We find that
\begin{equation*}
E\bigl(R_t^{\mathrm{ms}}(y_0)\bigr)
=\sum_{k\geq\lfloor n/2\rfloor+1}\frac{t^{k}}{k!}\tilde V_{0^k}(y_0).
\end{equation*}
Here we used that $E\bigl(I_w(t)\bigr)$ is only non-zero for words $w\in\{0\}^*$. 
The notation $0^k$ denotes such a word $w$ of length $k$, and $\mathrm{g}^{\mathrm{ms}}(0^k)=2k$. 
Using the linear growth estimates we observe for some constant $K>0$ we have
\begin{equation*}
\Bigl|E\bigl(R_t^{\mathrm{ms}}(y_0)\bigr)\Bigr|
\leq K\,(1+|y_0|^2)^{1/2}\,\Biggl(\sum_{k\geq\lfloor n/2\rfloor+1}\frac{t^{k}}{k!}\Biggr).
\end{equation*}
For any finite $t$ the sum is convergent. In particular for small $t$ the upper bound
is $\mathcal O(t^{\lfloor n/2\rfloor+1})$. Recall from above that 
$\|R_t^{\mathrm{ms}}(y_0)\|_{L^2}^2=(1+|y_0|^2)\cdot\mathcal O(t^{n+1})$. 
We now refer to the Generalized Milstein Theorem~\ref{mt conv} in \textsection\ref{sec:glob}.
Matching parameters we see that $p_1=\lfloor n/2\rfloor+1$ and $p_2=(n+1)/2$.
The theorem states that the approximation $\hat\varphi_t^{\mathrm{ms}}$
with remainder $R_t^{\mathrm{ms}}$ above will converge globally at rate $p_2-1/2$
if $p_1\geq p_2+1/2$. While this is true for when $n$ is even, it does \emph{not} hold
when $n$ is odd. This is simply due to the fact that pure deterministic terms in 
the stochastic Taylor series remainder have a whole integer less 
root mean-square global order of convergence compared to their local order of convergence.
To rectify this we can simply modify our scheme $\hat\varphi_t^{\mathrm{ms}}$
to 
\begin{equation*}
\hat\varphi_t^{\mathrm{ms}}
=\sum_{\mathrm{g}^{\mathrm{ms}}(w)\leq n}I_w(t)\tilde V_w+I_{0^{n^*}}(t)\tilde V_{0^{n^*}},
\end{equation*}
where $n^*\coloneqq\lfloor(n+1)/2\rfloor$ if $n$ is odd and zero if it is even.
This means that the leading order deterministic term in $R_t^{\mathrm{ms}}$
is of the same order as previously when $n$ is even, but is of order
$\lfloor(n+1)/2\rfloor+1=\lfloor(n+3)/2\rfloor$ when $n$ is odd. By inspection
we observe that $p_1\geq p_2+1/2$ is now satisfied. The modified scheme $\hat\varphi_t^{\mathrm{ms}}$
above has mean-square global order of convergence $n$ and the terms included
exactly match those specified in Platen \& Bruti--Liberati~\cite[p.~290]{PlaBl}.

We now consider word length grading which we employ for our main result 
in the next section.
\begin{definition}[Word length grading]
For a given word $w\in\mathbb{A}^*$, the word length grading is denoted 
$|w|$ or $\mathrm{g}^{\mathrm{wl}}$, 
and defined to be the number of letters in $w$.
\end{definition}
\begin{remark}[Computational effort]
The bulk of the computation effort, when implementing accurate strong 
numerical schemes derived from truncations of series representations of the flowmap, 
is associated with the simulation of the iterated integrals $I_w$.
See Lord, Malham \& Wiese~\cite{LMW} and Malham \& Wiese~\cite{MW} 
for more details in the drift-diffusion case.
In particular the iterated integrals involving the most distinct 
non-deterministic letters require the most effort.
Hence the additional computational cost required to simulate all 
the iterated integrals $\{I_w\colon |w|\leq n\}$ as opposed to
the subset $\{I(w)\colon\mathrm{g}^{\mathrm{ms}}(w)\leq n\}$, is minimal. 
\end{remark}
One benefit of truncating to word length as opposed to mean-square grading is 
the following.
\begin{theorem}[Mean-square versus word length graded truncations]\label{th:msvswo}
Let $R_t^{\mathrm{wl}}(y_0)$ and $R_t^{\mathrm{ms}}(y_0)$ denote
the remainders generated by truncating the separated stochastic Taylor expansion 
for $y_t$ respectively using the mean-square and word length gradings, 
both truncated at the same given grade $n$. Then at leading order for all $n$ we have
\begin{equation*}
\bigl\|R_t^{\mathrm{wl}}(y_0)\bigr\|_{L^2}^2\leqslant\bigl\|R_t^{\mathrm{ms}}(y_0)\bigr\|_{L^2}^2.
\end{equation*}
\end{theorem}
\begin{proof}
First, recall the definitions for the mean-square and word length gradings. 
We observe for any word $w\in\mathbb A^*$ we have 
$\mathrm{g}^{\mathrm{ms}}(w)=2\zeta(w)+\xi(w)$ and $\mathrm{g}^{\mathrm{wl}}(w)=\zeta(w)+\xi(w)$.
Hence we have $\mathrm{g}^{\mathrm{wl}}\leq\mathrm{g}^{\mathrm{ms}}$. This 
implies that for any alphabet $\mathbb A$ constructed from at least one deterministic 
and one non-deterministic letter, we have 
$\pi_{\mathrm{g}^{\mathrm{ms}}\leq n}\bigl(\mathbb A^*\bigr)
\subset\pi_{\mathrm{g}^{\mathrm{wl}}\leq n}\bigl(\mathbb A^*\bigr)$.
Thus correspondingly for their complements 
$\pi_{\mathrm{g}^{\mathrm{wl}}\geq n+1}\bigl(\mathbb A^*\bigr)
\subset\pi_{\mathrm{g}^{\mathrm{ms}}\geq n+1}\bigl(\mathbb A^*\bigr)$.
Second, let $R^{\mathrm{ms}}$ and $R^{\mathrm{wl}}$ denote the remainder \emph{endomorphisms}
associated with truncating the stochastic Taylor expansion by mean-square
and word length gradings, respectively. Hence for the inner products 
$\langle R^{\mathrm{ms}}, R^{\mathrm{ms}}\rangle$ and 
$\langle R^{\mathrm{wl}}, R^{\mathrm{wl}}\rangle$ we sum over words
in $\pi_{\mathrm{g}^{\mathrm{ms}}\geq n+1}\bigl(\mathbb A^*\bigr)$ and 
$\pi_{\mathrm{g}^{\mathrm{wl}}\geq n+1}\bigl(\mathbb A^*\bigr)$, respectively.
The difference remainder endomorphism $\hat R\coloneqq R^{\mathrm{ms}}-R^{\mathrm{wl}}$
is at leading order non-zero on words in 
$\pi_{\mathrm{g}^{\mathrm{ms}}=n+1}\bigl(\mathbb A^*\bigr)
\cap\pi_{\mathrm{g}^{\mathrm{wl}}\leq n}\bigl(\mathbb A^*\bigr)$.
For any word $w$ in this set we have $2\zeta(w)+\xi(w)=n+1$ and $\zeta(w)+\xi(w)\leq n$,
which imply $\zeta(w)>0$ and thus $\xi(w)<n+1$. 
Now consider the inner product $\langle R^{\mathrm{wl}},\hat R\rangle$.
Any word $u$ on which $R^{\mathrm{wl}}$ is non-trivial at leading order has
length $\zeta(u)+\xi(u)\geq n+1$. Any word $v$ on which $\hat R$ is non-trivial at leading order,
we have just shown that $\xi(v)<n+1$. Hence we must have $\xi(u)\neq \xi(v)$ and so 
$\mathrm{red}(u)\neq\mathrm{red}(v)$. We deduce that $\langle R^{\mathrm{wl}},\hat R\rangle$ is zero.
The result then follows from the relation 
$\|R^{\mathrm{ms}}\|^2=\|R^{\mathrm{wl}}\|^2+2\langle\hat R, R^{\mathrm{wl}}\rangle+\|\hat R\|^2$.
This result holds whether we include the $0^{n^*}$ in $\hat\varphi_t^{\mathrm{ms}}$ or not.
\end{proof}

\section{Antisymmetric sign reverse integrator}\label{sec:ASRI}
We now present the error analysis of map-truncate-invert schemes, 
concluding with the description of the antisymmetric sign reverse integrator 
and the main theorem concerning its efficiency. We also
present the direct map-truncate-invert schemes alluded to at the 
end of \textsection3.
We begin by introducing the pre-remainder endomorphism associated 
to any map-truncate-invert scheme, that is the remainder terms associated 
to the truncation of the series $F^{\star}(\id)$. 
We relate the pre-remainder and remainder of such schemes 
and use this relation to compute the inverse step of a direct map-truncate-invert scheme. 
We follow with the introduction of the antisymmetric sign reverse integrator 
and its explicit characterization. We then prove our main theorem 
on the efficiency of the antisymmetric sign reverse integrator.
\begin{definition}[Pre-remainder endomorphism]
Let $F\colon\mathrm{End}(\mathbb{R}\langle\mathbb{A}\rangle_*)
\rightarrow\mathrm{End}(\mathbb{R}\langle\mathbb{A}\rangle_*)$
be an invertible map and suppose $\pi_{\g\leq n}$ is a projection corresponding to a truncation. 
The pre-remainder endomorphism $Q$ associated 
to $\pi_{\g\leq n}\circ F(\id)$ is defined by
\begin{equation*}
Q\coloneqq F(\id)-\pi_{\g\leq n}\circ F(\id).
\end{equation*} 
\end{definition}
This definition allows for more general maps $F$ than power series in the convolution algebra, 
which we require presently. The relationship between the pre-remainder and remainder 
is critical to the error analysis of map-truncate-invert schemes. 

Hereafter we assume a grading $\mathrm{g}$ which is preserved by 
the quasi-shuffle product in question. By this we mean that 
for all words $u,v\in\mathbb{A}^*$, the polynomial $u*v$ is homogeneous of degree 
$\mathrm{g}(u)+\mathrm{g}(v)$, i.e.\/ it is a sum of words of grade 
$\mathrm{g}(u)+\mathrm{g}(v)$; see Ebrahimi--Fard \textit{et al.\/} \cite{EfLMMkW}.
In this case we observe that if $\pi$ represents one of the projectors 
$\pi_{\g\leq n}$, $\pi_{\g=n}$ or $\pi_{\g\geq n}$
according to such a grading $\mathrm{g}$,
then $\pi\circ \sJ^{\star k}=\sJ^{\star k}\circ \pi$ 
for any $k\in\mathbb N$. 
\begin{remark}[Gradings preserved by quasi-shuffles]
The power bracket grading introduced in 
Curry \textit{et al.\/} \cite[Section~4]{CEfMW}, 
defined on uncompensated brackets, is grading preserving for any quasi-shuffle. 
This assigns the value 2 to the letter $0$, the value $1$ to letters $1,\ldots,\ell$, 
and the sum of the gradings for the quadratic covariation bracket between any two words. 
Mean-square grading is preserved for drift-diffusions, however it is
not preserved in general for any quasi-shuffle containing discontinuous terms. 
Importantly, word length grading $\mathrm{g}^{\mathrm{wl}}$ is preserved when the 
quasi-shuffle product is in fact the shuffle product. \emph{Importantly}, 
this latter case will underlie our eventual application below.
\end{remark}
\begin{lemma}[Remainder and pre-remainder relation]\label{prerem conv lemm}
Suppose $f(1+x)=\sum_{k\geq 1}c_k x^k$ is a power series with inverse 
$f^{-1}(x)=1+\sum_{k\geq 1}b_kx^k$, where 
$b_1\equiv1/c_1$.
Let $\pi_{\g\leq n}$ be the projection according to
the grading preserved by the quasi-shuffle product.
Let $R$ and $Q$ be the remainder and pre-remainder, respectively, 
associated to the map-truncate-invert endomorphism 
$F^{-1}\circ\pi_{\g\leq n}\circ F^{\star}(\id)$. Then we have 
\begin{equation*}
R=\frac{1}{c_1}Q+\mathrm{h.o.t.}.
\end{equation*}
\end{lemma}
\begin{remark}[Higher order terms] 
The notation `$\mathrm{h.o.t.}$' in Lemma~\ref{prerem conv lemm} 
is used to denote higher order terms in the following sense. Suppose our
goal is to apply the endomorphism $G^\star(\id)=C_1\sJ+C_2\sJ^{\star2}+C_3\sJ^{\star3}+\cdots$
to all words $w$ with $\mathrm{g}(w)\leq n$. 
If we write $G^\star(\id)=C_1\sJ+\cdots+C_n\sJ^{\star n}+\mathcal O(H)$,
this implies the endomorphism $H$ annihilates all words $w$ with $\mathrm{g}(w)\leq n$.
In this sense the term(s) in $\mathcal O(H)$ represent higher order terms.
For example, suppose for a particular word $w$ we have $\mathrm{g}(w)=n$.
The term $\sJ^{\star(n+1)}$ splits $w$ into a sum of all its possible non-empty 
$(n+1)$-partitions quasi-shuffled together of the form $w_1*w_2*\cdots*w_{n+1}$. 
Since the grading is preserved by the quasi-shuffle product, we have 
$\mathrm{g}(w_1)+\cdots+\mathrm{g}(w_{n+1})=n$. However since 
the partitions $w_i$ for all $i=1,\ldots,n+1$ are all non-empty so that
$\mathrm{g}(w_i)\geq1$, we have a contradiction. By convention
we suppose $\sJ^{\star(n+1)}$ annihilates such a word $w$ and so $\sJ^{\star(n+1)}$
represents a higher order term in the sense we have outlined.
\end{remark}
\begin{proof}
We set $P\coloneqq \pi_{\g\leq n}\circ F^{\star}(\id)$. 
Then we see the pre-remainder $Q=\pi_{\g\geq n+1}\circ F^{\star}(\id)$ and 
thus also $P+Q=F^{\star}(\id)$. We observe that since $R=F^{-1}(P+Q)-F^{-1}(P)$
we find 
\begin{equation*}
R=\sum_{k\geq 1} b_k\,\bigl((P+Q)^{\star k}-P^{\star k}\bigr)
=b_1\,Q+b_2\,(P\star Q+Q\star P)+\mathrm{h.o.t.}.
\end{equation*} 
Then we have
$P\star Q=\bigl(\pi_{\g\leq n}\circ F^{\star}(\id)\bigr)
\star\bigl(\pi_{\g\geq n+1}\circ F^{\star}(\id)\bigr)$.
Since $F^{\star}(\id)=c_1\,\sJ+\mathrm{h.o.t.}$ we get
$P\star Q=\bigl(c_1\,\pi_{\g\leq n}\circ \sJ\bigr)
\star\bigl(\pi_{\g\geq n+1}\circ(c_1\,\sJ+\cdots+c_{n+1}\,\sJ^{\star(n+1)})\bigr)+\mathrm{h.o.t.}$.
Hence we deduce that $P\star Q$ annihilates any words of grade less than $n+2$.
Thus the term $b_2 P\star Q$, as well as similarly $b_2\, Q\star P$, represent
higher order terms. Using that $b_1\equiv1/c_1$ gives the result.
\end{proof}
A similar calculation was used to establish efficiency of the sinhlog integrator 
in Ebrahimi--Fard \textit{et al.\/} \cite{EfLMMkW}. 
We use it to compute the inverse step of a map-truncate-invert scheme as follows.
\begin{corollary}[Direct map-truncate-inverse schemes]\label{inv corollary}
Suppose  $f(1+x)=\sum_{k\geq 1}c_k x^k$ is an invertible power series
and let $\pi_{\g\leq n}$ be a truncation according to the grading preserved 
by the quasi-shuffle product. Then $F^{-1}\circ \pi_{\g\leq n}\circ F^{\star}(\id)$,
up to higher order terms, is given by
\begin{equation*}
\pi_{\g\leq n}+\pi_{\g=n+1}\circ\Bigl(\sJ-\frac{1}{c_1}F^\star(\id)\Bigr).
\end{equation*}
\end{corollary}
\begin{proof}
Starting with the definition of the remainder endomorphism $R$, and then
subsequently using our results above, we observe
\begin{align*}
F^{-1}\circ \pi_{\g\leq n}\circ F^{\star}(\id)
&=\id-R\\
&=\pi_{\g\leq n}+\pi_{\g\geq n+1}-\frac{1}{c_1}Q+\mathrm{h.o.t.}\\
&=\pi_{\g\leq n}+\pi_{\g\geq n+1}\circ\Bigl(\id-\frac{1}{c_1}\,F^{\star}(\id)\Bigr)+\mathrm{h.o.t.}\\
&=\pi_{\g\leq n}+\pi_{\g=n+1}\circ\Bigl(\sJ-\frac{1}{c_1}\,F^{\star}(\id)\Bigr)+\mathrm{h.o.t.}.
\end{align*}
In the last step we used that $\sJ$ is the identity on non-empty words.
\end{proof}
\begin{remark}[Direct map-truncate-invert in practice]\label{mti just}
Note that using the power series representation for $F^{\star}(\id)$
the direct map-truncate-inverse scheme in Corollary~\ref{inv corollary} is given by
\begin{equation*}
\pi_{\g\leq n}-\frac{1}{c_1}\pi_{\g=n+1}\circ\bigl(c_2\sJ^{\star 2}+c_3\sJ^{\star 3}+\cdots\bigr)
=\pi_{\g\leq n}-\frac{1}{c_1}\bigl(c_2\sJ^{\star 2}+\cdots+c_{n+1}\sJ^{\star(n+1)}\bigr)\circ\pi_{\g=n+1}.
\end{equation*}
We observe in this formula the action of direct map-truncate-invert schemes.
We simulate the corresponding truncated stochastic Taylor scheme represented by 
$\pi_{\g\leq n}$ and add the corresponding additional terms shown. 
Note the additional terms $\sJ^{\star 2}$, $\sJ^{\star 3}$ and so forth only 
involve products over lower order multiple It\^o integrals that have already 
been simulated in $\pi_{\g\leq n}$.
For instance, for a word $w$ with $\mathrm{g}(w)=n+1$, the term $\sJ^{\star 2}(w)$ is the 
sum of all products of It\^o integrals of the form $I_uI_v$ with non-empty
words $u$ and $v$ such that $uv=w$ and using grade preservation, $\mathrm{g}(u)+\mathrm{g}(v)=n+1$. 
In particular $\mathrm{g}(u)\leq n$ and $\mathrm{g}(v)\leq n$, so $I_u$ and $I_v$ are already
included in $\pi_{\g\leq n}$. The question arises as to whether a given 
direct map-truncate-inverse scheme converges.
For the integrator we construct presently, this will be an automatic consequence of
the convergence of the corresponding stochastic Taylor integrator according to 
word length grading.
\end{remark}
We turn our attention to establishing our main result concerning 
an efficient integrator for equations driven by L\'{e}vy processes.
For convenience we introduce the following bracketing notation for
the bracket $[\,\cdot\,,\,\cdot\,]$ in the quasi-shuffle product. 
For letters $a$ we set $[a]\coloneqq a$, while for words 
$w=a_1\ldots a_k$ we set $[w]\coloneqq [a_1,[a_2,\ldots,[a_{k-1},a_k]\ldots]]$. 
The commutativity of the bracket means the order of the letters $a_i$
is irrelevant and then its associativity means that all $k$-fold brackets
are equivalent to the canonical form of left-to-right bracketing shown. 
\begin{definition}[Reversal, sign reversal 
and quasi-shuffle antipode endomorphisms $|S|$, $S$ and $\hat S$]
We define three endomorphisms on $\mathbb{R}\langle\mathbb{A}\rangle$ as follows. If 
$a_i\in\mathbb A$ for $i=1,\ldots,n$ are letters, then we define the:
(i) Reversal map: $|S|\colon a_1\ldots a_n\mapsto a_n\ldots a_1$;
(ii) Sign reversal map: $S\colon a_1\ldots a_n\mapsto (-1)^n a_n\ldots a_1$; and
(iii) Quasi-shuffle antipode: $\hat S\colon a_1\ldots a_n\mapsto(-1)^n\sum [u_1]\,[u_2]\ldots[u_k]$, 
where for the quasi-shuffle antipode the sum is over all possible factorizations of 
$a_n\ldots a_1=u_1u_2\ldots u_k$ into non-empty subwords $u_i$ for all $k=1,\ldots,n$.
\end{definition}
\begin{remark}[Antipode and quasi-shuffle Hopf algebra]
The quasi-shuffle antipode is the quasi-shuffle convolution reciprocal
map of the identity, i.e.\/ $\id\star\hat S=\hat S\star\id=\nu$. Indeed, 
with this in hand, the quasi-shuffle algebra $\mathbb{R}\langle\mathbb{A}\rangle_*$, 
with deconcatenation $\Delta$ as coproduct, is also a Hopf algebra. 
For more details, see Hudson~\cite{Hudson_1} and Hoffman~\cite{Hof}.
\end{remark}
\begin{remark}
For the shuffle algebra where the bracket $[\,\cdot\,,\,\cdot\,]$ is trivial, 
we have $\hat{S}=S$. 
\end{remark}
Malham \& Wiese \cite{MalWie} and Ebrahimi--Fard \textit{et al.\/} \cite{EfLMMkW} considered 
drift-diffusion equations interpreted in the Stratonovich sense---which 
corresponds to the shuffle product case. The scheme of main interest therein
was the map-truncate-invert integration scheme generated from the map $f=\mathrm{sinhlog}$.
The shuffle convolution power series associated to the series 
$f(1+x)=\mathrm{sinhlog}(1+x)=x+\frac{1}{2}\sum_{k\geq2} (-1)^{k-1} x^k$ 
is expressible in the form  $F^{\scriptsize\sh}(X)=\frac{1}{2}(X-X^{\scriptsize\sh(-1)})$. 
Here $\sh$ represents the convolution product corresponding to the shuffle product. 
In the shuffle case, $\id\sh\,S=S\sh\,\id=\nu$ and we therefore have the identity 
$\mathrm{sinhlog}^{\scriptsize\sh}(\id) = \frac{1}{2}(\id - S)$, 
and similarly $\mathrm{coshlog}^{\scriptsize\sh}(\id)=\frac{1}{2}(\id + S)$. 
The efficiency results of Malham \& Wiese \cite{MalWie} 
and Ebrahimi--Fard \textit{et al.\/} \cite{EfLMMkW} 
for the sinhlog integrator rely on the identity
$\langle\!\langle\pi_{\g=n}\circ\mathrm{sinhlog}^{\scriptsize\sh}(\id),
\pi_{\g=n}\circ\mathrm{coshlog}^{\scriptsize\sh}(\id)\rangle\!\rangle=0$.
Here $\langle\!\langle\,\cdot\,,\,\cdot\,\rangle\!\rangle$ is an inner product on 
$\mathrm{End}(\mathbb{R}\langle\mathbb{A}\rangle_{\scriptsize\sh})$ 
similar to $\langle\,\cdot\,,\,\cdot\,\rangle$
but where the underlying expectation map differs 
due to the use of Stratonovich integrals. 
The above result utilizes the shuffle algebra structure 
of multiple Stratonovich integrals with respect to 
continuous semimartingale integrators. 
To obtain an appropriate extension of the sinhlog integrator 
to L\'{e}vy-driven equations, we broaden our definition 
of map-truncate-invert schemes. 
In particular, we introduce the antisymmetric sign reverse integrator 
as follows.
\begin{definition}[Antisymmetric sign reverse integrator]
The antisymmetric sign reverse integrator is the direct map-truncate-invert scheme 
associated with $\frac{1}{2}(\id-S)$, truncating according to word length.
\end{definition}
\begin{remark} 
We cannot in general give an expression for $\frac{1}{2}(\id - S)$ as a convolution power series 
in $\sJ$ in any quasi-shuffle convolution algebra with non-trivial bracket. 
To see this, compare the action of $S$ and a map 
$\alpha \sJ+\beta \sJ^{\star 2}$ on a general word of length two. We see that
$S(ab) = ba$ and the identity 
$(\alpha \sJ + \beta \sJ^{\star 2})(ab) = (\alpha+\beta)ab+\beta ba+\beta[a,b]$.
Comparing these two expressions, we see that $-\alpha=\beta=1$, 
and that $[a,b]=0$ for all $a,b$. This implies that $[\,\cdot\,,\,\cdot\,]$ must be trivial. 
\end{remark}
The following representation generates a practical implementation of the antisymmetric 
sign reverse integrator. In particular, it includes the inverse step. 
\begin{lemma}[Direct representation of the antisymmetric sign reverse integrator]\label{direct rep}
The direct antisymmetric sign reverse integrator is given by 
\begin{equation*}
\sum_{|w|\leq n} I_w(t)\tilde{V}_w  
+\sum_{|w|=n+1}I_{\mathrm{coshlog}^{\scriptsize\sh}(w)}(t)\tilde{V}_w,
\end{equation*}
or equivalently
\begin{equation*}
\sum_{|w|\leq n} I_w(t)\tilde{V}_w  
+\sum_{|w|=n+1}\Big(I_{\chl(w)}(t) 
+\tfrac12I_{(S-\hat S)(w)}(t)\Big)\tilde{V}_w.
\end{equation*}
\end{lemma}
\begin{proof}
The map $\frac{1}{2}(\id-S)$ is identified with $\mathrm{sinhlog}^{\scriptsize\sh}(\id)$. 
The shuffle product preserves the word length grading so $\pi_{\g=n+1}\circ S=S\circ\pi_{\g=n+1}$
and thus 
$\pi_{\g=n+1}\circ\mathrm{sinhlog}^{\scriptsize\sh}(\id)
=\mathrm{sinhlog}^{\scriptsize\sh}(\id)\circ\pi_{\g=n+1}$.  
Hence applying Corollary~\ref{inv corollary} with $F=\mathrm{sinhlog}^{\scriptsize\sh}$ and
using that $c_1=1$ in this case and 
the identity $\mathrm{sinhlog}^{\scriptsize\sh}(\id)+\mathrm{coshlog}^{\scriptsize\sh}(\id)=\id$, we have
\begin{align*}
\bigl(\mathrm{sinh}&\mathrm{log}^{\scriptsize\sh}\bigr)^{-1}
\circ \pi_{\g\leq n}\circ\mathrm{sinhlog}^{\scriptsize\sh}(\id)\\
=&\;\pi_{\g\leq n}+\bigl(\id-\mathrm{sinhlog}^{\scriptsize\sh}(\id)\bigr)\circ\pi_{\g=n+1}
+\mathrm{h.o.t.}\\
=&\;\pi_{\g\leq n}+\bigl(\mathrm{coshlog}^{\scriptsize\sh}(\id)\bigr)\circ\pi_{\g=n+1}
+\mathrm{h.o.t.}\\
=&\;\pi_{\g\leq n}+\bigl(\mathrm{coshlog}^{\star}(\id)\bigr)\circ\pi_{\g=n+1}
+\bigl(\mathrm{coshlog}^{\scriptsize\sh}(\id)-\mathrm{coshlog}^{\star}(\id)\bigr)\circ\pi_{\g=n+1}
+\mathrm{h.o.t.}\\
=&\;\pi_{\g\leq n}+\bigl(\mathrm{coshlog}^{\star}(\id)\bigr)\circ\pi_{\g=n+1}
+\tfrac12(S-\hat S)\circ\pi_{\g=n+1}+\mathrm{h.o.t.},
\end{align*}
using that $\mathrm{coshlog}^{\scriptsize\sh}(\id)=\frac12(\id+S)$
and $\chl=\frac12(\id+\hat S)$. Ignoring higher order terms, using the convolution algebra embedding
and applying $\mu\otimes\kappa$ then gives the desired result.
\end{proof}
\begin{remark}[Antisymmetric sign reverse integrator in practice]
The terms $\bigl(\mathrm{coshlog}^{\star}(\id)\bigr)\circ\pi_{\g=n+1}$
and $\frac12(S-\hat S)\circ\pi_{\g=n+1}$ correspond to polynomials
of words of lower word length $n$ or less. For the former term this
is straightforward following analogous arguments to those 
in Remark~\ref{mti just}---noting that the
projection operator $\pi_{\g=n+1}$ is applied before $\mathrm{coshlog}^{\star}(\id)$.
For the latter term, $(S-\hat S)\circ(a_1\ldots a_{n+1})$
consists of a sum of terms, each of which contains at least one bracket.
For example one term would be $(-1)^{n+1}[a_{n+1},a_n]a_{n-1}\ldots a_1$
which is a word of length $n$, and so forth. Hence the additional
terms $\bigl(\mathrm{coshlog}^{\star}(\id)\bigr)\circ\pi_{\g=n+1}$
and $\frac12(S-\hat S)\circ\pi_{\g=n+1}$ required for implementing the 
direct antisymmetric sign reverse integrator consist of lower
order It\^o integrals we have already simulated to construct the
stochastic Taylor integrator $\pi_{\g\leq n}$.
\end{remark}
\begin{lemma}\label{orthog lemma}
Let $\pi_{\g=n}$ be the projection according to the word length grading. 
Then for any $\mathbf{V}$, any words $u$ and $v$ and endomorphisms 
$X$ and $Y$ with $\pi=\pi_{\g=n}$, we have:
\begin{enumerate}
 \item $E\circ |S|\equiv E$;
 \item $|S|(u\ast v)\equiv \bigl(|S|(u)\bigr)\ast\bigl(|S|(v)\bigr)$;
 \item $\langle |S|\circ X,Y\rangle=\langle X,|S|\circ Y\rangle$; 
 \item $\langle X,Y\rangle=\langle|S|\circ X,|S|\circ Y\rangle$; 
 \item $\|\pi\circ S\|^2=\|\pi\circ\id\|^2$;
 \item $\bigl\langle\tfrac12(\id-S)\,,\,\tfrac12(\id+S)\bigr\rangle=0$.
\end{enumerate}
\end{lemma}
\begin{remark}\label{rmk:oddn}
An immediate consequence of item~(iii) in Lemma~\ref{orthog lemma}
is that $|S|$ is self-adjoint with respect to the inner product. Further
if we combine this result with item~(i) then for any endomorphism $Z$ we have 
$\langle \id,Z\rangle=\langle |S|,Z\rangle~\Leftrightarrow~|S|\circ Z=Z$. 
\end{remark}
\begin{proof}
We prove each result item by item. Result (i) follows from the fact that
$E(w)=0$ unless $w\in\{0\}^*$, so we have  $E(w)=E\bigl(|S|(w)\bigr)$.
From Hoffman \& Ihara~\cite[Prop. 4.2]{HofIha} we know that 
the reversal map $|S|$ is an automorphism for any quasi-shuffle algebra,
which establishes result (ii). Result (iii) 
follows by direct computation as follows. For any endomorphisms 
$X$ and $Y$, results (ii) and then (i) imply 
$E\,\bigl((|S|\circ X)(u)*Y(v)\bigr)
=E\,\Bigl(|S| \bigl(X(u)*(|S|\circ Y)(v)\bigr)\Bigr)
=E\,\bigl(X(u)*(|S|\circ Y)(v)\bigr)$.
Using the definition of the inner product establishes the result. 
Result (ii) implies that  
$E\big(|S|(u)*|S|(v)\big)=E\big(|S|(u*v)\big)=E(u*v)$, giving result (iv).
Finally using this, for any words $u,v\in\mathbb A^*$ we have
$E\big((\pi_{\g=n}\circ S)(u)*(\pi_{\g=n}\circ S)(v)\big)
=(-1)^{2n}E\big(|S|(\pi_{\g=n}\circ u)*|S|(\pi_{\g=n}\circ v)\big)
=E\bigl((\pi_{\g=n}\circ u)*(\pi_{\g=n}\circ v)\bigr)$, giving result (v).
Result (vi) follows from (v) using the bilinearity of the inner product.
\end{proof}
Finally, our main result is as follows. We consider the following perturbed
version of the antisymmetric sign reverse integrator, $\tfrac12(\id-S)+\epsilon\,Z$, 
where $Z$ is any endomorphism satisfying $|S|\circ Z=Z$ on words of
length $n+1$ and $\epsilon$ is a real-valued parameter.
\begin{theorem}[Main result: Efficiency of the antisymmetric sign reverse integrator]
\label{main result}
Consider the flowmap of a stochastic differential system 
driven by independent L\'{e}vy processes with moments of all orders. 
Assume the flowmap possesses a separated stochastic Taylor expansion, 
and the vector fields are sufficiently smooth to ensure convergence 
of the truncated stochastic Taylor integration schemes to all orders. Then:

(a) The antisymmetric sign reverse integrator at a given truncation level 
$n$ is efficient in the sense that its local leading order mean-square errors are always smaller 
than those of the truncated stochastic Taylor scheme of the same order 
according to word length grading, independent of the driving vector fields 
and of the initial conditions, i.e.\/ at leading order we have
\begin{equation*}
\bigl\|R_t^{\mathrm{ASRI}}(y_0)\bigr\|_{L^2}^2\leqslant\bigl\|R_t^{\mathrm{wl}}(y_0)\bigr\|_{L^2}^2;
\end{equation*}

(b) When $n$ is odd the antisymmetric sign reverse approximation
is optimal in the following sense. If we perturb the antisymmetric sign reverse integrator
by the perturbations described above, then the difference between 
the two quantities shown in the inequality above is minimized when the perturbation is zero.
\end{theorem}
\begin{proof}
It suffices to show the leading order remainder endomorphism 
$\pi_{\g=n+1}\circ\id$ associated with the truncated stochastic Taylor expansion 
has greater norm than the remainder endomorphism $\pi_{\g=n+1}\circ\frac12(\id-S)$ 
associated with the antisymmetric sign reverse integrator, independent of $\mathbf{V}$. 
For convenience we set $Q_\epsilon\coloneqq\frac12(\id-S)+\epsilon\,Z$.
Using the identity $\id=Q_\epsilon+\tfrac12(\id+S)-\epsilon\,Z$, 
and the results of Lemma~\ref{orthog lemma}, setting $\pi=\pi_{\g=n+1}$ we have
\begin{align*}
\|\pi\circ\id\|^2
=&\;\Bigl\langle Q_\epsilon+\tfrac12(\id+S)
-\epsilon\,Z\,,\,Q_\epsilon+\tfrac12(\id+S)-\epsilon\,Z\Bigr\rangle\\
=&\;\bigl\|Q_\epsilon\bigr\|^2+2\bigl\langle Q_\epsilon,\tfrac12(\id+S)\bigr\rangle
-2\epsilon\langle Q_\epsilon\,,\,Z\rangle+\bigl\|\tfrac12(\id+S)\bigr\|^2\\
&\;-2\epsilon\bigl\langle\tfrac12(\id+S)\,,\,Z\bigr\rangle+\epsilon^2\|Z\|^2\\
=&\;\bigl\|Q_\epsilon\bigr\|^2+\bigl\|\tfrac12(\id+S)\bigr\|^2
+2\bigl\langle\tfrac12(\id-S)\,,\,\tfrac12(\id+S)\bigr\rangle
+2\epsilon\bigl\langle Z\,,\,\tfrac12(\id+S)\bigr\rangle\\
&\;-2\epsilon\bigl\langle\tfrac12(\id-S)\,,\,Z\bigr\rangle-2\epsilon^2\|Z\|^2
-2\epsilon\bigl\langle\tfrac12(\id+S)\,,\,Z\bigr\rangle+\epsilon^2\|Z\|^2\\
=&\;\bigl\|\pi\circ Q_\epsilon\bigr\|^2+\bigl\|\pi\circ\tfrac12(\id+S)\bigr\|^2
-\epsilon\bigl\langle\pi\circ\tfrac12(\id-S)\,,\,\pi\circ Z\bigr\rangle-\epsilon^2\|\pi\circ Z\|^2.
\end{align*}
Note that in all the intermediate calculations above all the endomorphisms shown, 
$Q_\epsilon$, $\id$, $S$ and $Z$ and so forth should be preceded by `$\pi\circ$'
which we left out for clarity. First we observe that when $\epsilon=0$ we conclude
that $\bigl\|\pi\circ\tfrac12(\id-S)\bigr\|^2\leq\|\pi\circ\id\|^2$, 
giving the first result. Second we observe that when $n$ is odd
with $\pi=\pi_{\g=n+1}$, then using that 
$\langle\id,Z\rangle=\langle|S|,Z\rangle~\Leftrightarrow~|S|\circ Z=Z$
from Remark~\ref{rmk:oddn} above, the linear term in $\epsilon$ on the right
above is zero. This establishes the second result of the theorem. 
\end{proof}
\begin{remark}[Characterization of the perturbations]\label{pertchar}
A wide range of endomorhpisms $Z$ satisfy the condition
$|S|\circ Z=Z$ on words of length $n+1$. 
For example the endomorphism $\sJ^{\diamond (n+1)}$,
where `$\diamond$' is the convolutional product associated with any
quasi-shuffle product such as the one above or simply the shuffle, 
satisfies the stated condition. Another example is $\frac12(\id+|S|)$.
\end{remark}
\begin{remark}
Using Theorem~\ref{th:msvswo} on mean-square versus word length graded 
truncations, the result above implies that the antisymmetric sign reverse integrator 
is more accurate than the corresponding stochastic Taylor approximation truncated 
according to mean-square grading.
\end{remark}
\begin{remark}[Global strong mean-square order of convergence]
Using Theorem~\ref{mt conv} from \textsection~\ref{sec:glob} below, 
we conclude that the antisymmetric sign reverse integrator in Theorem~\ref{main result}
converges with global strong mean-square order of convergence $n$. 
\end{remark}
The argument underlying our optimality claim for the antisymmetric sign
reverse integrator (or the direct antisymmetric sign reverse integrator)
when $n$ is odd is as follows. Consider the case when $n=1$ corresponding to
global order $1/2$ convergence. In this case the antisymmetric sign
reverse integrator coincides with the
Castell--Gaines integrator which we already know is an efficient integrator. 
The antisymmetric sign reverse integrator 
in this instance is effectively $(\nu+\sJ)(w)\equiv\id(w)$ 
on words of length $|w|=1$ but has a remainder consisting of words $w$
of length $|w|=2$ generated by the form $(\nu+\sJ-\frac12\sJ^{\sh 2})(w)$. 
It is natural to perturb this integrator by adding the term $\epsilon Z$
with $Z=\sJ^{\sh 2}$ which satisfies the condition $|S|\circ Z=Z$
on words of length $2$---our premise is not to perturb the term $\sJ$
which would affect the role of the single letters, i.e.\/ the role of
the individual processes, at this level. Then Theorem~\ref{main result} 
shows that $\epsilon=0$ gives the optimal efficient integrator, consistent
with the fact that it is an efficient integrator. Next we consider
the case $n=2$. In this case the antisymmetric sign reverse integrator is
$(\nu+\sJ-\frac12\sJ^{\sh 2})(w)$ on words of length $|w|\leqslant2$
with a remainder consisting of words $w$
of length $|w|=3$ generated by the form 
$(\nu+\sJ-\frac12\sJ^{\sh 2}+\frac12\sJ^{\sh 3})(w)$.
Since it is natural to ask/require our integrator to retain
its efficiency properties if we truncate it to generate the
corresponding integrator at the next order down, we do not
perturb $(\nu+\sJ-\frac12\sJ^{\sh 2})$ and consider perturbing
the integrator by adding the perturbation $\epsilon Z$
with $Z=\sJ^{\sh 3}$ which satisfies the required property for
words of length $3$. In this instance Theorem~\ref{main result}
tells that the optimally efficient integrator may be realized
for a non-zero value of $\epsilon$, however for $\epsilon=0$
it \emph{is} an efficient integrator independent of the underlying
stochastic system, so we stick with that form. Continuing to 
the case $n=3$, the antisymmetric sign reverse integrator is
$(\nu+\sJ-\frac12\sJ^{\sh 2}+\frac12\sJ^{\sh 3})(w)$
on words of length $|w|\leqslant3$ with remainder consisting of
words of length $4$ generated by
$(\nu+\sJ-\frac12\sJ^{\sh 2}+\frac12\sJ^{\sh 3}-\frac12\sJ^{\sh 4})(w)$.
Following the sequence of arguments we established for the previous case
implies we consider perturbing the antisymmetric sign
reverse integrator by $\epsilon\sJ^{\sh 4}$.
Theorem~\ref{main result} implies that $\epsilon=0$ gives the 
optimal efficient integrator, and so forth.

\begin{remark}[Modified direct antisymmetric sign reverse integrator]
The optimality result concerns odd values of $n$. In the case of even $n$, there is a small
modification of the direct antisymmetric sign reverse integrator 
which performs well in practical tests. Consider the coefficient
of $\tilde{V}_{a^{n+1}}$ for some letter $a$, i.e. the operators of leading order corresponding to words made up of 
only one letter. This coefficient is $\frac{1}{2}I_{(1 + (-1)^{n+1})a^{n+1}}$, which is equal to $I_{a^{n+1}}$
in the case of odd $n$, but is zero for even $n$. For odd $n$, the term is exactly that appearing
in the stochastic Taylor expansion of higher order, whilst for even $n$, this term vanishes completely. 
It is also possible in the case of even $n$ to generate the terms $I_{a^{n+1}}$ from iterated integrals 
of lower word order. In the modified direct antisymmetric sign reverse integrator scheme for even $n$, 
we add these additional terms, i.e.\/ the integrator is given by
\begin{equation*}
\sum_{|w|\leq n} I_w(t)\tilde{V}_w  
+\sum_{\substack{|w|=n+1 \\ w\neq a^{n+1}}}I_{\mathrm{coshlog}^{\scriptsize\sh}(w)}(t)\tilde{V}_w
+ \sum_{w=a^{n+1}} I_{a^{n+1}}(t)\tilde{V}_w.
\end{equation*}
Adding the additional terms produces a better approximation as, by 
Lemma~\ref{MSG lemma}, each $w=a^{n+1}$ is orthogonal to all other words, 
so reproducing the coefficient of every such $I_w$ in the exact remainder 
improves the accuracy of the scheme. As we do not require the 
simulation of additional iterated integrals, 
we guarantee higher accuracy for minimal extra computational effort.
\end{remark}

\section{Global convergence}\label{sec:glob} 
The error estimates derived in \textsection 4 are local estimates. 
We require a means of obtaining global convergence from local error estimates. 
In the case of drift-diffusion equations, such a mechanism was 
established by Milstein \cite{Mil,MilBook}. In this section we present 
a natural generalization of Milstein's theorem to equations driven by L\'{e}vy processes.  
The proof of Milstein's theorem requires the following bounds 
and continuity results for the exact flow. The proofs are standard and 
do not rely on the analysis of the previous sections; they  
are not reproduced here, for details see Situ \cite[p.~76--78]{Sit}, 
Fujiwara \& Kunita \cite[p.~84--86]{FujKun}, Applebaum \cite[p.~332--336]{App}. 
\begin{lemma}[Continuity and bounds for flows of L\'{e}vy-driven equations]\label{flow lemmas}
Let $\varphi_{s,t}$ be the flow of a L\'{e}vy-driven equation. 
For any $\mathcal F_s$-measurable $x,y\in L^2$, 
and $s<t$ in $[0,T]$, there exists a constant $K>0$ such that: 
(i) $\|\varphi_t(y)\|^2_{L^2}\leq K(1+\|y\|^2_{L^2})$;
(ii) $\|\varphi_{s,t}(y)-y\|^2_{L^2}\leq K(t-s)(1+\|y\|_{L^2}^2)$;
(iii) $\|\varphi_t(x)-\varphi_t(y)\|^2_{L^2}\leq e^{Kt}\|x-y\|^2_{L^2}$
and (iv) $\|\varphi_{s,t}(x)-\varphi_{s,t}(y)-(x-y)\|_{L^2}^2\leq K(t-s)\|x-y\|^2_{L^2}$.
\end{lemma}
We now present the generalization of Milstein's theorem to L\'{e}vy-driven equations. 
Essentially, it states that any Markovian integration scheme for which 
the local mean-square error converges with order $p$, converges globally 
in the mean-square sense with order $p-1/2$, 
provided the expectation of the local error is at least of order $p+1/2$. 
\begin{theorem}[Generalized Milstein Theorem]\label{mt conv}
Let $\hat{\varphi}_{s,t}$ be an approximate flow defined for values $s\leq t$ 
such that $s, t\in\{h,2h,\ldots,T\}$, where $h$ is the step size of the scheme. 
Suppose that the expectation of the local error of the approximation is of order $p_1$, 
and that the local mean-square error is of order $p_2$, 
i.e.\/ for any $t<T$ there exists a constant $K$ such that
\begin{align*}
\Bigl|E\bigl(\varphi_{t,t+h}(y)-\hat{\varphi}_{t,t+h}(y)\bigr)\Bigr|&\leq K(1+|y|^2)^{1/2}h^{p_1},\\
\|\varphi_{t,t+h}(y)-\hat{\varphi}_{t,t+h}(y)\|_{L^2}&\leq K(1+|y|^2)^{1/2}h^{p_2},
\end{align*}
where $p_2\geq1/2$ and $p_1\geq p_2 + 1/2$. 
Suppose further that $\hat{\varphi}_{s,t}$ is independent of $\mathcal{F}_s$ for all $s<t$. 
Then the approximation converges globally to the exact flow 
with strong order $p_2 - 1/2$, 
i.e. there exists a constant $K$ such that for all $t\in\{h,2h,\ldots,T\}$ we have
$\|\varphi_t(y_0)-\hat{\varphi}_t(y_0)\|_{L^2}\leq K(1+\|y_0\|_{L^2}^2)^{1/2} h^{p_2-1/2}$.
\end{theorem}
The proof given in Milstein~\cite{Mil,MilBook} for equations driven 
by Wiener processes relies only on the local accuracy of the 
approximation and the continuity and growth bounds for the exact flow 
given in the preceding lemma. By Lemma~\ref{flow lemmas}, these properties 
of the exact flow hold for L\'{e}vy-driven equations, 
and hence the proof of Milstein immediately generalizes to give the above theorem. 
We remark that the assumption that $\hat{\varphi}_{s,t}$ is independent of 
$\mathcal{F}_s$ is required, as the local error estimates must hold across 
each computational interval $[s,t]$ where the expectation is taken 
with respect to $\mathcal{F}_s$. All the schemes we consider fulfil this property.

\section{Example integrators and numerical simulations}\label{sec:NS}

\begin{figure}[!h]
\includegraphics[width=0.48\textwidth]{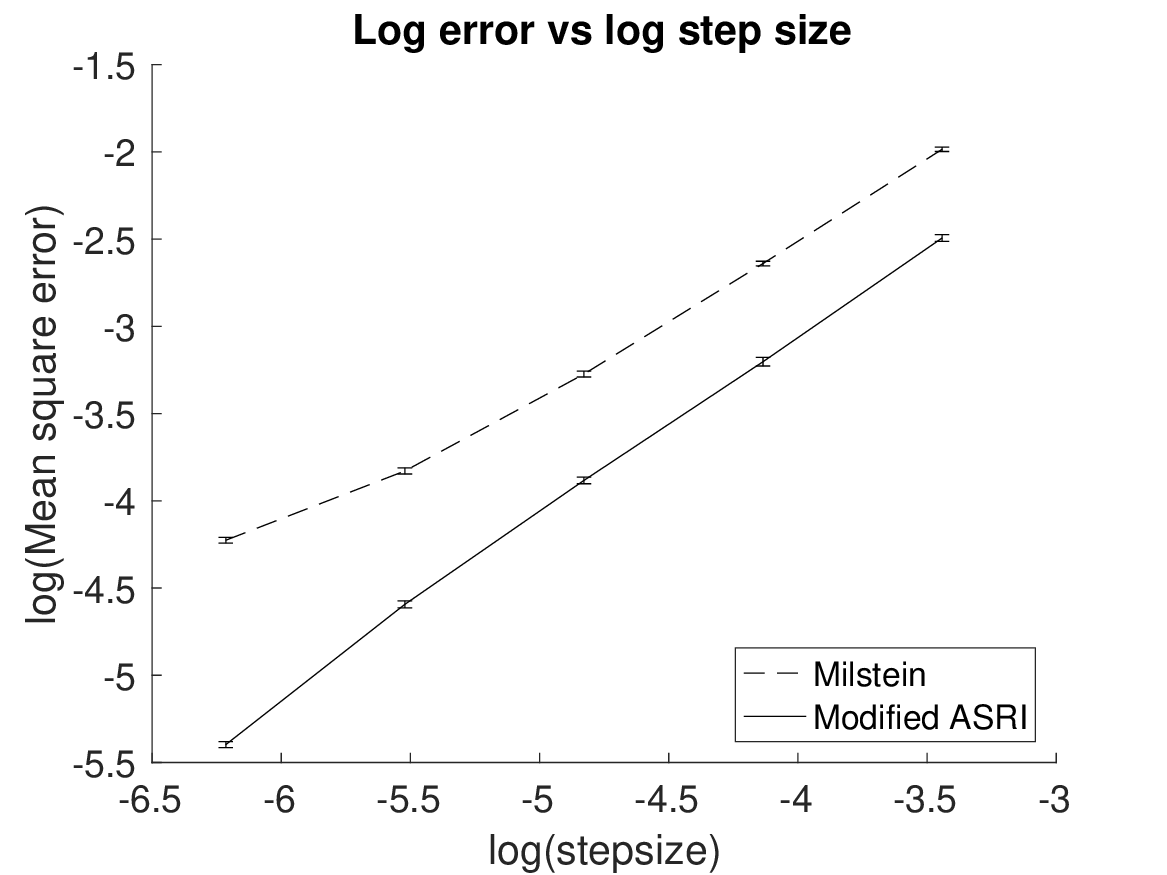}
\includegraphics[width=0.48\textwidth]{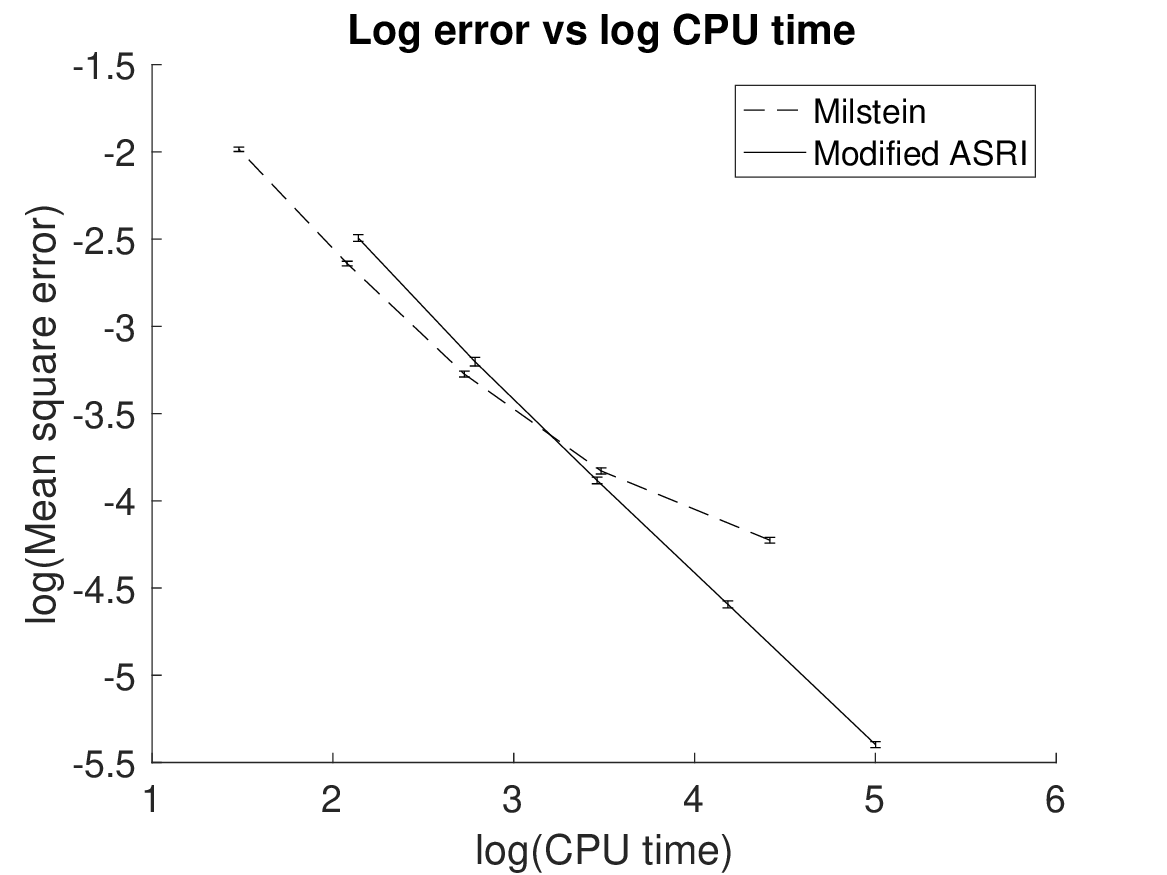}
\caption{The order 1 modified antisymmetric sign reverse integrator scheme is more accurate 
than the Milstein scheme for any step size, as shown in the left panel.
Further it is more efficient than the Milstein scheme for sufficiently small error tolerances, 
as can be seen in the right panel.}
\label{fig:main}
\end{figure}
We present some example antisymmetric sign reverse integrators (see Lemma~\ref{direct rep}) as well
as some numerical experiments which demonstrate the results proved above. 
\begin{example}
Consider the order one antisymmetric sign reverse approximation.
Since $\mathrm{coshlog}^\star(\id)=\nu+\frac12\sJ^{\star2}-\frac12\sJ^{\star3}+\cdots$
and $(S-\hat S)(a_1a_2a_3)=\frac12\bigl([a_3a_2]a_1+a_3[a_2a_1]+[a_3a_2a_1]\bigr)$
the direct antisymmetric sign reverse approximation has the form
\begin{equation*}
\sum_{|w|\leq 2} I_w \tilde{V}_w
+\!\!\!\sum_{w=a_1a_2a_3}\!\!\tfrac12\bigl(I_{a_1}I_{a_2a_3}+I_{a_1a_2}I_{a_3}-I_{a_1}I_{a_2}I_{a_3}
+I_{[a_3a_2]a_1}+I_{a_3[a_2a_1]}+I_{[a_3a_2a_1]}\bigr)\tilde{V}_w.
\end{equation*}
The terms in the second sum with $a1=a2=a3$ are zero and need not be evaluated. 
The modified antisymmetric sign reverse integrator includes the following additional terms, 
\begin{equation*}
\sum_{w=aaa} \!\frac{1}{3} \bigl(I_{aa}I_a - I_{[a,a]}I_a + I_{[a,a,a]}\bigr) \tilde{V}_w.
\end{equation*}
\end{example}

\begin{example}
We now examine the case of order one half integrators.
Here $\mathrm{coshlog}^\star(\id)=\nu+\frac12\sJ^{\star2}$ 
and $(S-\hat S)(a_1a_2) = -\frac{1}{2}[a_2a_1]$, hence the approximation takes the form
\begin{equation*}
\sum_{|w|\leq 1} I_w \tilde{V}_w 
+\!\sum_{w=a_1a_2} \!\tfrac12\bigl(I_{a_1}I_{a_2} + I_{a_2}I_{a_1} - I_{[a_2 a_1]}\bigr)\tilde{V}_w
\end{equation*}
\end{example}

Consider the equation
$\mathrm{d}y_t=A_0y_t\mathrm{d}t+A_1y_t\mathrm{d}W^1_t+V_2(y_t)\mathrm{d}W^2_t
+ A_3 y_t\mathrm{d}\tilde{N}_t$, 
where $W^1_t,W^2_t$ are Wiener processes and $\tilde{N}_t$ 
is a standard Poisson process with intensity $\lambda$. 
Here $V_2$ is the nonlinear vector field
$V_2\big(x_1,x_2,x_3,x_4\big) = \big(\sin(x_1),\cos(x_2),x_4,-\sin(x_3)\big)^T$,
whilst the constant coefficient linear vector fields are defined by the 
$\R^{4\times4}$-valued matrices $A_0$, $A_1$ and $A_2$. Their explicit form is
given in the electronic supplementary material. 
We compare the global mean square error $E(\sup_{0\leq t\leq T} |y_t - \hat{y}_t|^2)^{1/2}$, 
estimated by sampling 1000 paths, for two approximations of $\hat{y}_t$: 
the Milstein scheme and the order 1 modified antisymmetric sign reverse integrator. 
The intensity of $\tilde{N}_t$ was taken to be $\lambda=50$ and 
the initial condition was $y_0=(1,0.8,0.6,0.4)^T$. 
The code employed is available in the electronic supplementary material 
together with similar plots for the drift-diffusion case, 
linear vector fields and so forth. 
We observe in Figure~\ref{fig:main} (left panel) that the 
antisymmetric sign reverse integrator is more accurate than the Milstein scheme 
for any given step size, in accordance with the theory. 
The situation is more nuanced when the mean square error is plotted against 
the computational time as in Figure~\ref{fig:main} (right panel). 
Generically the Milstein scheme outperforms the modified 
antisymmetric sign reverse integrator for large step sizes, 
when the computational effort of the schemes is dominated by 
the evaluation of the vector fields. For smaller stepsizes, 
the computational cost of simulating iterated integrals dominates, 
upon which the modified antisymmetric sign reverse integrator outperforms the Milstein scheme.

\section*{Acknowledgement}
C.C. would like to thank Michael Tretyakov and Seva Shneer for useful comments.
We are very grateful to all the referees whose comments and suggestions significantly
helped to improve the presentation of the manuscript; in particular 
for pointing us towards reference~\cite{NewRad} on the quasi-shuffle product.

\appendix

\section{Electronic supplementary material}
We present additional details concerning the results of our manuscript. 
In Section (1) we provide an introduction to the role of quasi-shuffle algebras
and quasi-shuffle convolution algebras in the development and analysis of
strong integrators for stochastic differential equations. 
In Section (2), we relate the stochastic Taylor expansion given to that derived in 
Platen \& Bruti-Liberati, showing their equivalence. In Section (3), 
we show that the separated Taylor expansion proposed in the text 
takes a particularly simple form in the case of linear, constant coefficient equations. 
Finally, in Section (4) we present in greater detail the results of our numerical experiments, 
including more plots covering more cases, and with more comments concerning methodology. 
The MATLAB code is available as a separate file.

\subsection{SDEs and quasi-shuffle algebras}\label{sec:QSintro}
Our goal in this section is to provide an introduction to the role
of quasi-shuffle algebras as well as quasi-shuffle convolution algebras
in the development and analysis of strong integrators for stochastic
differential equations. We will assume some basic knowledge of 
stochastic differential equations and their strong/pathwise 
solution. We consider It\^o stochastic differential systems
driven by scalar stochastic processes and governed by 
vector fields of the form:
\begin{equation*}
Y_t=Y_0+\sum_{i=1}^d\int_0^t V_i(Y_\tau)\,\rd X_\tau^i,
\end{equation*}
for time interval $t\in[0,T]$ for some $T>0$. The stochastic
solution process $Y$ is assumed to be $\R^N$-valued 
for some $N\in\mathbb N$. We assume the initial data $Y_0\in\R^N$ 
is a given deterministic vector. The functions $V_i$, $i=1,\ldots,d$,
are assumed to be smooth governing vector fields and in general non-commuting;
we explain what we mean by this further below. The $X^i$, $i=1,\ldots,d$,
are scalar driving stochastic processes; their precise
characterization will be discussed presently. 
Our goals in this section are as follows, to:
\begin{enumerate}
\item Illustrate the development of the stochastic Taylor series
expansion for the solution $Y_t$ to the It\^o stochastic differential system
above. The stochastic Taylor series expansion is the basis of the 
construction of many classes of strong numerical approximation schemes;
\item Demonstrate why a precise understanding of the relationship
between the repeated integrals of the driving stochastic processes $X^i$
is important, first, in the efficient implementation of strong numerical approximation
schemes, and second, in the design of classes of numerical schemes such as those based
on the exponential Lie series. A precise understanding
of the algebraic structure generated by the product of such repeated
integrals is crucial to both of these components in the theory and implementation
of strong numerical approximation schemes; and
\item Show how to abstract the relationship between the repeated integrals,
first to the quasi-shuffle algebra, and then second, to one further level of
abstraction, the quasi-shuffle convolution algebra. We demonstrate
how this abstraction and the concepts and structures provided therein,
can be used to inform the construction of new numerical approximation schemes.    
\end{enumerate}
As we intend this to be a self-contained introduction there 
will inevitably be some repetition of the material in the 
main manuscript. We have endeavoured to achieve the three goals 
above with minimal repetition while at the same time trying
to give a broader perspective.
Our setting is a complete, filtered probability space 
$\bigl(\Omega,\mathcal F,(\mathcal F_t)_{t\geqslant0},P\bigr)$
assumed to satisfy the usual hypothesis; see Protter~\cite[p.~3]{Pro}. 
We assume the solution process $Y\in\R^N$ exists on some finite time interval. 

Our first goal in this section is to develop the stochastic Taylor series
expansion for the solution $Y_t$, without loss of generality, about $t=0$.
We illustrate how this is achieved in the context where the driving
processes $X^i$, $i=1,\ldots,d$, are independent Wiener processes.
As such the quadratic covariations $[X^i,X^j]=0$ for all $i\neq j$.
The quadratic covariation of $X^i$ with itself $[X^i,X^i]$ is known 
as its quadratic variation.
\begin{remark}
In fact our development of the stochastic Taylor series expansion
here is sufficiently general to allow for the $X^i$, $i=1,\ldots,d$,
to be continuous semimartingales. These are a generalization of 
Wiener processes. The L\'evy processes we consider in the main text,
which include jumps, are a different generalization. 
For a collection of Wiener processes all the quadratic variations
$[X^i,X^i]_t=t$ for all $i=1,\ldots,d$. However, as we see
below, our derivation of the stochastic Taylor expansion allows
for more general quadratic variation and indeed is consistent
with the derivation of the stochastic Taylor expansion for the 
case when the $X^i$, $i=1,\ldots,d$, are continuous semimartingales;
see Ebrahimi--Fard, Malham, Patras and Wiese~\cite{EfMPW2} for more 
details.
\end{remark}
The key tool for developing the stochastic Taylor series expansion for
$Y_t$ is It\^o's Lemma (chain rule). For any smooth function $f\colon\R^N\to\R^N$,
this states that if $Y_t$ satisfies the stochastic differential equation 
above, then $f(Y_t)$ satisfies (see for example Protter~\cite{Pro})
\begin{equation*}
f(Y_t)=f(Y_0)+\sum_{i=1}^d\int_0^t\bigl(V_i\cdot\pa\bigr)f(Y_\tau)\,\rd X_\tau^i
+\tfrac12\sum_{i=1}^d\int_0^t
\bigl(V_i\otimes V_i\colon\pa^2\bigr)f(Y_\tau)\,\rd[X^i,X^i]_\tau.
\end{equation*}
Let us explain carefully the notation we have used here. For any $Y\in\R^N$
we have 
\begin{equation*}
\bigl(V_i\cdot\pa\bigr)f(Y)=\sum_{j=1}^NV_i^j(Y)\pa_{Y_j}f(Y).
\end{equation*}
In other words when we consider the evolution of any function
of the solution, here $f(Y_t)$, then the corresponding first order 
terms involve the first order partial differential operators $V_i\cdot\pa$
acting on $f(Y_t)$. The noncommutativity of the vector fields above
refers to the fact that in general we assume that any two
first order partial differential operators $V_i\cdot\pa$ and $V_j\cdot\pa$
with $i\neq j$ do not commute. Hereafter we take the perspective that instead
of the vector fields $V_i(Y_t)$ being functions of $Y_t$ as they
are in the original stochastic differential equation, they are now first
order partial differential operators $V_i\cdot\pa$ acting on $f(Y_t)$.
Naturally when $f=\id$ so that $f(Y_t)=Y_t$ then 
$\bigl(V_i\cdot\pa\bigr)f(Y)=\bigl(V_i\cdot\pa\bigr)Y=V(Y)$, and the
evolution equation for $f(Y_t)$ reverts to the stochastic differential
equation for $Y_t$. The It\^o chain rule also includes second order terms as shown, 
where we have used the notation
\begin{equation*}
\bigl(V_i\otimes V_i\colon\pa^2\bigr)f(Y)
\coloneqq\sum_{j,k=1}^NV_i^j(Y)V_i^k(Y)\pa_{Y_j}\pa_{Y_k}f(Y).
\end{equation*}
The terms $[X^i,X^i]$ represent the quadratic variation of the $X^i$ for $i=1,\ldots,d$.
We now utilize the following succinct notation. We set 
\begin{equation*}
D_i\coloneqq V_i\cdot\pa,\qquad 
D_{[i,i]}\coloneqq\tfrac12 V_{i}\otimes V_{i}\colon\pa^2
\qquad\text{and}\qquad
X^{[i,i]}\coloneqq[X^{i},X^{i}].
\end{equation*}
Then the It\^o chain rule takes the form 
\begin{equation*}
f(Y_t)=f(Y_0)+\sum_{a\in\Ab}\int_0^t D_af(Y_\tau)\,\rd X_\tau^a,
\end{equation*}
where $\Ab$ denotes the alphabet of letters $\{1,\ldots,d,[1,1],\ldots,[d,d]\}$.
We use the terminology `alphabet' and `letters' as opposed to `indices'
with an eye on the algebraic structures ahead. This integral equation 
applies for any smooth function $f\colon\R^N\to\R^N$.
For example we could take $f$ to be $D_af$, and the 
relationship just above holds for the function $D_af\colon\R^N\to\R^N$.
In other words we have 
\begin{equation*}
D_af(Y_t)=D_af(Y_0)+\sum_{b\in\Ab}\int_0^t D_bD_af(Y_{\tau_2})\,\rd X_{\tau_2}^b.
\end{equation*}
Substituting this for the integrand on the right-hand side in the
equation for $f(Y_t)$, we obtain
\begin{align*}
f(Y_t)
=&\;f(Y_0)+\sum_{a\in\Ab}\int_0^t \biggl(D_{a}f(Y_0)
+\sum_{b\in\Ab}\int_0^{\tau} D_bD_af(Y_{\tau_2})\,\rd X_{\tau_2}^b
\biggr)\,\rd X_\tau^a\\
=&\;f(Y_0)+\sum_{a\in\Ab}\int_0^t \,\rd X_{\tau_1}^a\,D_{a}f(Y_0)
+\sum_{a,b\in\Ab}\int_0^t\int_0^{\tau_1}D_bD_af(Y_{\tau_2})\,\rd X_{\tau_2}^b\rd X_{\tau_1}^a.
\end{align*}
Now we replace $f$ in the equation above by $D_bD_af$ so that we have 
\begin{equation*}
D_bD_af(Y_t)=D_bD_af(Y_0)+\sum_{c\in\Ab}\int_0^t D_cD_bD_af(Y_{\tau_3})\,\rd X_{\tau_3}^c.
\end{equation*}
Substituting this for the integrand in the double integral term on
the first iteration above we find 
\begin{align*}
f(Y_t)=&\;f(Y_0)+\sum_{a\in\Ab}\int_0^t \,\rd X_{\tau_1}^a\,D_{a}f(Y_0)
+\sum_{a,b\in\Ab}\int_0^t\int_0^{\tau_1}\,\rd X_{\tau_2}^b\rd X_{\tau_1}^a\,D_bD_af(Y_0)\\
&\;+\sum_{a,b,c\in\Ab}\int_0^t\int_0^{\tau_1}\int_0^{\tau_2}
D_cD_bD_af(Y_{\tau_3})\,\rd X_{\tau_3}^c\rd X_{\tau_2}^b\rd X_{\tau_1}^a.
\end{align*}
It is now clear that we can repeat this procedure ad infinitum and the precise
form all the subsequent terms in this series expansion will take. 
Further let us introduce an even more succinct notation.  
Consider a word $w=a_1a_2\cdots a_n$ constructed from letters $a_i$, 
$i=1,\ldots,n$, from the alphabet $\Ab$. Then let us use 
$D_w\coloneqq D_{a_1}\cdots D_{a_n}$ to denote the successive 
compositon of the partial differential operators shown.
Further let us denote 
\begin{equation*}
I_w\coloneqq\int_{0\leqslant\tau_{1}\leqslant\cdots\leqslant\tau_{n}\leqslant t}
\,\rd X^{a_1}_{\tau_1}\cdots\,\rd X^{a_n}_{\tau_n}.
\end{equation*}
The iteration procedure above thus produces the solution expansion
\begin{equation*}
f(Y_t)=\sum_{w}I_wD_wf(Y_0).
\end{equation*}
Here the sum is over all words/multi-indices $w$ that can be 
constructed from the alphabet $\Ab$. This is the stochastic 
Taylor expansion for $f(Y_t)$. The stochastic Taylor expansion
for $Y_t$ itself can be recovered by setting $f=\id$. 
The stochastic Taylor expansion is the starting point 
for strong stochastic differential numerical approximation
schemes of higher order beyond the Euler--Maruyama approximation.
In the deterministic setting, and in the stochastic
setting, a very useful concept in the construction of numerical
approximation schemes is the notion of the flowmap.
\begin{definition}[Flowmap]
For any smooth function $f\colon\R^N\to\R^N$, we define the flowmap
$\varphi_t$ associated with the stochastic differential equation above
as the map prescibing the transport of the initial data $f(Y_0)$ 
to the solution $f(Y_t)$ at time $t\geqslant0$, in other words
$\varphi_t\colon f(Y_0)\mapsto f(Y_t)$. 
\end{definition}
We note the following. Naturally for the case $f=\id$ we have 
$Y_t=\varphi_t(Y_0)$. By substituting this into the definition 
of the flowmap we deduce $\varphi_t(f)=f(\varphi_t)$.
Further, from the stochastic Taylor expansion we already have an 
explicit representation for the flowmap, namely
\begin{equation*}
\varphi_t=\sum_{w}I_wD_w.
\end{equation*}
\begin{remark}[Separated stochastic Taylor expansion]
The stochastic Taylor expansion for $f(Y_t)$ above, as well as for
the flowmap $\varphi_t$, is separated. It is a sum over terms which
are the real product of the time-dependent stochastic repeated integrals 
$I_w$ and the time-independent deriative terms $D_wf(Y_0)$. This separated
form arose naturally and directly in our derivation of the 
stochastic Taylor expansion above which includes the case when
the driving processes $X^i$, $i=1,\ldots,d$, are independent
Wiener processes, as well as more generally, continuous
semimartingales. In the main text the driving stochastic processes
are independent L\'evy processes. 
The derivation of the stochastic Taylor expansion in this case
is more convoluted due to the jumps of a L\'evy process
and the analogous derivation to that above does
not produce a separated stochastic Taylor expansion directly. 
However, as shown in the main text, further Taylor expansion
of some of the governing vector fields eventually results in
a separated stochastic Taylor expansion analogous to that above.
\end{remark}

Our second goal in this section is to elucidate the relationship
between the repeated integrals $I_w$ in the stochastic Taylor expansion.
Here again for the moment, we assume that the driving stochastic
processes $X^i$, $i=1,\ldots,d$, are Wiener processes. However,
in fact, our discussion will be sufficiently general to apply to
the case when the $X^i$, $i=1,\ldots,d$, are continuous semimartingales.
We provide two demonstrations of why a precise understanding of 
these relationships is crucial to the efficient strong numerical
simulation of stochastic differential equations.
Such strong numerical approximations are constructed via finite
truncations of the stochastic Taylor expansion for $Y_t$,
or equivalently, the flowmap which is applied to $Y_0$. 
Assuming that computing the terms $D_w(Y_0)$, which involves
computing derivatives of the functions $V_i$ from the original
stochastic differential equation, is relatively straightforward, 
then the main task is to simulate the repeated integrals $I_w$. 
This is where most of the computational burden in the strong simulation 
of stochastic differential equations lies and any possible efficiencies 
should be utilized. With this in mind we notice that not all 
repeated integrals of a given order in the expansion are independent. 
For example integration parts, i.e.\/ the product rule, reveals that
\begin{equation*}
I_iI_j=I_{ij}+I_{ji}+\delta_{ij}I_{[i,j]},
\end{equation*}
for any letters $i,j\in\Ab$ and where $\delta_{ij}$ is the 
Kronecker delta function. Thus for instance if $i=j$ then
we do not need to simulate $I_{ii}$ as we can construct them from
$I_i\equiv X^i$ and $I_{[i,i]}\equiv X^{[i,i]}$. Or for example 
if $i\neq j$ and we have simulated $I_{ij}$, then we can construct 
$I_{ji}$ from $I_{ij}$ and $X^i$ and $X^j$. We start to see how
a clear understanding of the algebraic structure that underpins 
the relationships between the repeated integrals $I_w$ has 
practical impact on the numerical simualation of stochastic
differential equations. This is our first example 
demonstrative evidence. Now for our second. Other classes of
numerical methods, for example those which utilize the exponential Lie series,
can be constructed by considering functions of the flowmap.
For a given function $F\colon\mathrm{Diff}(\R^N)\to\mathrm{Diff}(\R^N)$, 
such a strong numerical simulation method would be constructed from
$\varphi_t$ as follows:
\begin{enumerate}
\item Construct a new series $\psi_t=F(\varphi_t)$; 
\item Truncate this series to produce the finite expansion $\hat\psi_t$;
\item Reconstruct an approximate flowmap as $\hat\varphi_t\coloneqq F^{-1}(\hat\psi_t)$; and
\item Use $\hat\varphi_t$ as basis of a strong numerical approximation scheme.
\end{enumerate}
If $F=\id$, the identity map, this approach corresponds to 
implementing a truncated stochastic Taylor expansion as a numerical approximation
scheme. The archetypical numerical scheme constructed using
the procedure just outlined is the case when $F=\log$. 
The series expansion $\log\varphi_t$ can be shown to be a Lie
series for quite general scenarios, including the continuous semimartingale
context, see Ebrahimi--Fard \textit{et al.\/} \cite{EfMPW2}. 
As such it is known as the exponential Lie series. 
We call the corresponding numerical methods based on this approach 
Castell--Gaines methods after their implementation in
Castell \& Gaines~\cite{CasGai96}. To construct these methods 
we actually need an explicit representation for $\log\varphi_t$. This
can be achieved as follows. We in fact compute $F(\varphi_t)$ for 
any function $F$ with the power series representation
\begin{equation*}
F(\phi)=\sum_{k\geq0}c_k\phi^k,
\end{equation*}
with coefficients $c_k\in\R$ for $k\in\mathbb N\cup\{0\}$.
By direct computation we see that
\begin{align*}
F(\varphi_t)&=\sum_{k\geq0}c_k\Biggl(\sum_{w}I_wD_w\Biggr)^k\\
&=\sum_{k\geq0}c_k\sum_{u_1,\ldots,u_k}
(I_{u_1}I_{u_2}\cdots I_{u_k})(D_{u_1}D_{u_2}\cdots D_{u_k})\\
&=\sum_{w}\Biggl(\sum_{k=1}^{|w|}c_k\sum_{u_1u_2\cdots u_k=w}I_{u_1}I_{u_2}\cdots I_{u_k}\Biggr)D_w.
\end{align*}
In this calculation the: summations $\sum_{w}$ are over all possible words $w$
that can be constructed from the alphabet $\Ab$; summation $\sum_{u_1,\ldots,u_k}$ is over all
possible words $u_1$, \ldots, $u_k$ that can be constructed from the alphabet $\Ab$;
and summation $\sum_{u_1u_2\cdots u_k=w}$ is the sum over all possible collections of $k$ words
$u_1$, \ldots, $u_k$ that can be concatenated together to make the given word $w$---all 
the words $u_1$, \ldots, $u_k$ and $w$ are again constructed from the alphabet $\Ab$.
Finally $|w|$ denotes the length of the word $w$, i.e.\/ the sum of all
the individual letters from the alphabet $\Ab=\{1,\ldots,d,[1,1],\ldots,[d,d]\}$ 
used in the word. We observe that 
an integral part of the construction above are products of repeated integrals
of the form $I_{u_1}I_{u_2}\cdots I_{u_k}$. This is our second example 
demonstrative evidence. Looking ahead to our third goal which we address presently,
we can already see the need to consider combinatorial calculations, for example
we need to consider all possible ways of splitting the word $w$ into $k$ words which
concatenate together to create $w$. Further we need to consider the algebra of 
repeated integrals, i.e.\/ we need to be able to consider linear combinations of
real products of repeated integrals $I_w$. Before we move onto these structures,
let us round off our second goal in this section with a concrete demonstration
of how, for example, the Castell--Gaines numerical method can be implemented in 
practice. For simplicity assume there are only two driving processes so $d=2$
and they are both Wiener processes so $X^1=W^1$ and $X^2=W^2$. If
we truncate the series $\psi_t=\log\varphi_t$, then across the computation interval
$[t_m,t_{m+1}]$ we have
\begin{equation*}
\hat\psi_{t_m,t_{m+1}}=\hat I_1(t_m)V_1\!\cdot\!\pa+\hat I_2(t_m)V_2\!\cdot\!\pa\\
+\tfrac12\bigl(\hat I_{12}(t_m)-\hat I_{21}(t_m)\bigr)
\bigl((V_1\!\cdot\!\pa)(V_2\!\cdot\!\pa)-(V_2\!\cdot\!\pa)(V_1\!\cdot\!\pa)\bigr)
\end{equation*}
and then computing the inverse $F^{-1}$ of this and applying it
to the data $Y_{t_m}$, we have
\begin{equation*}
Y_{t_{m+1}}\approx\exp\bigl(\hat\psi_{t_m,t_{m+1}}\bigr)(Y_{t_m}).
\end{equation*}
The truncated series $\psi_t=\log\varphi_t$ in this case is a Lie series
and the terms associated with the words $w=12$ and $w=21$ can be expressed in
the form shown involving the Lie bracket of the two vector fields. The hats
on the increments $\hat I_1(t_m)=\Delta W^1(t_m)$ and $\hat I_2(t_m)=\Delta W^2(t_m)$
and repeated integrals $\hat I_{12}(t_m)$ and $\hat I_{21}(t_m)$ indicate
realizations of these random variables---or suitable approximations thereof.
The approximate flowmap is thus 
$\hat\varphi_{t_m,t_{m+1}}=\exp\bigl(\hat\psi_{t_m,t_{m+1}}\bigr)$ and is
the basis for a strong numerical approximation scheme as outlined in (iv) above.
Finally we compute $Y_{t_{m+1}}$ in practice by using a suitably high order
ordinary differential numerical method to integrate
$u'(\tau)=\hat\psi_{t_m,t_{m+1}}(u(\tau))$,
across $\tau\in[0,1]$ with $u(0)=Y_{t_m}$, generating $u(1)\approx Y_{t_{m+1}}$.

Our third and final goal in this section is to concretely establish 
the connection between some specific Hopf combinatorial algebras,
namely those involving the quasi-shuffle and concatenation products,
and the analysis of functions of the stochastic Taylor expansion 
for the flowmap associated with a given stochastic differential 
equation. In fact we take this one abstraction step further and
establish the connection between a quasi-shuffle convolution
algebra and the analysis of the flowmap. At this stage we want
to include the possibility that the driving stochastic processes 
are independent L\'evy processes as outlined in the main text. 
We also want to assume a separated stochastic Taylor expansion 
for the flowmap as our starting point and so a few remarks 
are required to bridge the gap between our presentation hitherto 
and the full L\'evy process case. To construct the 
separated stochastic Taylor expansion for the flowmap when the
$X^i$, $i=1,\ldots,d$, are continuous semimartingales, we 
iteratively applied the It\^o chain rule which involved
the driving processes and the full set of quadratic variations
$[X^i,X^i]$, $i=1,\ldots,d$. Once we start computing the  
product of repeated integrals further quadratic variations or
power brackets such as $[X^i,[X^i,X^i]]$ will be generated.
For the case when the driving stochastic processes 
$X^i$, $i=1,\ldots,d$, are L\'evy processes, then already
at the stage of deriving a separated form for the 
stochastic Taylor expansion for the flowmap, higher nested
power brackets are generated. They are 
also generated once we start considering products of repeated integrals.
Importantly, the covariation bracket is both commutative and 
associative so any nested bracket is invariant to the order of 
the processes and bracketing therein; see 
Curry, Ebrahimi--Fard, Malham and Wiese~\cite{CEfMW} for more details.
As shown in the main text, provided we extend our alphabet
to account for these higher nested power brackets, then by 
further Taylor expansions for the vector fields involving
the jump processes, we can derive a separated 
stochastic Taylor expansion for the flowmap in the 
L\'evy process case. We thus take as our starting point 
a separated stochastic Taylor expansion for the flowmap
as indicated above, albeit with an extended alphabet as we
have outlined.

The basis for our subsequent development hereafter is 
the product rule for any two repeated integrals and
our computation for $F(\varphi_t)$ just above. We assume
without loss of generality that $X^i_0=0$ for all $i=1,\ldots,d$.
From Protter~\cite[p.~58]{Pro} and 
Curry \textit{et al.}\/ \cite[Section~4]{CEfMW}
for words $u,v$ and letters $a,b$ from the extended 
alphabet $\Ab$ we have the following product rule
\begin{equation*}
I_{ua}(t)I_{vb}(t)=\int_0^t\!I_u(\tau_{\scriptscriptstyle{-}})
I_{vb}(\tau_{\scriptscriptstyle{-}})\,\rd I_a(\tau)
+\int_0^t\!I_{ua}(\tau_{\scriptscriptstyle{-}})I_{v}(\tau_{\scriptscriptstyle{-}})\,\rd I_b(\tau)
+\int_0^t\!I_u(\tau_{\scriptscriptstyle{-}})I_{v}(\tau_{\scriptscriptstyle{-}})\,\rd [I_a,I_b](\tau).
\end{equation*}
Let us denote by $\Ab^*$ the free monoid over $\Ab$. This is the set
of all words $w=a_1a_2\cdots a_n$ that can be constructed from the 
letters $a_i\in\Ab$. We denote the empty word by $\mathbbold{1}\in\Ab$.
Let $\R\Ab$ denote the $\R$-linear span of $\Ab$,
and let $\R\langle\Ab\rangle$ denote the vector space of polynomials
in the non-commuting variables in $\Ab^*$. We assume that 
$[\,\cdot\,,\,\cdot\,]\colon\mathbb{RA}\otimes\mathbb{RA}\to\mathbb{RA}$ 
is a commutative, associative product on $\mathbb{RA}$. 
The quasi-shuffle product 
on $\mathbb{R}\langle \mathbb{A}\rangle$, which is commutative, 
is generated inductively as follows:
if $\mathbbold{1}$ is the empty word then $u*\mathbbold{1}=\mathbbold{1}*u=u$ and
\begin{equation*}
ua*vb = (u*vb)a + (ua*v)b + (u*v)[a,b],
\end{equation*}
for all words $u,v\in\mathbb{A}^*$ and letters $a,b\in\mathbb{A}$. 
Endowed with the product `$*$' just defined, 
$\mathbb{R}\langle \mathbb{A}\rangle$ is a commutative and associative
algebra known as the quasi-shuffle algebra, which we denote by 
$\R\langle\Ab\rangle_*$. The connections between $[\,\cdot\,,\,\cdot\,]$
and the quadratic covariation, and then between quasi-shuffle product
of the words $ua$ and $vb$ and the real product of the repeated integrals
$I_{ua}$ and $I_{vb}$ are immediate. Indeed as outlined in the main
text, we can define a word-to-integral map $\mu\colon w\mapsto I_w$ which,
extended linearly to $\R\langle\Ab\rangle_*$, is an algebra isomorphism. 
Here we use the convention that repeated integrals indexed by polynomials 
are defined by linearity, i.e.\/ 
$I_{k_u u + k_v v} = k_u I_{u} + k_v I_{v}$,
for any constants $k_u,k_v\in\mathbb{R}$ and words $u,v\in\mathbb{A}^*$.
Thus we have established that the algebra of repeated integrals
and the quasi-shuffle algebra have precisely the same structure.
Thus for example, any linear combination of products of repeated 
integrals has an exact corresponding representation (under $\mu^{-1}$) 
and evaluation (under $\mu$) in the quasi-shuffle algebra.

There is another natural algebra implicit in our 
stochastic Taylor expansion representation for the flowmap, 
which we assume has a separated form so that
\begin{equation*}
\varphi_t=\sum_{w\in\Ab^*}I_wD_w.
\end{equation*}
The other natural algebra is the algebra associated with the composition
of the partial differential operators $D_w$. Recall that for any word 
$w=a_1a_2\cdots a_n\in\Ab^*$, the partial differential operator $D_w$ 
denotes $D_{a_1}D_{a_2}\cdots D_{a_n}$,
where the partial differential operators are compositionally evaluated
right to left on any suitable target functions. We extend this ``product'' 
convention linearly to any linear combinations of such expressions.
At the abstract level on the vector space $\R\langle\Ab\rangle$, instead
of the quasi-shuffle product, we can define another product, namely 
the concatenation product. 
The concatenation product of two words $u,v\in\Ab^*$ results in their
concatenation $uv\in\Ab^*$ which is extended linearly to $\R\langle\Ab\rangle$.
Further, we can define a word-to-operator map 
$\kappa\colon w\mapsto D_w$ which is an algebra homomorphism. 
The domain of this homomorphism is $\R\langle\Ab\rangle$ equipped 
with the concatenation product, which we denote simply by 
$\R\langle\Ab\rangle$ hereafter. We thus have two combinatorial
algebras, the quasi-shuffle algebra $\R\langle\Ab\rangle_*$ and
the concatenation algebra $\R\langle\Ab\rangle$. We can consider 
their completed tensor product 
$\R\langle\Ab\rangle_*\overline{\otimes}\,\R\langle\Ab\rangle$.
Further we see that the flowmap is the image under $\mu\otimes\kappa$
of the element in $\R\langle\Ab\rangle_*\overline{\otimes}\,\R\langle\Ab\rangle$
given by
\begin{equation*}
\sum_{w\in\Ab^*}w\otimes w.
\end{equation*}
Note the product of any two elements $u\otimes u^\prime$
and $v\otimes v^\prime$ in the tensor algebra 
$\R\langle\Ab\rangle_*\overline{\otimes}\,\R\langle\Ab\rangle$
is given by 
$(u\otimes u^\prime)(v\otimes v^\prime)=u*u^\prime\otimes vv^\prime$; 
see for example Reutenauer~\cite{Reu}.

We now turn our attention to the computation, evaluation and  
representation of functions $F(\varphi_t)$ of the flowmap $\varphi_t$.
Assuming that $F$ has a power series representation as indicated 
above, recall our previous result from computing $F(\varphi_t)$, namely:
\begin{equation*}
F(\varphi_t)
=\sum_{w\in\Ab^*}\Biggl(\sum_{k=1}^{|w|}c_k\sum_{u_1u_2\cdots u_k=w}
I_{u_1}I_{u_2}\cdots I_{u_k}\Biggr)D_w.
\end{equation*}
We see that this is the image under $\mu\otimes\kappa$
of the element in $\R\langle\Ab\rangle_*\overline{\otimes}\,\R\langle\Ab\rangle$
given by
\begin{equation*}
F\Biggl(\sum_{w\in\Ab^*}w\otimes w\Biggr)=
\sum_{w\in\Ab^*}\Biggl(\sum_{k=1}^{|w|}c_k\sum_{u_1u_2\cdots u_k=w}u_1*u_2*\cdots *u_k\Biggr)\otimes w.
\end{equation*}
We have essentially completed our first level of abstraction, the flowmap
or any function of it can be represented in the tensor algebra 
$\R\langle\Ab\rangle_*\overline{\otimes}\,\R\langle\Ab\rangle$.
Note that when computing the function $F$ on the flowmap, we made a choice
to rearrange the sum by collecting all words $w$ on the right of the tensor 
`$\otimes$'. There is dual procedure which involves collecting all the words on 
the left instead. The interested reader should consult 
Reutenauer~\cite[Chapter~3]{Reu}, our main inspiration for the abstractions
and computations in this section, for this dual point of view.
Back to our convention above of collecting all the words on the right
and combinatorial operations thereof on the left. We see that we can write
\begin{equation*}
\sum_{w\in\Ab^*}w\otimes w=\sum_{w\in\Ab^*}\mathcal I(w)\otimes w
\qquad\text{and}\qquad
F\Biggl(\sum_{w\in\Ab^*}w\otimes w\Biggr)=
\sum_{w\in\Ab^*}\mathcal F(w)\otimes w,
\end{equation*}
where $\mathcal I$ is the identity endomorphism on $\R\langle\Ab\rangle_*$, while
also on $\R\langle\Ab\rangle_*$, $\mathcal F$ is the endomorphism 
\begin{equation*}
\mathcal F(w)=\sum_{k=1}^{|w|}c_k\sum_{u_1u_2\cdots u_k=w}u_1*u_2*\cdots *u_k.
\end{equation*}
\begin{remark}[Quasi-shuffle Hopf algebras]
We remark that: (i) Classical introductions to Hopf algebras can be found
in Radford~\cite{Ra} and Sweedler~\cite{S}; (ii) It is well-known 
in some physics communities that Hopf algebras provide efficient computational 
methods for working with groups, see for example Cartier~\cite{Car} and Manchon~\cite{Man};
and (iii) As we remark in the main manuscript, Hudson and collaborators 
studied quasi-shuffle Hopf algebras in the guise of sticky shuffle product Hopf algebras 
in the context of quantum stochastic calculus, see Hudson~\cite{Hudson_2}. 
\end{remark}

We now begin our second level of abstraction, and as such,
we have come to the point where we need to endow 
$\R\langle\Ab\rangle_*$ and $\R\langle\Ab\rangle$ with a bialgebra
structure. A bialgebra is a vector space equipped not only with a product,
but also a coproduct, with the product and coproduct obeying certain
compatability relations; see Reutenauer~\cite{Reu} for more details. 
We adhere to a minimalist briefing on this for simplicity.
Let $\la\,\cdot\,,\,\cdot\,\ra\colon\R\la\Ab\ra\otimes\R\la\Ab\ra\to\R$ 
denote the bilinear form where, for any words $u,v\in\Ab^\ast$, we set
$\langle u,v\rangle$ to be $1$ if $u=v$ and $0$ if $u\neq v$.
For this scalar product, the free monoid $\Ab^\ast$ forms
an orthonormal basis. With this in hand, 
we can further endow $\R\langle\Ab\rangle_*$ and $\R\langle\Ab\rangle$
with the following respective coproducts.
\begin{definition}[Deconcatenation and de-quasi-shuffle coproducts]
We define the deconcatenation coproduct 
$\Delta\colon\R\la\Ab\ra\to\R\la\Ab\ra\otimes\R\la\Ab\ra$
for any word $w\in\R\la\Ab\ra$ by
\begin{equation*}
\Delta(w)\coloneqq\sum_{u,v}\la uv,w\ra\,u\otimes v.
\end{equation*}
We also define the de-quasi-shuffle coproduct
$\Delta'\colon\R\la\Ab\ra\to\R\la\Ab\ra\otimes\R\la\Ab\ra$
for any word $w\in\R\la\Ab\ra$ by
\begin{equation*}
\Delta'(w)\coloneqq\sum_{u,v}\la u\,\quas\, v,w\ra\,u\otimes v.
\end{equation*}
\end{definition}
Note for example, the action of the deconcatenation coproduct on a word $w\in\Ab^*$,
is to generate a sum of all of its possible two-partitions $u\otimes v$ 
including $\mathbbold{1}\otimes w$ and $w\otimes\mathbbold{1}$.
Endowed with quasi-shuffle product and deconcatenation coproduct
$\R\langle\Ab\rangle_*$ is a bialgebra. Further, endowed with 
the concatenation product and de-quasi-shuffle coproduct
$\R\la\Ab\ra$ is also a bialgebra. An endomorphism known as
the antipode can be defined for both these bialgebras, and is 
given in Hoffman~\cite{Hof}. Equipped with antipodes, both these 
bialgebras become Hopf algebras. We can now define the convolution of any
two endomorphisms on the quasi-shuffle Hopf algebra $\R\la\Ab\ra_*$.
\begin{definition}[Convolution product]
Suppose $\mathcal G_1$ and $\mathcal G_2$ are two linear endomorphisms on the Hopf quasi-shuffle
algebra $\R\la\Ab\ra_*$. 
We define their quasi-shuffle convolution product 
$\mathcal G_1\star\mathcal G_2$ by the formula
\begin{equation*}
\mathcal G_1\star\mathcal G_2
\coloneqq *\circ(\mathcal G_1\otimes\mathcal G_2)\circ\Delta.
\end{equation*}
\end{definition}
In other words, since the deconcatenation product $\Delta$ splits any word 
$w\in\Ab^*$ in the sum of all its two-partitions $u\otimes v$, including when
$u$ or $v$ are the empty word $\mathbbold{1}$, we have
\begin{equation*}
(\mathcal G_1\star\mathcal G_2)(w)
\coloneqq \sum_{uv=w}\mathcal G_1(u)*\mathcal G_2(v).
\end{equation*}
Let us denote by $\mathrm{End}(\R\la\Ab\ra_*)$ the $\R$-module of 
linear endomorphisms of $\R\la\Ab\ra_*$.
The quasi-shuffle convolution product on $\mathrm{End}(\R\la\Ab\ra_*)$
naturally extends as follows 
\begin{equation*}
(\mathcal G_1\star\mathcal G_2\star\cdots\star\mathcal G_k)(w)
\coloneqq\sum_{u_1u_2\cdots u_k=w}\mathcal G_1(u_1)*\mathcal G_2(u_2)*\cdots *\mathcal G_k(u_k),
\end{equation*}
for any $\mathcal G_i\in\mathrm{End}(\R\la\Ab\ra_*)$, $i=1,\ldots,k$.
By convention if $k>|w|$ then we set the convolution product to zero.
Now recall our endomorphism $\mathcal F\in\mathrm{End}(\R\la\Ab\ra_*)$
generated when we computed the function $F$ of the flowmap. We can
now express $\mathcal F$ as follows,
\begin{align*}
\mathcal F(w)&=\sum_{k=1}^{|w|}c_k\sum_{u_1u_2\cdots u_k=w}u_1*u_2*\cdots *u_k\\
&=\sum_{k\geqslant 1}c_k\mathcal I^{\star k}(w).
\end{align*}
In other words we can identify the flow map with the identity endomorphism $\mathcal I$
as we have already indicated, and we can identify the function $F$ of the flowmap with
the endomorphism
\begin{equation*}
\mathcal F^\star(\mathcal I)=\sum_{k\geqslant 1}c_k\mathcal I^{\star k}.
\end{equation*}
We have now essentially completed our second level of abstraction. The idea now
is as follows. We assume a given fixed stochastic differential system, 
i.e.\/ one for which the driving stochastic processes and governing vector fields
are fixed. The extended alphabet and words are thus fixed. Then different 
numerical integration schemes based on the map-truncate-invert approach, 
are distinguished by their actions on those words represented by the corresponding
endomorphisms acting on them. On further piece in the puzzle is required,
a measure to compare the different remainders generated by different endomorhpisms
$\mathcal F^\star(\mathcal I)$. This is provided by constructing an inner
product on $\mathrm{End}(\R\la\Ab\ra_*)$ via a corresponding expectation
map which associates the correct expectation with any word $w$,
i.e.\/ repeated integral $I_w$, constructed from the extended alphabet.

\subsection{Platen and Bruti--Liberati form}\label{app:STequivform}
We show that the stochastic Taylor expansion derived in 
Platen \& Bruti-Liberati \cite{PlaBl} is an equivalent though 
different representation of the stochastic Taylor expansion
we give in Theorem 2.1 of the manuscript.
The equivalence of the expansions is seen from the identity
\begin{equation*}
\int_0^t \int_{\mathbb{R}^{\ell-d}} (\tilde{V}_{-1}\circ f)(y_{s_-},v) 
\,\overbar{Q}(\mathrm{d}v,\mathrm{d}s) 
= \sum_{i=d+1}^\ell \int_0^t \int_{\mathbb{R}} (\tilde{V}_i\circ f)(y_{s_-},v) 
\,\overbar{Q}^i(\mathrm{d}v,\mathrm{d}s).
\end{equation*} 
The expression on the left is the encoding employed by Platen and Bruti-Liberati of 
the jump terms in the stochastic differential equation while our encoding is
that on the right. The equivalence is explained as follows. On the left above
$\overbar{Q}$ is a compensated Poisson random measure on 
$\mathbb{R}^\ell \times \mathbb{R}_+$. 
The L\'{e}vy-driven equation may be written in this form, 
where $V_{-1}(x,v) = \sum_{i=1}^{\ell-d} v_i V_{i+d}(x)$ 
and $\overbar{Q}$ is the compensation of the Poisson random measure $Q$ 
defined such that 
$Q(B,(a,b]) = \#\{\Delta J^1_s\in B_1, \ldots, \Delta J^\ell_s 
\in B_\ell\, \colon\;\; s\in (a,b]\}$ 
for a Borel set $B = B_1\times\cdots\times B_\ell\subset\mathbb{R}^\ell$ 
bounded away from the origin, i.e.\/ $0\notin \bar{B}$. 
As independent L\'{e}vy processes almost surely never jump simultaneously 
(see Cont \& Tankov \cite[Theorem 5.3]{ConTan}), the measure
$Q$ is concentrated on sets of the form 
$0\times\cdots\times 0\times B_i\times 0\times\cdots\times 0$ 
with intensity measure 
$\meas(0\times\cdots\times 0\times B_i\times 0\times\cdots\times 0)\mathrm{d}t
=\meas^i(B_i)\mathrm{d}t$. 
The identity above thus follows.
The stochastic Taylor expansion given in Platen \& Bruti-Liberati \cite{PlaBl} 
is of the same form as the expansion we have given, but where the 
alphabet is instead $\{-1,0,\ldots,d\}$, 
and the operator associated to the letter $-1$ is the multi-dimensional shift 
$\tilde{V}_{-1}\colon f(y)\mapsto f\bigl(y+V_{-1}(y,v)\bigr)-f(y)$.

\subsection{Linear vector fields and linear diffeomorphisms}\label{app:STlinearform}
The separated stochastic Taylor expansion has a simple form when the governing vector
fields are linear with constant coefficients and we consider the action of the flowmap 
on homogeneous linear diffeomorphisms, say $f(y)=Fy$, where $F=[f_{ij}]$ is 
an $N\times N$ matrix. The identity map is a special case of such a linear
diffeomorphism which generates the solution $y_t$ directly.
In this case, the operators in the expansion are given by matrix multiplication. 
We write $V_i(y) = A^i y$, where $A^i=[a^i_{jk}]$ are constant $N\times N$ matrices 
and consider the action of the flowmap on linear functions.
First, for $i=1,\ldots,d$ by direct computation we find $\tilde{V}_i\circ f(y)=FA^iy$.
Moreover, for $i=d+1,\ldots,\ell$ the higher order differential operators 
$\tilde{V}_{i^{(m)}}$ with $m\geq 2$ vanish due to the linearity, and we obtain 
$(\tilde{V}_{i}\circ f)(y,v)=v(V_i\cdot\nabla)f(y)=vFA^iy$.
The above relation shows in addition that the term in 
$\tilde{V}_0\circ f$ involving an integral over the jump sizes vanishes. 
As the functions we act on are linear, the second order terms of $\tilde{V}_0$ also vanish. 
We therefore have $\tilde{V}_0\circ f(y) = f(V_0(y)) = FA^0 y$. 
The operators $\tilde{V}_i$ act by matrix multiplication. 
It follows that the operators $\tilde{V}_w$ act on linear functions 
by matrix multiplication in the reverse order,
$\tilde{V}_{a_1\ldots a_k}\circ f(y) = F A^{a_k}\cdots A^{a_1}y$.

\subsection{Numerical experiments}

Our numerical investigations principally concern the equation
$\mathrm{d}y_t=A_0y_t\mathrm{d}t+A_1y_t\mathrm{d}W^1_t+V_2(y_t)\mathrm{d}W^2_t
+ A_3 y_t\mathrm{d}\tilde{N}_t$, 
where $W^1_t,W^2_t$ are Wiener processes and $\tilde{N}_t$ is a standard 
Poisson process with intensity $\lambda$. Here $V_2$ is the nonlinear vector field
\begin{equation*}
V_2\big(x_1,x_2,x_3,x_4\big) = \big(\sin(x_1),\cos(x_2),x_4,-\sin(x_3)\big)^T,
\end{equation*}
whilst the constant coefficient linear vector fields are defined by the following matrices
\begin{equation*}
 A_0=\begin{pmatrix} 
 0.314724 & 0.132359 & 0.457507 & 0.457167 \\
 0.405792 & -0.402460 & 0.464889 & -0.014624 \\
 -0.373013 & -0.221502 & -0.342387 & 0.300280 \\
 0.413376 & 0.046882 & 0.470593 & -0.358114
 \end{pmatrix}
 \end{equation*}
 \begin{equation*}
 A_1=\begin{pmatrix} 
 -0.078239 & 0.155741 & 0.178735 & 0.155478 \\
 0.415736 & -0.464288 & 0.257740 & -0.328813 \\
 0.292207 & 0.349129 & 0.243132 & 0.206046 \\
 0.459492 & 0.433993 & -0.107773 & -0.468167
 \end{pmatrix}
\end{equation*}
\begin{equation*}
 A_3=\begin{pmatrix} 
 -0.223077 & 0.194829 & -0.061256 & -0.313127 \\
 -0.453829 & -0.182901 & -0.118442 & -0.010236 \\
 -0.402868 & 0.450222 & 0.265517 & -0.054414 \\
 0.323458 & -0.465554 & 0.295200 & 0.146313
 \end{pmatrix}.
\end{equation*}
We also considered the special cases where: (i) $V_2$ was set to zero, 
so we have only linear coefficients (`Linear jump diffusion'), and 
(ii) $\tilde{N}_t$ was set to zero, so there are only continuous driving processes 
(`Nonlinear diffusion'). We compare the global mean square error 
$E(\sup_{0\leq t\leq T} |y_t - \hat{y}_t|^2)^{1/2}$, estimated by sampling 1000 paths, 
for two approximations of $\hat{y}_t$: the Milstein scheme and the order 1 modified 
antisymmetric sign reverse integrator (mASRI). The intensity of $\tilde{N}_t$ was taken 
to be $\lambda=50$ and the initial condition was $y_0=(1,0.8,0.6,0.4)^T$.

We briefly describe the methods used to generate the iterated integrals. 
In the absence of jumps, the integrals $\int_t^{t+h} W^1_s dW^2_s$ and 
$\int_t^{t+h} W^2_s dW^1_s$ were simulated using a truncated Fourier transform, 
as detailed in Kloeden and Platen \cite{KloPla}. Where one or more jumps 
occur in a timestep, the situation changes. To simulate the integrals 
$\int_t^{t+h} W^i_s dN_s$ and $\int_t^{t+h} N_s dW^i_s$, we simulate the 
$W^i_t$ at each jump time, which allows for exact computation of the 
aforementioned integrals. The integrals $\int_t^{t+h} W^1_s dW^2_s$ and 
$\int_t^{t+h} W^2_s dW^1_s$ are not independent of the values of $W^i_t$ 
at the jump times, so the truncated Fourier method becomes impractical. 
We therefore simulate the $W^i_t$ on a finer grid, incorporating the jump times, 
and use the trapezoidal approximation of the integrals, 
see Milstein and Tretyakov \cite{MilTre}. There are two relevant parameters 
pertaining to the simulation of the iterated Wiener integrals: 
the number of terms $p$ retained in the Fourier series, and the number of 
points $M$ used in the trapezoidal approximation. To obtain comparable 
accuracy for the two methods we take $M=5(p+1)$. For integrators of order 1, 
$p$ should scale like $1/h$, where $h$ is the timestep, 
see Milstein and Tretyakov \cite{MilTre}.

As the equation has no explicit solution, we estimate the mean square error 
by comparing the sample paths simulated with those obtained by a 
numerical scheme for the same simulated driving processes and iterated integrals, 
but employing a smaller time step. In practice, we simulate only at the finest scale, 
and then extrapolate the values to the coarser scale. To obtain estimates 
of the computational time, we run and time the code to generate the 
iterated integrals at the coarser scales remembering to scale $p$ with $1/h$. 
It is important to note that the mean square estimates obtained by this method does not include 
the simulation error for the iterated integrals.
We therefore experimented with different values of $p$ on the coarsest time step
(in the figures below the values of $p$ quoted correspond to the value for $p$ used
at the most coarse scale; it is increased in proportion with $1/h$ for smaller stepsizes). 
The relative size of the simulation time and the evaluation time is key 
to the effectiveness of the (m)ASRI scheme. As the simulation time increases with $p$, 
and $p$ scales like $1/h$, it can be shown that below a critical step size, 
the simulation time dominates the evaluation time, see 
Lord, Malham and Wiese \cite{LMW}. This step size depends on the value of 
$p$ at the coarse scale, and is lower for the (m)ASRI scheme than the 
Milstein scheme, due to the increased evaluation cost of the former.

We show the results of our simulations. In accordance with the theory, 
in all cases the mASRI scheme has lower mean square error for a given 
timestep than the Milstein scheme. We then plot the mean square error 
against the CPU time taken in each case, together with the plot of 
simulation time versus evaluation time. The time steps employed are 
typically such that these values are comparable, although in practice 
smaller time steps would frequently be used. In all cases, 
the evaluation time of the mASRI scheme was greater than the simulation time, 
whilst the simulation time often overtook the evaluation time of the 
Milstein scheme as the timestep became smaller. Below its critical step size, 
the mASRI scheme will always be more efficient than the Milstein scheme. 
Generically, it tends to outperform the Milstein scheme also for step sizes 
below the critical step size of the Milstein scheme, but above the mASRI critical size. 
If less accuracy is required, the Milstein scheme may be preferable, 
although this is not always the case.

\begin{figure}[!h]
\centering
\includegraphics[width=0.7\textwidth]{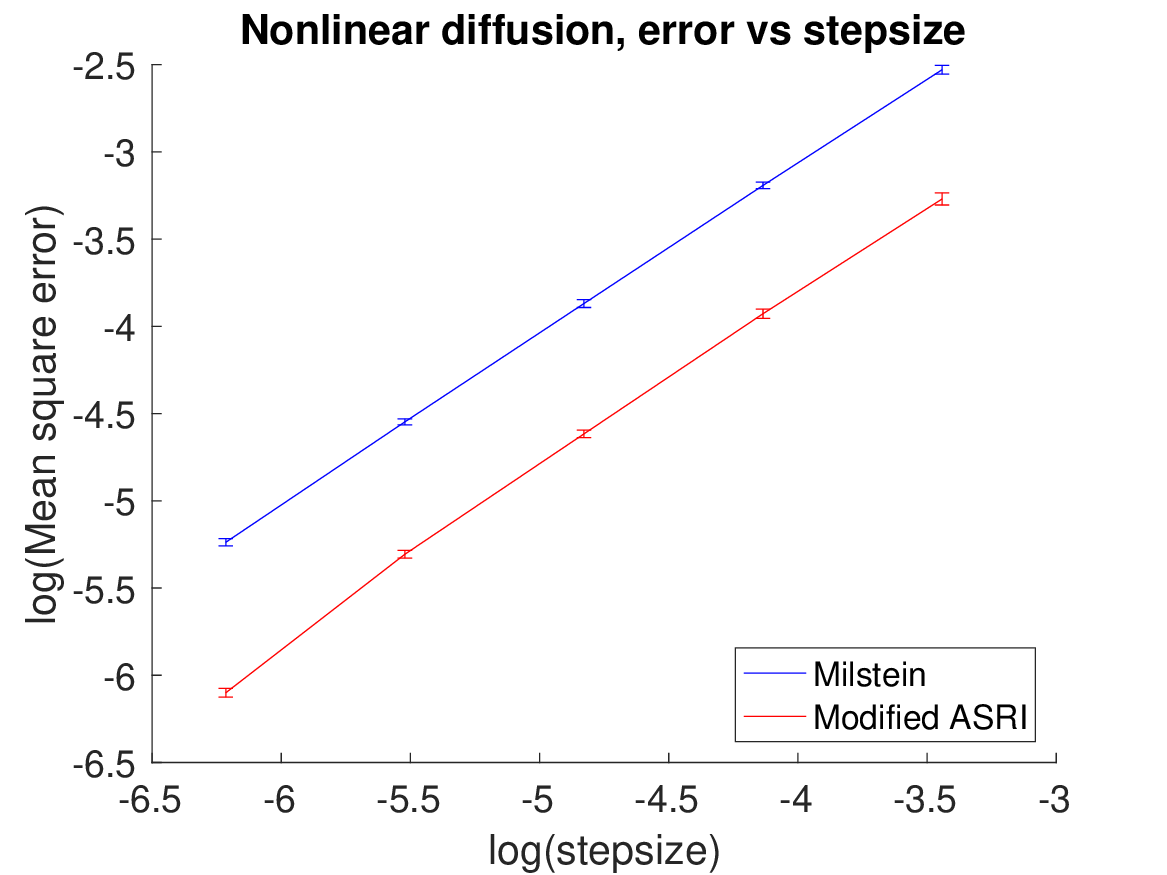}
\end{figure}

\begin{figure}[!h]
\centering
\includegraphics[width=0.7\textwidth]{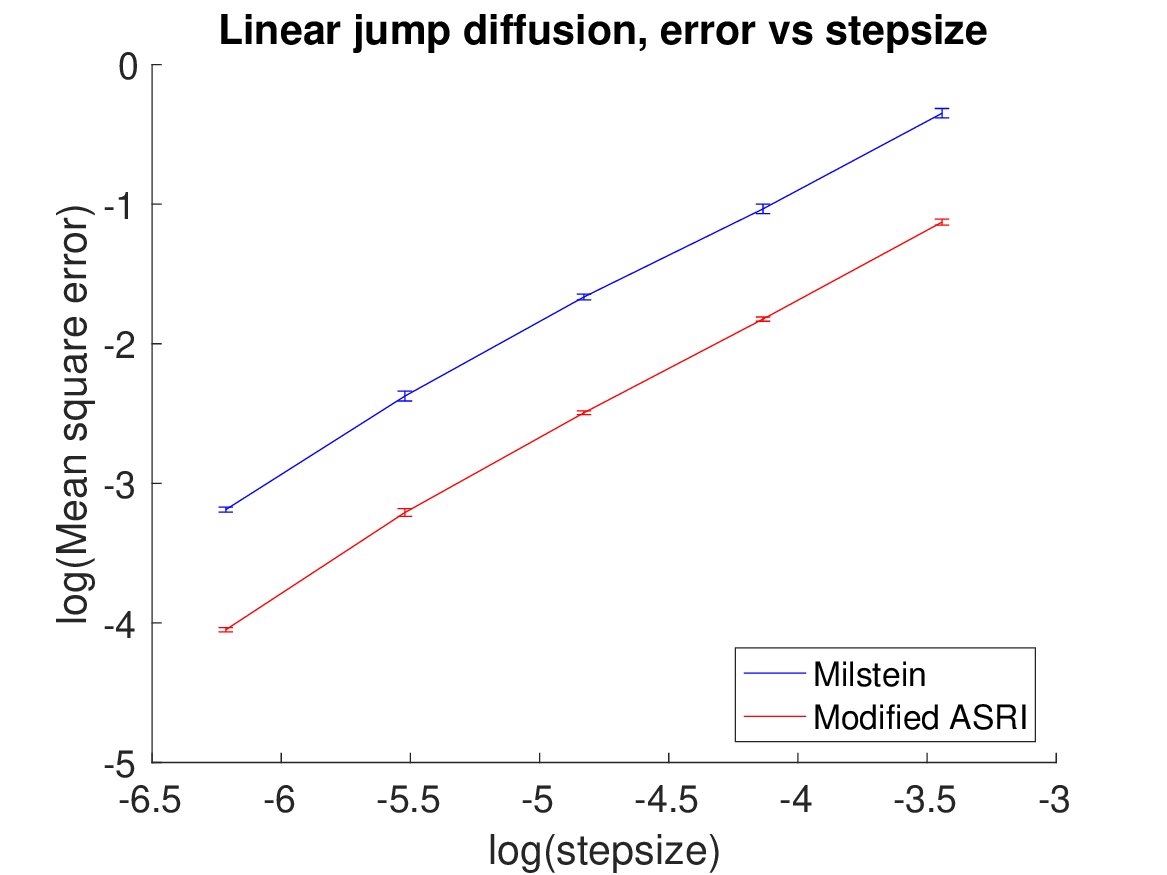}
\end{figure}

\begin{figure}[!h]
\centering
\includegraphics[width=0.7\textwidth]{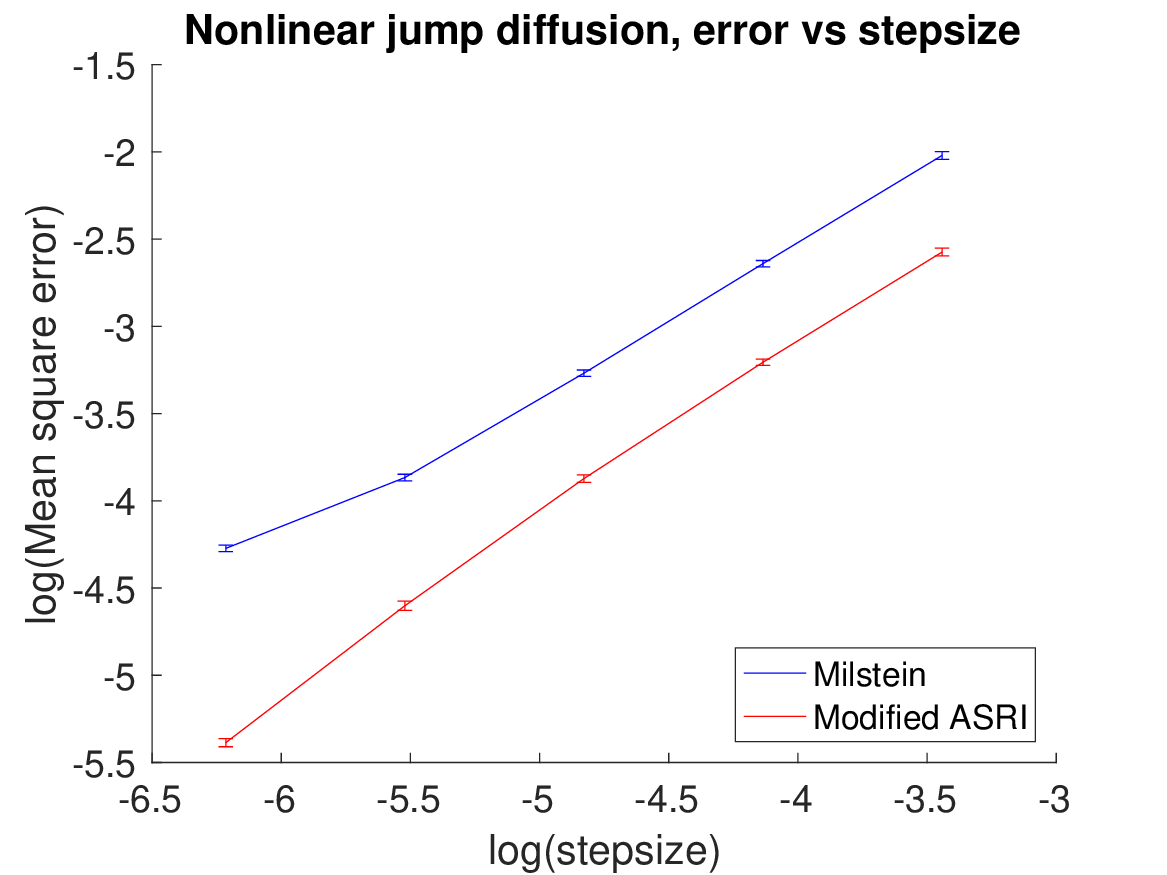}
\end{figure}

\begin{figure}[ht]\label{fig:1}
\includegraphics[width=0.48\textwidth]{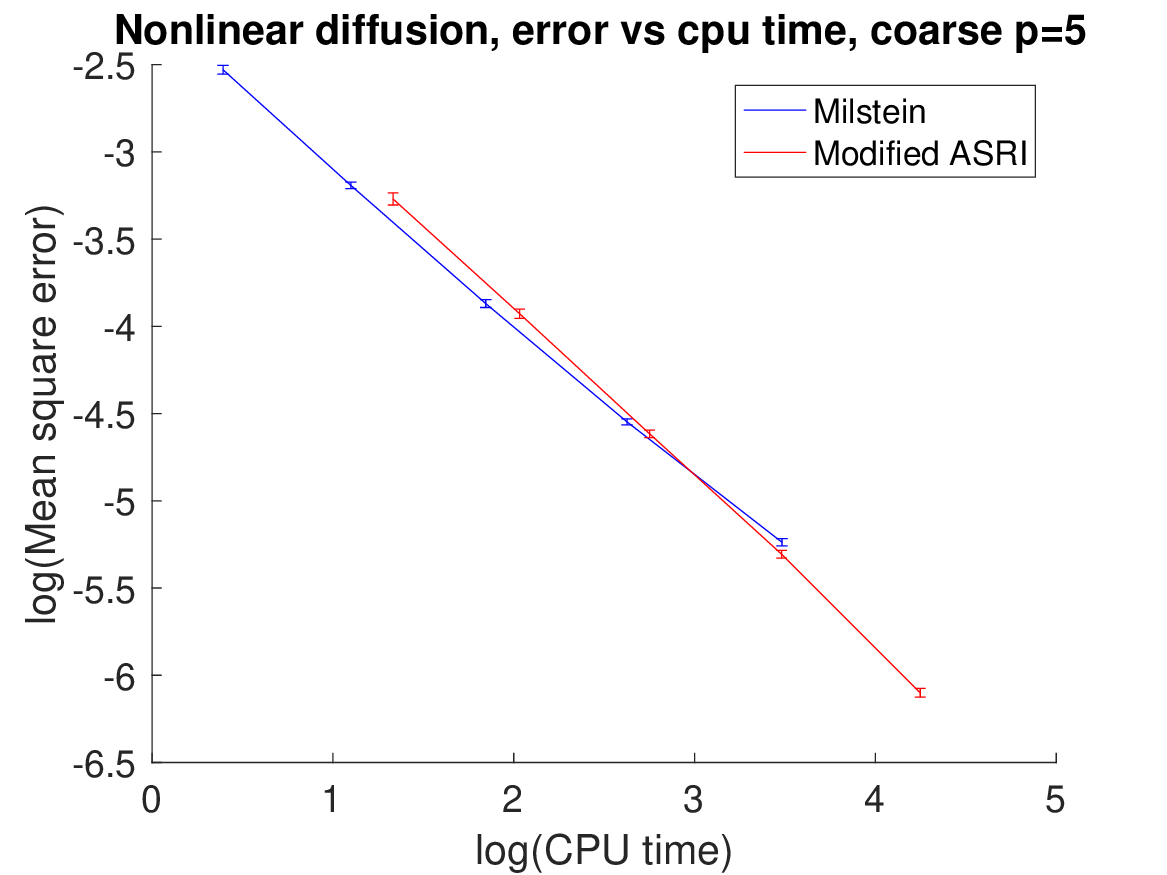}
\includegraphics[width=0.48\textwidth]{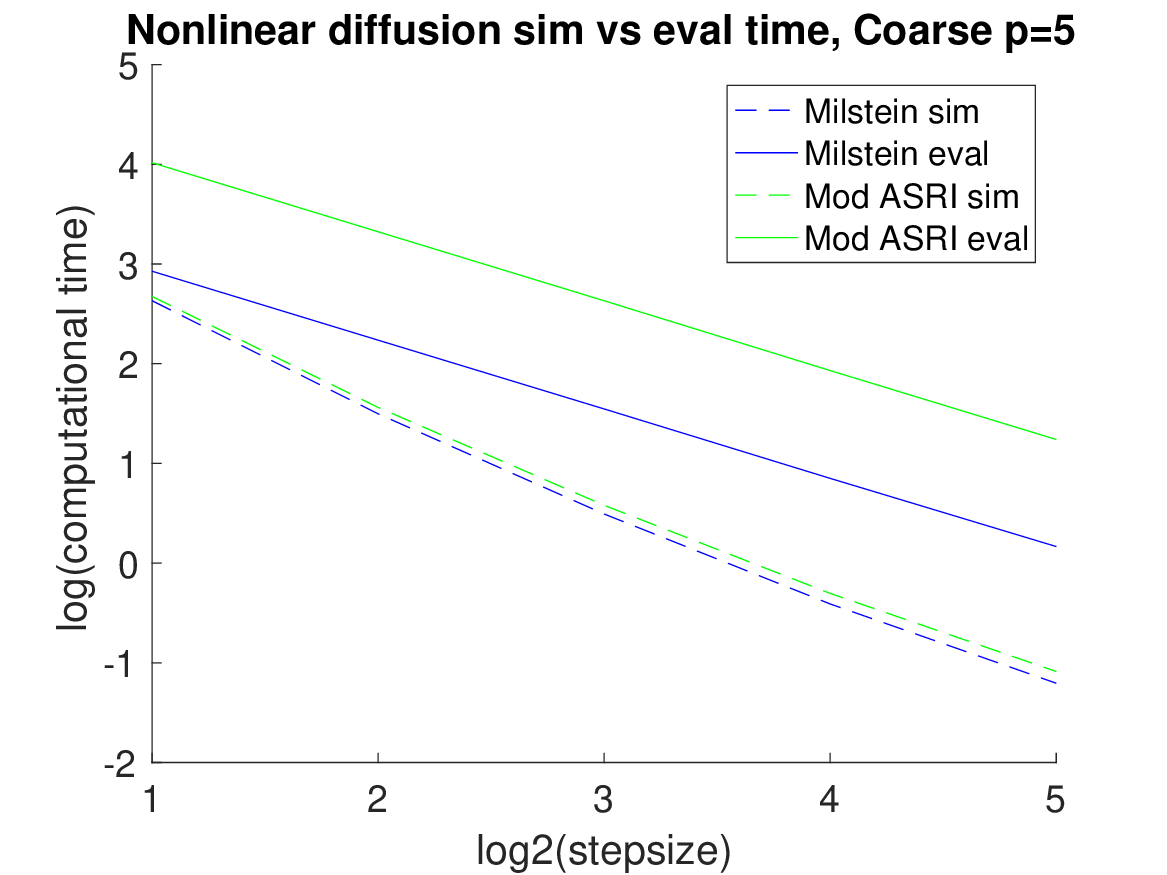}
\end{figure}

\begin{figure}[ht]\label{fig:2}
\includegraphics[width=0.48\textwidth]{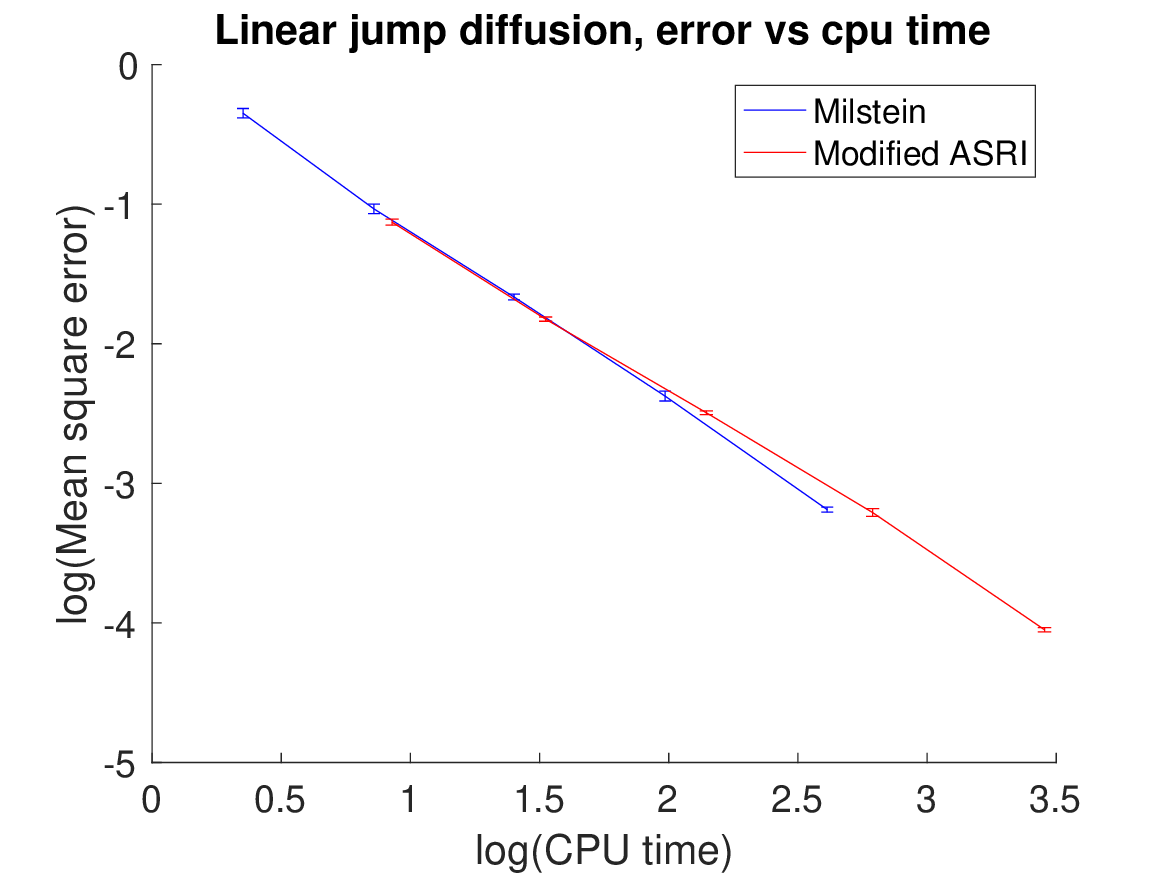}
\includegraphics[width=0.48\textwidth]{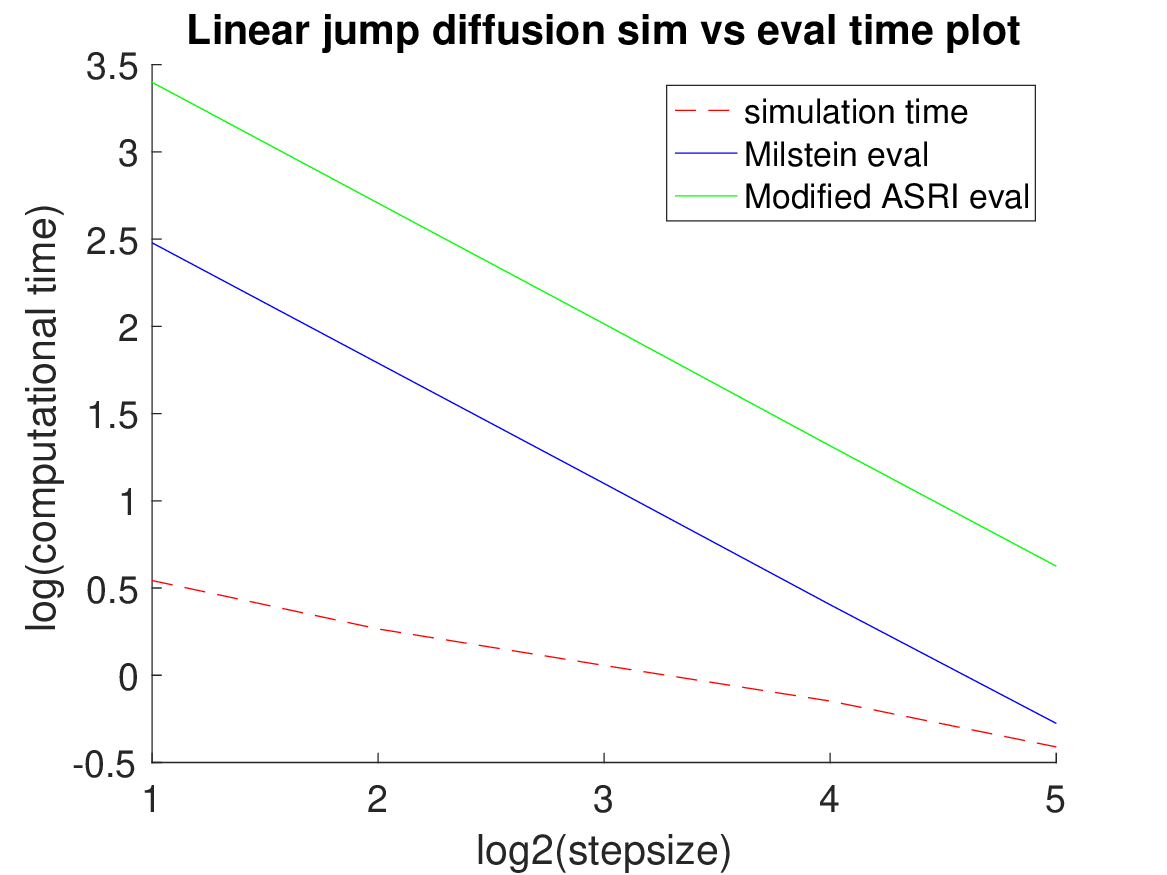}
\end{figure}

\begin{figure}[ht]\label{fig:3}
\includegraphics[width=0.48\textwidth]{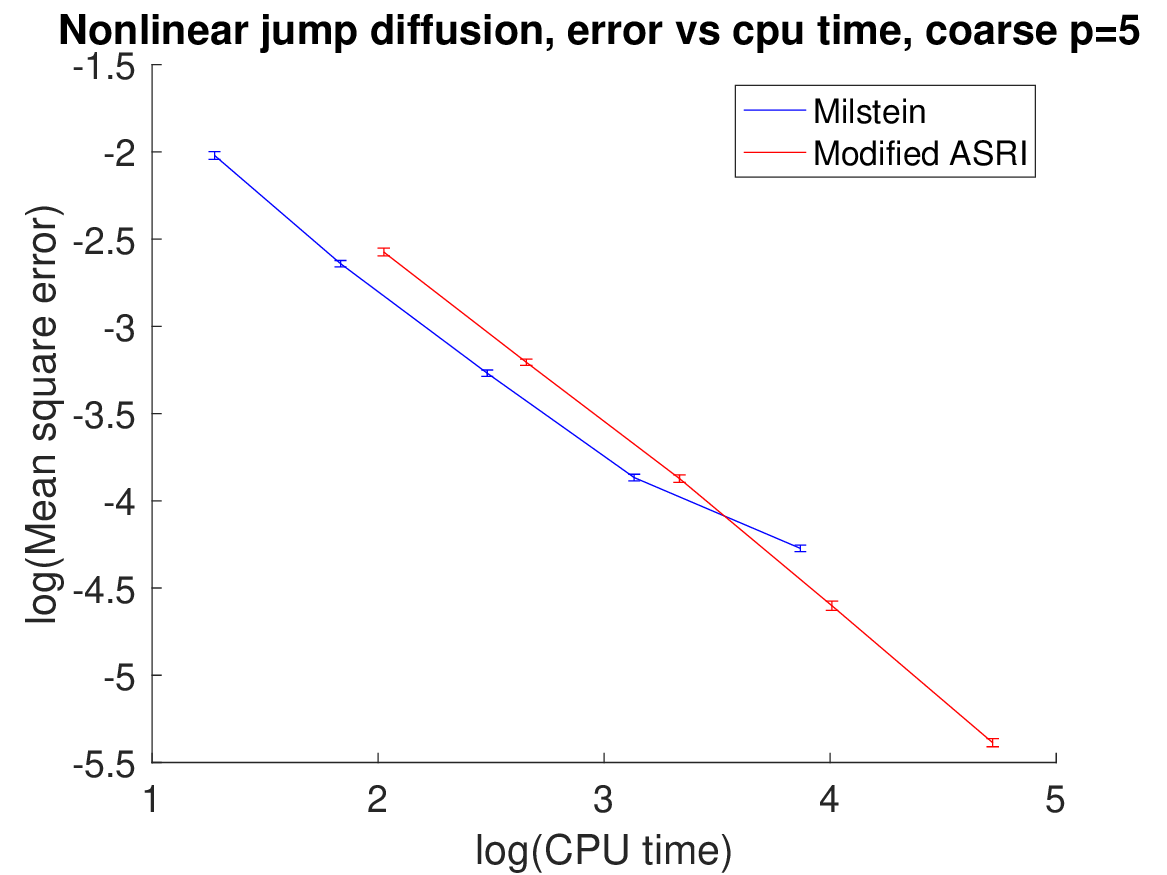}
\includegraphics[width=0.48\textwidth]{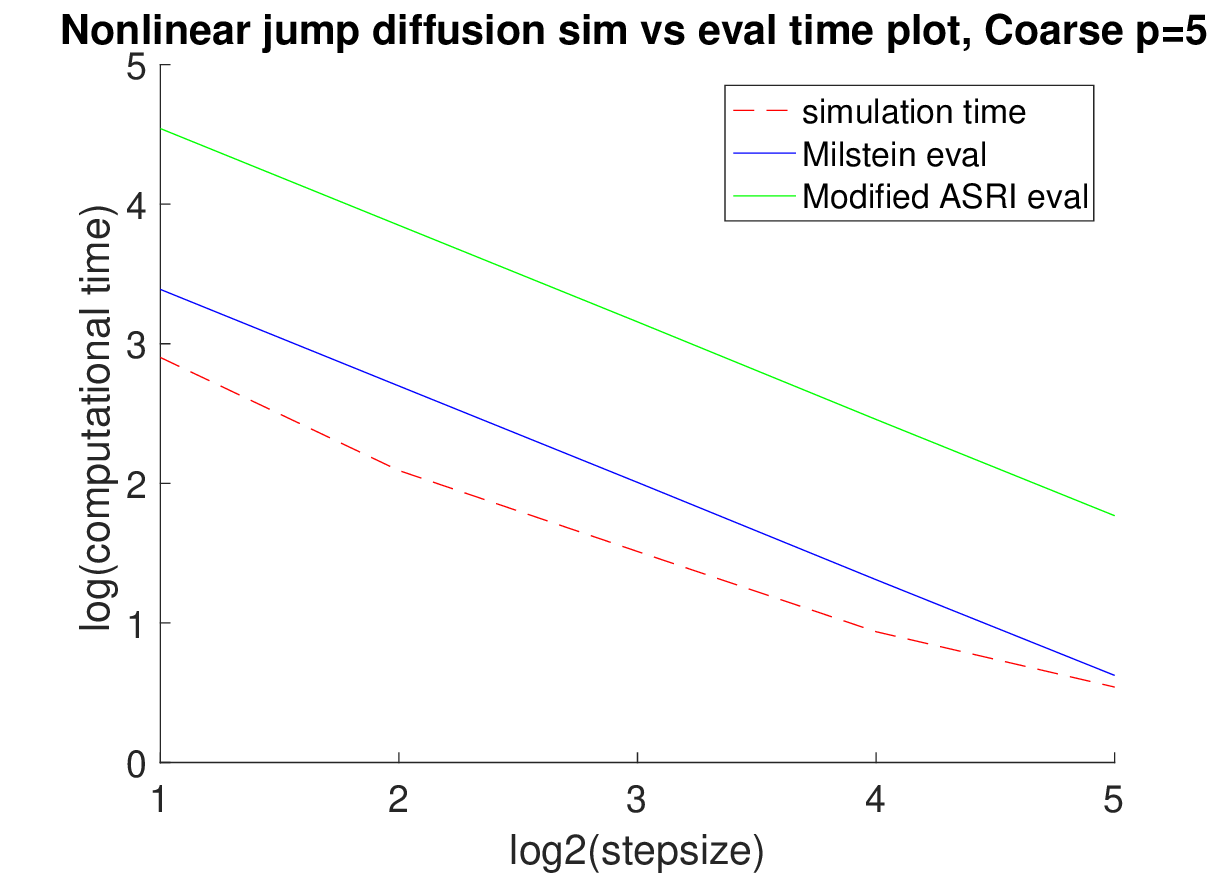}
\end{figure}

\begin{figure}[ht]\label{fig:4}
\includegraphics[width=0.48\textwidth]{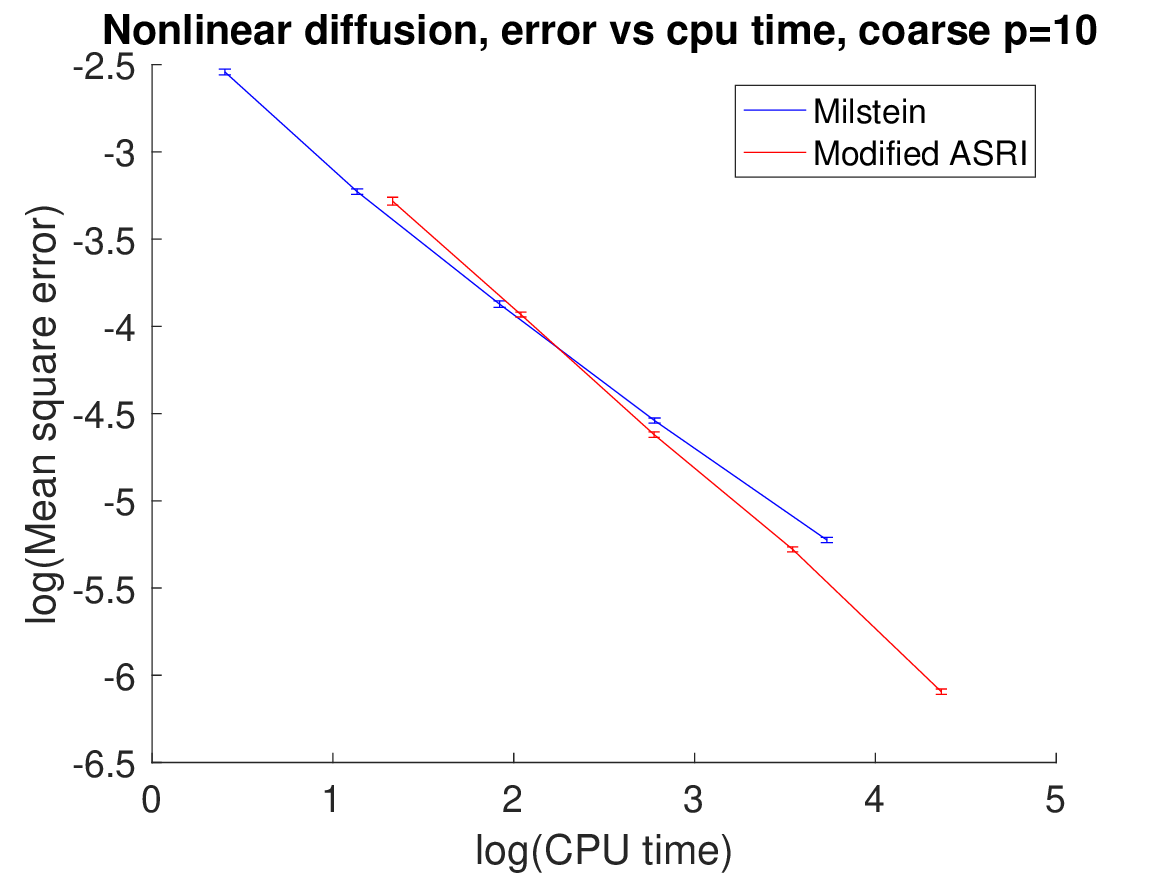}
\includegraphics[width=0.48\textwidth]{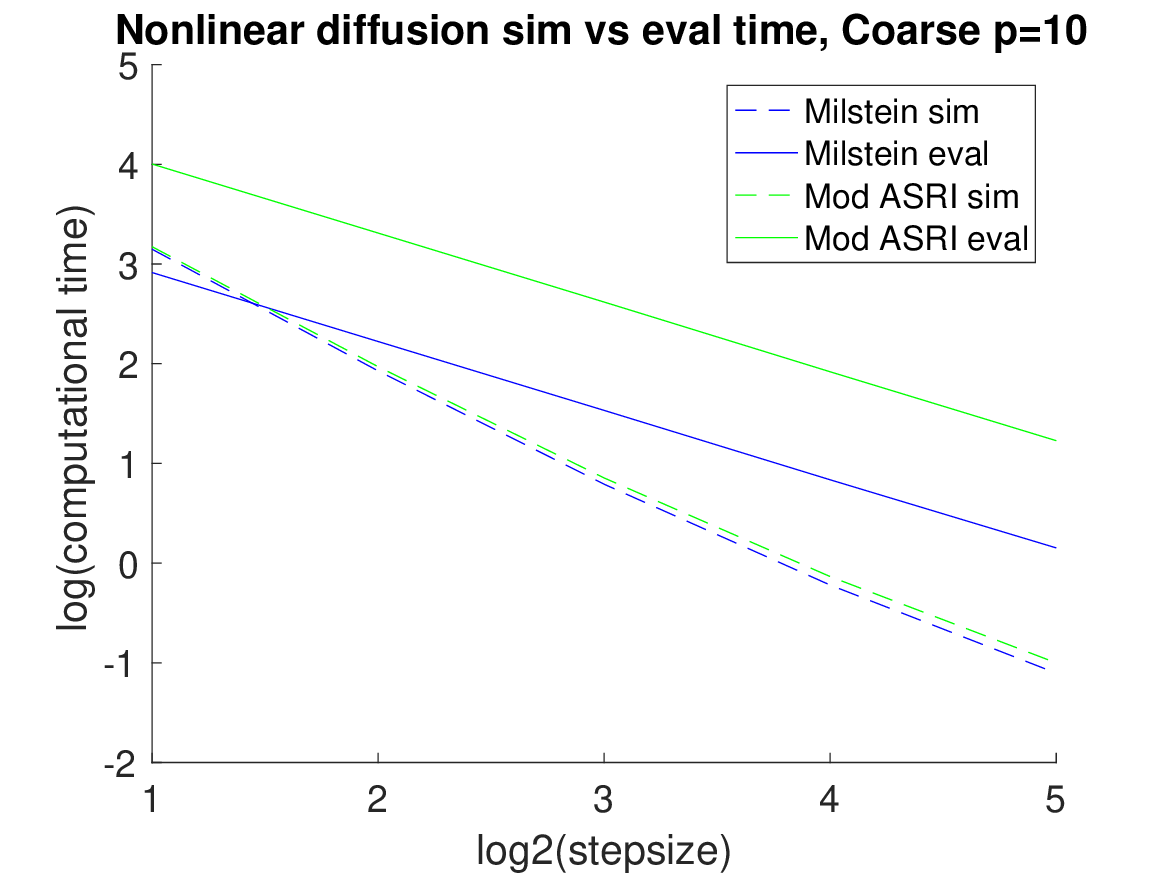}
\end{figure}

\begin{figure}[ht]\label{fig:5}
\includegraphics[width=0.48\textwidth]{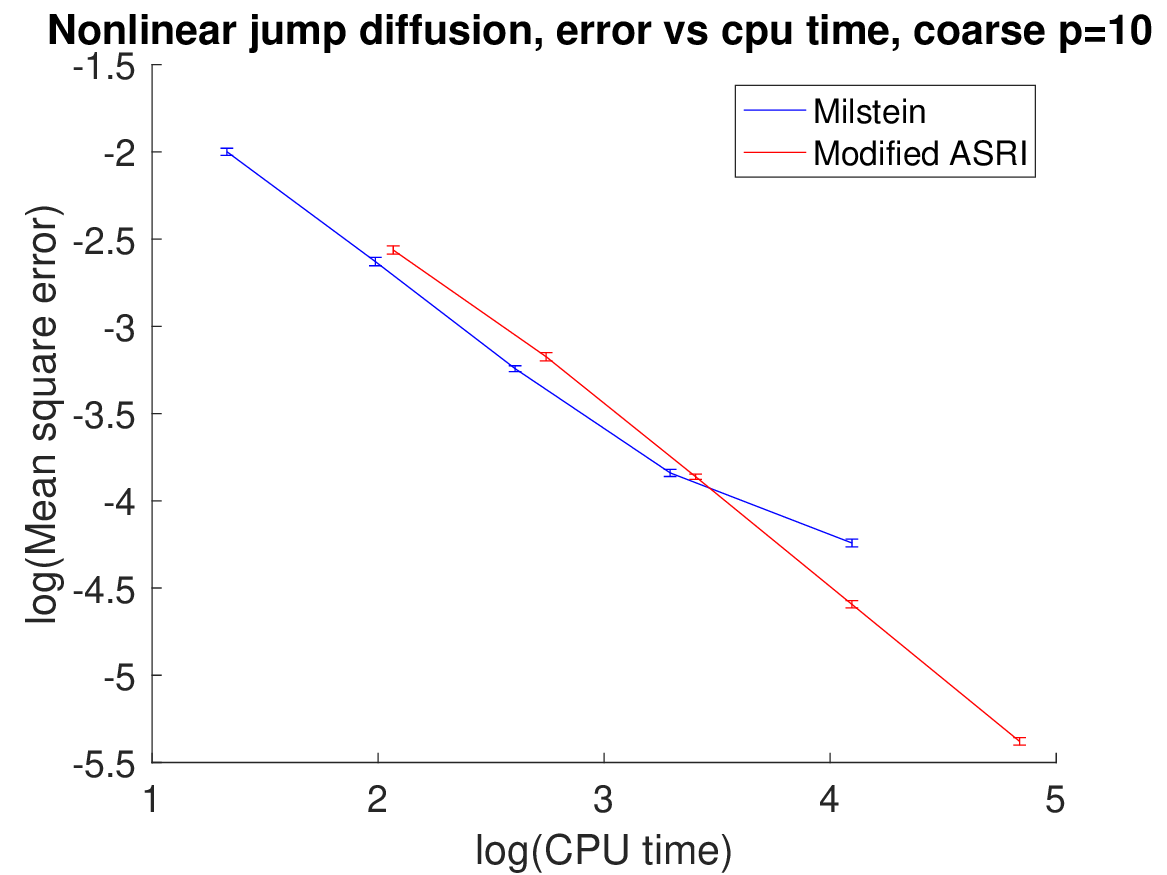}
\includegraphics[width=0.48\textwidth]{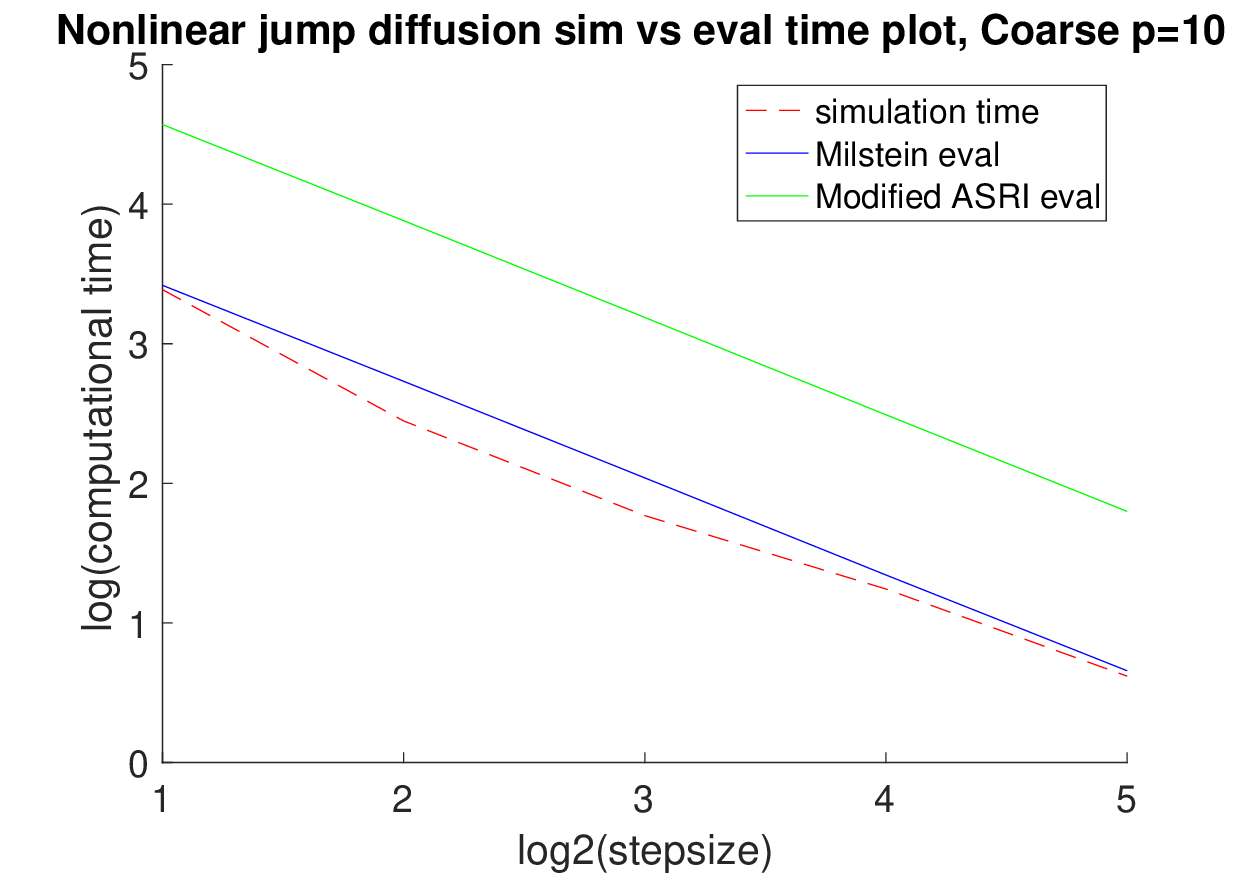}
\end{figure}

\begin{figure}[ht]\label{fig:6}
\includegraphics[width=0.48\textwidth]{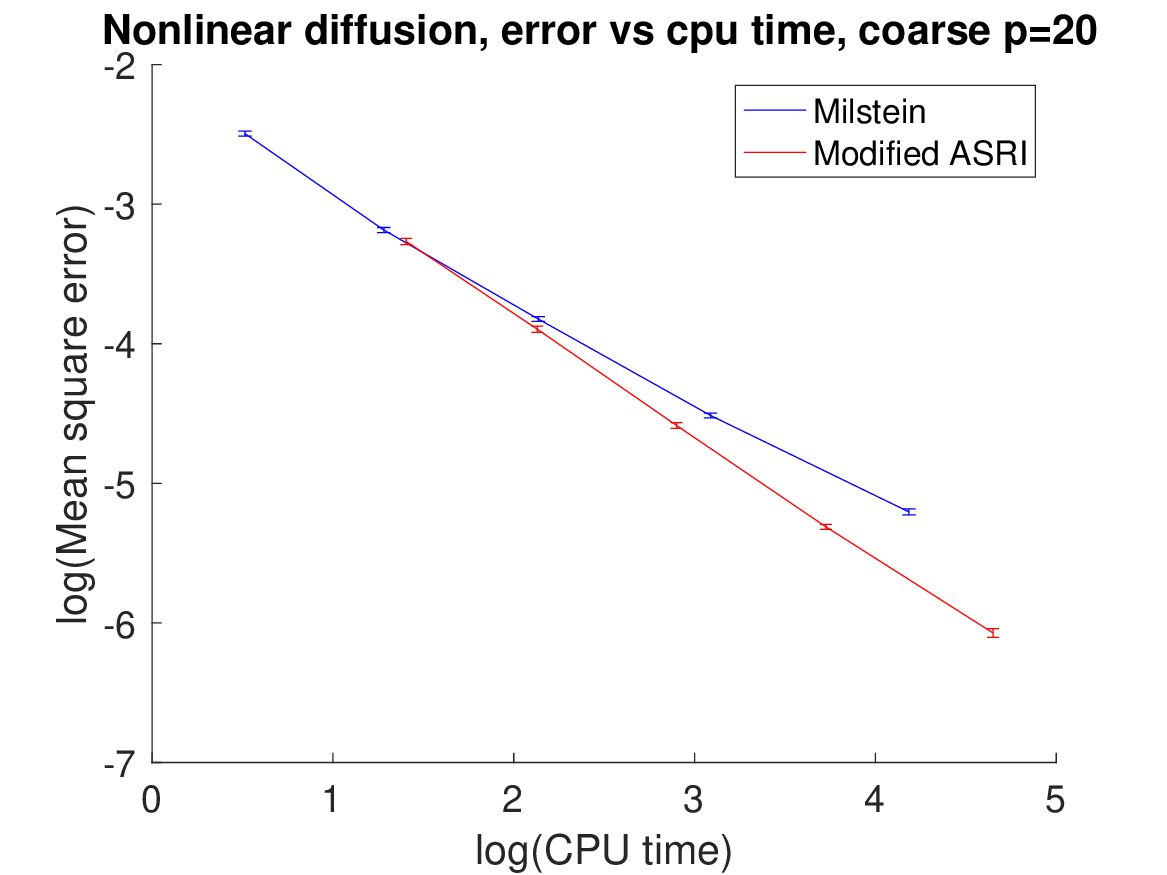}
\includegraphics[width=0.48\textwidth]{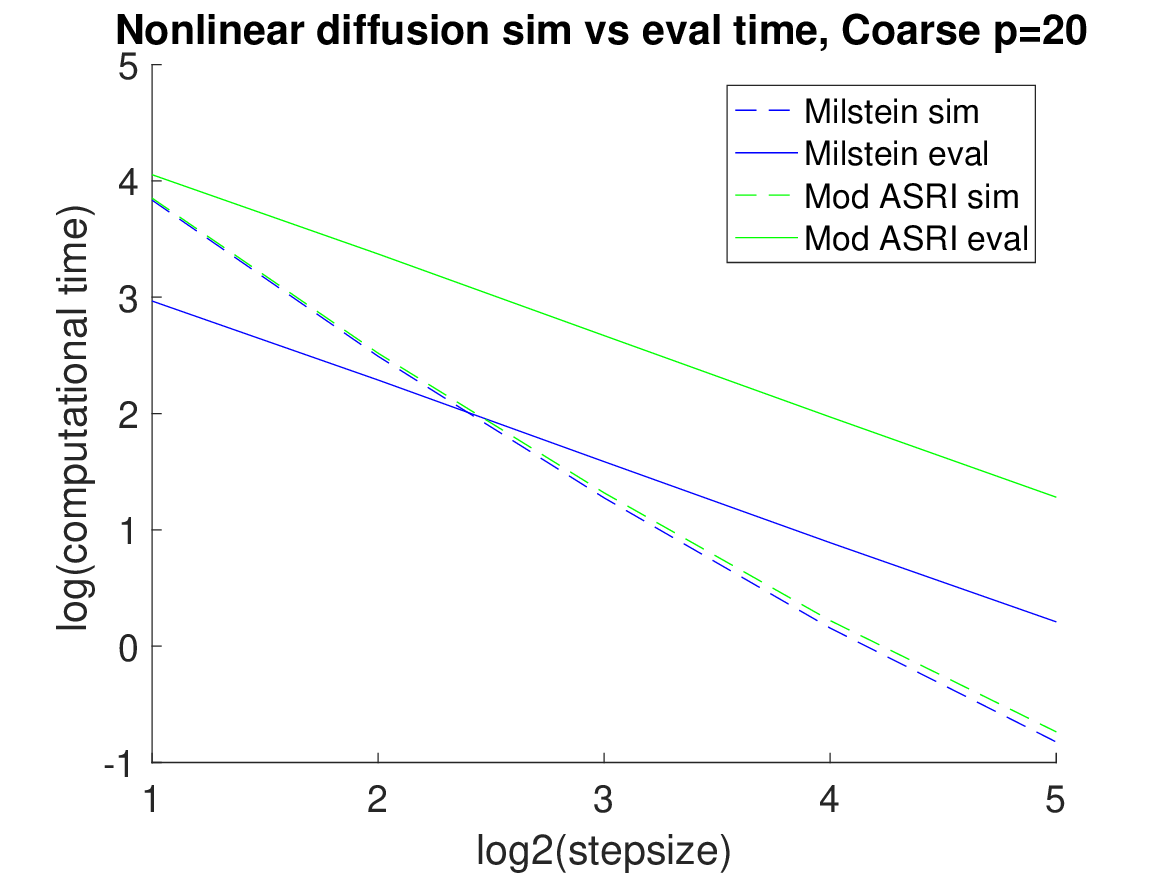}
\end{figure}

\begin{figure}[ht]\label{fig:7}
\includegraphics[width=0.48\textwidth]{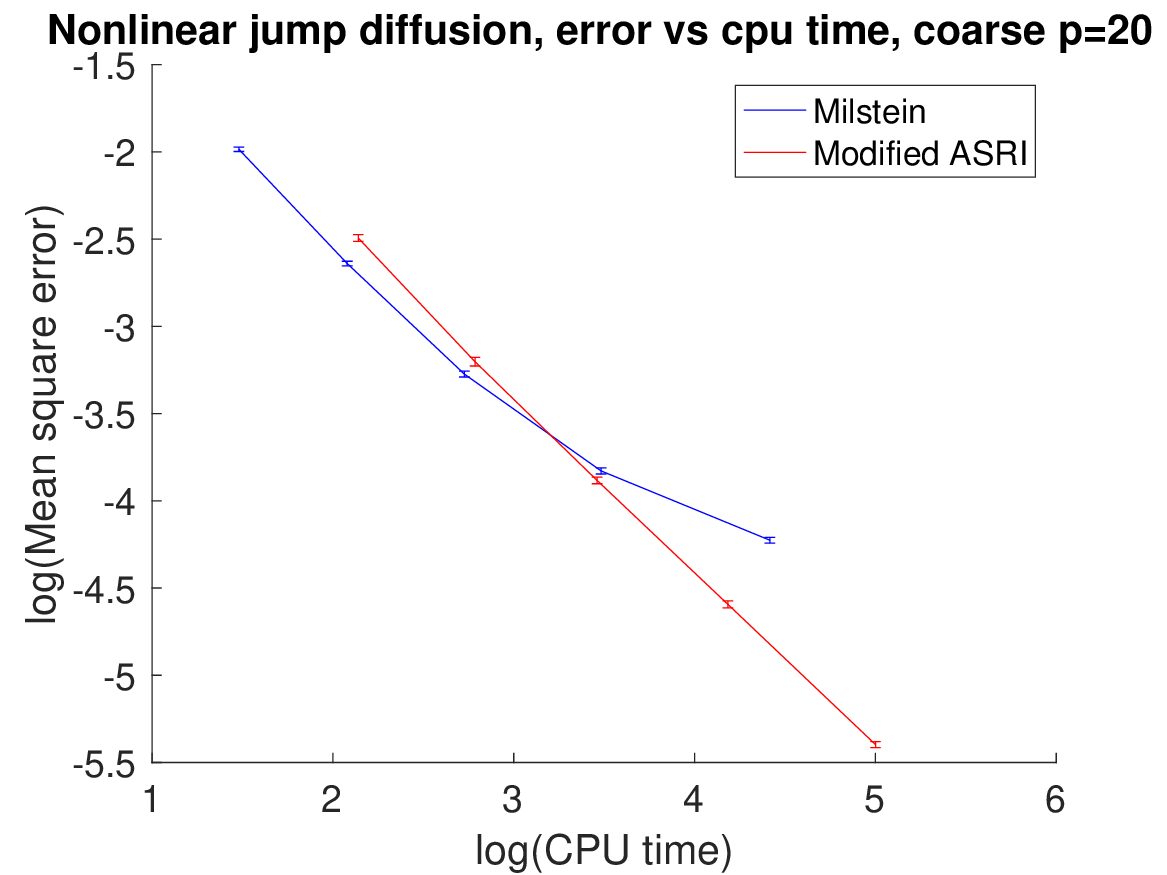}
\includegraphics[width=0.48\textwidth]{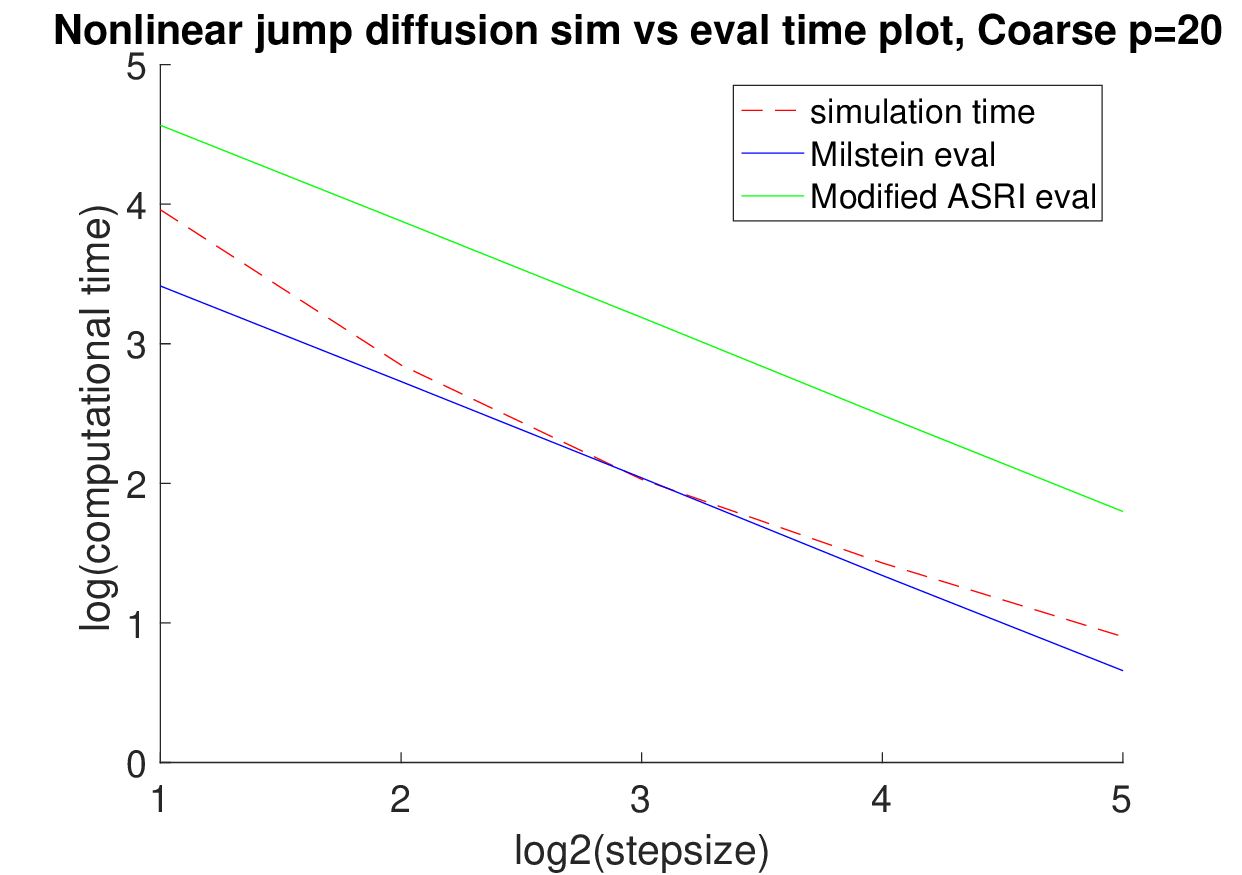}
\end{figure}

\end{document}